\documentclass[11pt,reqno]{amsart}

\usepackage{tikz}
\usetikzlibrary{calc}
\usepackage{wasysym}
\usepackage{accents,esint,amssymb}
\usepackage{tableaucolors}

\newtheorem{theorem}{Theorem}[section]
\newtheorem{lemma}[theorem]{Lemma}
\newtheorem{corollary}[theorem]{Corollary}
\def\t{\theta}

\def\T{\Theta}
\def\sg{\sigma}
\def\a{\alpha}
\def\d{\delta}
\def\w{\omega}
\def\g{\gamma}
\def\G{\Gamma}
\def\c{{\rm c}}
\def\s{{\rm s}}
\def\tf{{\rm t}}
\def\del{\partial}
\def\dot{\accentset{\mbox{\Large\bfseries .}}}
\def\st{\accentset{\star}}
\def\ol{\overline}
\def\ul{\underline}
\def\vp{\varphi}
\def\B#1{\mathbb#1}
\def\mc#1{\mathcal#1}
\def\com#1{\quad\text{#1}\quad}
\def\TS{\textstyle}
\def\wh{\widehat}
\def\wt{\widetilde}
\def\whc#1{\wh{\mc{#1}}}
\def\wht#1{\wt{\mc{#1}}}
\def\ut#1{\undertilde{#1}}
\def\ulc#1{\ul{\mc{#1}}}
\def\olc#1{\ol{\mc{#1}}}
\def\IR#1{\scalebox{1.1}{$\frac{\mc I#1\mc R}2$}}

\def\piot{{\TS\frac{2\pi}T}}
\def\piotk{{\TS\frac{2\pi}{T_k}}}
\def\x#1{$#1\times #1$}
\def\({\left(\begin{array}{cccccc}}
\def\){\end{array}\right)}
\def\dot{\accentset{\mbox{\Large\bfseries .}}}
\def\ddot{\accentset{\mbox{\Large\bfseries .\kern-1.75pt.}}}
       
\numberwithin{equation}{section}

\begin{document}

\title{The Nonlinear Theory of Sound}

\author{Blake Temple}
\address{Department of Mathematics,
  University of California, Davis, CA 95616}
\email{temple@math.ucdavis.edu}

\author{Robin Young}
\address{Department of Mathematics and Statistics,
  University of Massachusetts, Amherst, MA 01003}
\email{rcy@umass.edu}

\date{\today}

\begin{abstract}
  We prove the existence of ``pure tone'' nonlinear sound waves of all
  frequencies.  These are smooth, space and time periodic, oscillatory
  solutions of the $3\times3$ compressible Euler equations in one
  space dimension.  Being perturbations of solutions of a linear wave
  equation, they provide a rigorous justification for the centuries
  old theory of Acoustics.  In particular, Riemann's celebrated 1860
  proof that compressions always form shocks holds for isentropic and
  barotropic flows, but for generic entropy profiles, shock-free
  periodic solutions containing nontrivial compressions and
  rarefactions exist for every wavenumber $k$.
\end{abstract}

\maketitle

\section{Introduction} 

We prove the existence of $1$-dimensional space and time periodic
solutions of the $3\times3$ compressible Euler equations, thereby
providing the first existence proof for globally bounded solutions of
Euler's equations exhibiting sustained nonlinear interactions and
large total variation.  By this, parallel to the classical linear
theory of sound, there really is a nonlinear theory of sound, even
though it was thought not to exist since Stokes and Riemann considered
the problem in the mid-19th century.  Specifically, we prove that
under an arbitrarily small perturbation of any given entropy profile,
the equations obtained by linearization about stationary (weak)
solutions of constant pressure and zero velocity, admit $k$-mode
solutions for every wave number $k$, and each such linear sinusoidal
$k$-mode perturbs to a one parameter family of \emph{pure tone} space
and time periodic solutions of the nonlinear equations admitting the
same frequency in space, but with time periods $T_k$ depending on $k$,
determined by the linearized equation.

By a \emph{pure tone} nonlinear solution we mean a solution of the
$3\times3$ compressible Euler equations which agrees with a linear
sinusoidal $k$-mode solution of the linearized equations, to leading
order in the perturbation parameter.  This nonlinear theory of sound
requires both genuine nonlinearity \emph{and} varying entropy profiles.
Thus the $2\times 2$ theory of shock wave formation for barotropic
equations of state ($p=p(\rho)$, including isentropic and isothermal
flows), established by Riemann in 1860, and made definitive in the
Glimm-Lax decay result of 1970, is \emph{not} indicative of what
happens in the full $3\times3$ system of compressible Euler when the
entropy is not constant.

For $3\times3$ systems, there are two competing physical effects at
play: on the one hand, waves steepen due to genuine nonlinearity,
which is the dependence of sound speed on the state.  On the other
hand, entropy variations cause nonlinear interaction effects which are
manifest as \emph{echoes}, resulting in a scattering of waves which
mitigates the steepening, and which can ultimately prevent shock
formation even when compressions are present.  Our results show that
periodic entropy variations can act to bring compression and
rarefaction of waves into perfect balance and thus prevent shock
formation.

\subsection{Scientific Context}

Our results resolve a long-standing open problem in the theory of
Acoustics which dates to the mid-nineteenth century.  Namely, how is
music possible when nonlinearities always drive oscillatory solutions
into shock waves.  Recall that in the 1750's, Euler developed the
correct extension of Newton's laws of motion to the continuum, and
then linearized the equations to produce the wave equation which
D'Alembert had earlier derived to describe infinitesimal displacements
of a vibrating string.  By this Euler solved arguably the greatest
intellectual problem of his time -- he gave a \emph{mechanical
explanation for music}: vibrations of an instrument produce sinusoidal
oscillations in air pressure, frequencies of which correspond to the
pure tones of sound we hear when, say, a violin is played.  But in the
mid-19th century, mathematicians including Stokes and Riemann
discovered a problem with this theory: solutions containing
compressions could not be sustained, and this would destroy the
musical tones of the linear theory.  After Challis identified the
issue in 1848, Stokes in his paper \emph{``On a difficulty in the
theory of sound''}~\cite{Stokes}, showed that oscillations break down
in finite time, and proposed a resolution using shock waves.  In 1860
Riemann proved that a compression \emph{always} produces shock waves
in isentropic flows.  A century later, this was made definitive in the
celebrated Glimm-Lax decay result of 1970, which established that
space periodic solutions of isentropic Euler, or any genuinely
nonlinear barotropic system, necessarily form shocks and decay to
average at rate $1/t$.  At that time, it was believed that the same
result was true for $3\times3$ non-isentropic Euler as well.  Thus
Euler's original question, why does music resonate so beautifully,
remained unexplained at the level of the fully nonlinear equations --
until now.

In this context, our results establish that the theory of Acoustics
and music based on linear modes of propagation is \emph{not}
inconsistent with nonlinear evolution.  Persistence of sound waves is
inherent in compressible Euler, but only if the entropy is
\emph{non-constant}, so that echoes are present.  A region of shock
free periodic sound wave propagation opens up around every
\emph{non-resonant}, non-constant entropy profile.  The echoes, which
are nonlinear waves scattered by the entropy profile, are on the order
of the incident nonlinear waves for large entropy jumps, and nonlinear
periodic solutions overcome Glimm-Lax shock formation via
characteristics moving ergodicly though the periods, balancing
compression and rarefaction, \emph{on average}.  This is a new point
of view for shock free wave propagation in compressible Euler: instead
of Riemann invariants propagating as constant coordinates along
characteristics (as in \x2 isentropic and barotropic systems
\cite{S}), in this new \x3 regime, every characteristic 
cycles through a dense set of values of each Riemann invariant.

Our results raise the interesting question as to whether this
shock-free regime is the actual regime of ordinary sounds of speech
and musical tones heard in nature and everyday life.  Glimm-Lax theory
is based on approximating smooth solutions by weak shock waves, but as
far as we can tell, only strong shocks are actually observed in
nature.  Our results and the success of the field of Acoustics
indicate that this regime of nonlinear shock-free wave propagation is
more fundamental to ordinary sounds and musical tones than formation
and propagation of ``weak shock waves’’.  Equi-temperament tuning of
the piano makes frequencies irrationally related, which is precisely
our \emph{non-resonance} condition, sufficient to imply perturbation
of linear pure tones to nonlinear pure tones.

The essential physical ideas in this paper were first understood by
the authors within the context of the theory of nonlinear wave
interactions introduced by Glimm and Lax in~\cite{G,GL}.  The
interaction of a nonlinear (acoustic) wave of strength $\gamma$ with
an entropy jump $[s]$ produces an ``echo'', which is a reflected
acoustic wave, whose strength is $O([s]\,\gamma)$, on the order of the
incident acoustic wave~\cite{TY}.  On the other hand, the interaction
of any two (weak) acoustic waves is linear, with an error which is
\emph{cubic} in wave strength.  We began this project with the insight
that \emph{the echoes produced by finite entropy jumps are at the
critical order sufficient to balance rarefaction and compression}.  By
this we might also expect that a theory of nonlinear superposition of
the ``pure tone'' nonlinear sound waves constructed here could also
produce perturbations which provide general shock-free solutions of
the nonlinear equations, although these would no longer be periodic in
time.  Mathematically, this raises the question as to whether, by the
same mechanism, quasi-periodic mixed modes of the linearized theory
also perturb to nonlinear.

Regarding mathematical methods employed, our results demonstrate that
the problem of expunging resonances inherent in the Nash-Moser method,
can be overcome when enough symmetries are present to impose
periodicity by \emph{projection}, rather than by \emph{periodic
return}.  Thus taking into account all of the physical symmetries in
the problem has led to a dramatic simplification of the mathematical
techniques and tools needed.

\subsection{Statement of Results}

The compressible Euler equations are the generalization of Newton's
laws of motion to a continuous medium, in the absence of viscous or
thermal dissipation.  In a spatial (Eulerian) frame, they consist of
equations representing conservation of mass, momentum and energy, and
in one space dimension take the form
\begin{equation}
  \label{euler-eul}
  \begin{aligned}
    \rho_t &+ \big(\rho\,u\big)_X = 0,\\
    (\rho\,u)_t &+ \big(\rho\,u^2 + p\big)_X = 0,\\
    \big(\TS{\frac12}\,\rho\,u^2 + \rho\,e\big)_t &+
    \big(\TS{\frac12}\,\rho\,u^3 + \rho\,e\,u + u\,p\big)_X = 0.
  \end{aligned}
\end{equation}
Here $X$ is the spatial variable and $u$ is the fluid velocity, while
$\rho$, $p$ and $e$ are the fluid density, pressure and specific
energy, respectively.  These constitutive variables are related
through the Second Law of Thermodynamics,
\begin{equation}
  \label{td2}
  de =  \Theta\,ds - p\,dv,
\end{equation}
in which $v=1/\rho$ is the specific volume, $\Theta$ is the temperature,
and $s$ is the specific entropy.  For reversible solutions,
which do not contain shocks, \eqref{td2} is equivalent to the
\emph{entropy equation},
\begin{equation}
  \label{entr-eul}
  (\rho\,s)_t + \big(\rho\,u\,s\big)_X = 0,
\end{equation}
which states that the entropy is preserved along particle paths~\cite{CF}.

To rewrite the equations in \emph{Lagrangian} form, introduce the
\emph{material coordinate} $x$ by
\begin{equation}
  \label{mat}
  x = \int_0^X \rho(\chi)\;d\chi,
\end{equation}
which after manipulation yields the equivalent system
\begin{equation}
  \label{euler-lag}
  \begin{aligned}
    v_t &- u_x = 0,\\
    u_t &+ p_x = 0,\\
    \big(\TS{\frac12}\,u^2 &+ e \big)_t + \big(u\,p\big)_x = 0,
  \end{aligned}
\end{equation}
see \cite{CF,S}.  In this frame, for reversible solutions, in which
$p$ and $u$ are globally continuous, the entropy equation takes the
simple form
\begin{equation}
  \label{entr-L}
  s_t = 0,
\end{equation}
which is solved by $s=s(x)$.  Because the solutions we construct
here are time reversible, we may assume that $s=s(x)$ has been
prescribed and can then drop the third energy equation, so that the
system is fully described as
\begin{equation}
  \label{system}
  \begin{aligned}
    v\big(p,s(x)\big)_t - u_x &= 0,\\
    u_t + p_x &= 0,
  \end{aligned}
\end{equation}
in which the specific volume $v = v(p,s)$ is our explicitly given
constitutive law~\cite{CF,S}.  We can eliminate $u$ in \eqref{system}
to obtain the nonlinear wave equation
\begin{equation}
  \label{wave}
  v\big(p,s(x)\big)_{tt} + p_{xx} = 0.
\end{equation}
We make the standard physical assumption that $v_p(p,s) < 0$, which
implies that \eqref{wave} is hyperbolic and thus a wave equation.  We
will use the Lagrangian frame throughout the paper, but because
systems \eqref{euler-eul} and \eqref{euler-lag} are equivalent, our
results also apply in the Eulerian frame.

We note that the \emph{stationary solution} given by
\[
  s = s(x), \qquad p(x,t) = \ol p,  \qquad u(x,t) = 0,
\]
is a time reversible exact solution of the system \eqref{euler-lag} or
\eqref{system}, even when $s(x)$ and $\rho(x)$ are discontinuous, and
we refer to this as a \emph{quiet state}.  Our periodic sound wave
solutions are perturbations of this quiet state in $p$ and $u$, whose
leading order is a continuous $k$-mode solution of the linearization
of \eqref{system} or \eqref{wave} around this quiet state.

Because we are regarding the entropy profile $s=s(x)$ as given, and we
wish to find time periodic solutions, we treat the material variable
$x$ as the evolution variable.  This allows us to describe the initial
data and corresponding solutions at any fixed $x$ in terms of Fourier
series in time, and in particular, allows an efficient description of
the linearized operator.  Fundamental to our analysis is the
observation that the symmetry condition $p$ even and $u$ odd as
functions of time $t$ is preserved under both nonlinear and linearized
evolution in $x$.  Our final breakthrough was the realization that a
corresponding symmetry in $x$ then allows for a \emph{reflection
  principle} for generating a periodic tiling of the plane from
solutions of a reduced boundary value problem.

The reduced problem is to solve the compressible Euler equations, or
equivalently the $2\times2$ system \eqref{system}, evolving in $x$,
from an initial condition
\begin{equation}
  \label{ic}
  u(0,\cdot) = 0, \qquad p(0,\cdot) \text{ even},
\end{equation}
to $x=\ell$, where we impose the boundary condition
\begin{equation}
  \label{bc}
  \IR-\,\mc S^{T/4}\,u(\ell,\cdot) = 0,  \qquad
  \IR-\,\mc S^{T/4}\,p(\ell,\cdot) = 0.
\end{equation}
Here $T$ is the time period and we have defined the \emph{reflection
  operator} $\mc R$ and \emph{shift operator} $\mc S^{T/4}$ by
\begin{equation}
  \label{RSdef}
  \mc R\,f(t) := f(-t), \com{and}
  \mc S^{T/4}\,f(t) := f\big(t-T/4\big),
\end{equation}
respectively, so that $\IR-$ is the projection onto the odd part of a
function.

\begin{theorem}
  \label{thm:symm}
  A solution of the nonlinear boundary value problem \eqref{system},
  \eqref{ic}, \eqref{bc} determines a space and time periodic solution
  to the compressible Euler equations via a reflection symmetry
  principle.
\end{theorem}

We solve the boundary value problem as a perturbation from quiet state
solutions.  In order to do so, we must first develop a detailed
understanding of the linearized problem.  We first accomplished this
for the simplest non-trivial entropy profile consisting of a single
jump between two constant entropy states.  Our previous
work~\cite{TYperStr, TYperLin, TYperG} led to the understanding that
the key nonlinear effect is balancing of compression and rarefactions
to avoid shock formation.  Identifying the simplest periodic pattern
that balances rarefaction and compression gave the intuition that led
to the discovery of the boundary conditions \eqref{ic} and \eqref{bc}
and the corresponding reflection principle for generating periodic
solutions.

In this paper, we develop these ideas first in this simplest case, and
then successively generalize to piecewise constant profiles and
finally to general entropy profiles.  In doing the general case, we
realized that the boundary conditions are self-adjoint, which allows
us to analyze the linearization as a Sturm-Liouville system, and the
previous cases can be incorporated into this general Sturm-Liouville
framework.  Here we state our results for the general case.

Given a quiet state with entropy profile $s(x)$ based at constant
pressure $\ol p$, we define the \emph{inverse linearized wavespeed}
$\sg=\sg(x)$ to be
\begin{equation}
  \label{wvspd}
  \sigma(x) := \sqrt{-v_p\big(\ol p,s(x)\big)},
\end{equation}
recalling that we are evolving in the material variable $x$.  Define
the set of allowable entropy profiles $s=s(x)$ to be
\[
\mc B := \Big\{s\in L^1[0,\ell]\;\Big|\;
\sg \in L^1,\ \log\sg\in BV\Big\},
\]
together with the $L^1$ topology.  Note that $s(x)$ is general, and we
do not require it to be periodic over $[0,\ell]$.

For any $s\in\mc B$, linearizing the compressible Euler equations
\eqref{system} around the quiet state $\ol p$ yields the linear wave
equation
\begin{equation}
  \label{linWv}
  P_x + U_t = 0, \qquad U_x + \sg^2(x)\,P_t = 0,
\end{equation}
or equivalently
\[
  P_{xx} - \sg^2(x)\,P_{tt} = 0,
\]
in which $x$ is the evolution variable.  We separate variables with
the ansatz
\begin{equation}
  \label{PUans}
  P(x,t) := \vp(x)\,\c(\w\,t),\qquad
  U(x,t) := \psi(x)\,\s(\w\,t),
\end{equation}
where $\c$ and $\s$ denote cosine and sine, respectively.  Together
with the self-adjoint boundary conditions \eqref{ic} and \eqref{bc}
this yields a Sturm-Liouville eigenvalue problem.  In a solution of
\eqref{PUans}, the \emph{eigenfrequency} $\w$ is the square root of
the corresponding Sturm-Liouville eigenvalue.

\begin{theorem}
  \label{thm:slev}
  For $s\in\mc B$, there is a monotone increasing set $\w_k$ of
  \emph{eigenfrequencies} with $\w_k/k$ bounded, and corresponding
  eigenfunctions $\vp_k$ and $\psi_k$, such that the functions
  \[
    P_k(x,t) := \vp_k(x)\,\c(\w_k\,t),\qquad
    U_k(x,t) := \psi_k(x)\,\s(\w_k\,t),
  \]
  solve the linear wave equation \eqref{linWv} together with boundary
  conditions \eqref{ic} and \eqref{bc}.  We call these \emph{pure
    tone} solutions of the linearized equation.
\end{theorem}

Our main result is that under a generic nonresonance assumption,
\emph{each of} these linearized pure tone solutions perturbs to a
one-parameter family of pure tone solutions of the \emph{nonlinear}
compressible Euler equations, with the same space and time periods,
parameterized by amplitude.

We say  a linearized $k$-mode is \emph{nonresonant} if its frequency
$\w_k$ is not a rational multiple of any other eigefrequency,
\[
\frac{\w_j}{\w_k} \notin \B Q, \com{for all} j\ne k.
\]

\begin{theorem}
  \label{thm:onemode}
  For each constant pressure $\ol p>0$ and \emph{nonresonant} linearized
  $k$-mode, there exists $\ol\a_k>0$ such that the $k$-mode perturbs
  to a periodic solution of the nonlinear compressible Euler equations
  with the same space and time periods, taking the form
  \[
  \begin{aligned}
    p(x,t) &= \ol p + \a\,\vp_k(x)\,\c(\w_k\,t) + O(\a^2),\\
    u(x,t) &= \a\,\psi_k(x)\,\s(\w_k\,t) + O(\a^2),
  \end{aligned}
  \]
  for each $|\a|<\ol\a_k$.  Here $p$ and $u$ are the pressure and
  velocity in the Lagrangian frame, $\a$ is the amplitude, used as a
  perturbation parameter, and the \emph{modulations} $\vp_k$ and
  $\psi_k$ are the linearized eigenfunctions of the Sturm-Liouville
  problem.
\end{theorem}

Note that the amplitude bound $\ol\a_k>0$ for which the perturbation
is proven to hold, depends on $k$ through the entire eigenfrequency
structure of the linearized problem.

Our next theorem shows that for \emph{generic} non-constant entropy
proflies, \emph{all} linearized $k$-modes are nonresonant, and so
every $k$-mode perturbs to periodic sound waves solutions of the
nonlinear compressible Euler equations.

\begin{theorem}
  The set of \emph{completely nonresonant entropy profiles}, which
  consists of those entropy profiles for which \emph{every} $k$-mode is
  nonresonant, and so perturbs, is generic in the sense that it is
  residual, or second Baire category, in $\mc B$.
\end{theorem}

Recall that a set is residual if its complement is the countable union
of nowhere dense sets.  Moreover, when we restrict to the set of
piecewise constant entropy profiles, the completely nonresonant set
also has full measure, so that almost every piecewise constant entropy
profile with $n$ jumps is such that \emph{all} linearized $k$-modes
perturb to nonlinear periodic sound wave solutions of the compressible
Euler equations.

Because the leading order terms of our nonlinear pure tone solutions
solve the wave equation, an immediate corollary is the first
\emph{rigorous} mathematical justification for the field of Acoustics,
which uses the wave equation to study sound propagation, which has
been developed since the time of Euler.

\begin{corollary}
  \label{cor:acoustics}
  The use of the linear wave equation
  \[
    \frac1{c^2(x)}\,p_{tt} - p_{xx} = 0, \qquad
    c(x) := \Big(-\frac{\partial v}{\partial p}\Big)^{-1/2},
  \]
  as an approximation for the propagation of one-dimensional sound
  waves is \emph{mathematically justified}, for nonconstant entropy.
\end{corollary}

We note that no such statement can be made for constant entropy, as
has been known since Riemann.  Indeed, if the entropy is constant, the
isentropic case, or if the fluid is barotropic, $p=p(\rho)$, then all
modes are \emph{fully resonant}, in that they are all rational
multiples of each other.  Thus Theorem~\ref{thm:onemode} does not
apply for any mode, and our results are consistent with the celebrated
results of Riemann, Lax and Glimm-Lax, which establish that spatially
periodic solutions to any $2\times2$ genuinely nonlinear system always
form shock waves and subsequently decay to average at rate
$1/t$~\cite{Riemann,Lax64,GL}.  In particular, our results imply that
the complete \x3 system of compressible Euler is fundamentally
different from the isentropic \x2 system.

\begin{corollary}
  \label{cor:noDecay}
  Generically, space periodic solutions of the \x3 compressible Euler
  equations will \emph{not} decay to constant (or quiet state), and in
  particular, solutions containing compressions need not form shock
  waves.
\end{corollary}

The breakthrough in solving the problem was the realization that we
can impose periodicity within a smaller more symmetric class of
solutions by \emph{projection}, rather than by \emph{periodic return}.
The standard way to impose periodicity is by periodic return, by which
we pose data $U_0$ at $x=0$, evolve nonlinearly in $x$ through one
period by $\mc N$, and set $\mc N\,U_0 = U_0$.  This yields a
nonlinear equation of the form
\[
\mc F_1(U_0) := (\mc N - \mc I)\,U_0 = 0,
\]
and we wish to solve for $U_0$.  Our initial attempt to do so was by a
Newton-Nash-Moser iteration argument, as in~\cite{CW, Moser, Deimling,
  Rabinowitz, TYperNM}.  The Newton method requires inversion of nearby
linearized operators $D\mc F_1(U^{(k)})[\cdot]$ at each approximation
$U^{(k)}$ of the iteration.  In our previous work, we identified the
main technical difficulty of this approach, namely, the decay rate of
the small divisors does not depend continuously on the constant state
$\ol p$ and the resonances, at which divisors vanish, cannot be
controlled or predicted.  In~\cite{TY}, we formulated a strategy for
expunging parameters in order to effectively control the small
divisors, and carried this out in a scalar warm-up problem.  Although
this approach is plausible, implementing it is a daunting project
because of the very delicate estimates required to expunge small
divisors.

In this paper we identify symmetry properties of the nonlinear
solution which allow us to impose periodicity by \emph{projection}.
The idea is that within a smaller more symmetric class of solutions,
periodicity can be imposed by a projection consistent with the
nonlinear symmetries.  With this reformulation the problem is
effectively changed from periodic return $\mc N\,U_0=U_0$ to
periodicity by projection,
\[
\mc F_2(U_0) := \Pi\,\mc N\,U_0 = 0,
\]
where again $\mc N$ is nonlinear evolution, and $\Pi$ is a half-space
projection.

A remarkable further simplification is that this nonlinear operator
now also factors as
\[
\mc F_2(U_0) = \Pi\,\mc L\,\big(\mc L^{-1}\,\mc N\big)\,U_0 = 0,
\]
where $\mc L$ is linearized evolution by $D\mc N(\ol p)[\cdot]$, and
$\mc L^{-1}\,\mc N$ is a {\it bounded invertible} nonlinear operator
satisfying
\[
D\big(\mc L^{-1}\,\mc N\big)(\ol p) = \mc I,
\]
the identity.  Here $D\mc F_2(\ol p)=\Pi\,\mc L$ is the \emph{fixed}
operator obtained by linearizing $\mc F_2$ around the quiet state
$(\ol p,0)$, and this \emph{uniformly} encodes the small divisors.  In
this way, we are able to avoid Nash-Moser altogether and treat the
problem as a regular bifurcation problem which can be solved using the
implicit function theorem.

This simplification allows us to dramatically extend the analysis to
general entropy profiles and wave numbers, subject only to a
non-resonance condition which is generically satisfied.  Although the
background pressure $\ol p$ is perturbed to $\ol p+\a\,z$ in solving
the problem, the linear factor $\Pi\,\mc L$, which is the
linearization at $\ol p$, is constant, and hence so are the small
divisors, independent of the perturbation parameter $\a$.

\subsection{History of the Problem}

We give a contextual history of the problem, focusing on the
development of the theory of sound and its relation to that of
continuum mechanics and shock waves.  Our main sources are several
works of Truesdell~\cite{Tru2, Tru1}, together with the collections of
Lindsay~\cite{Lindsay, Lindsay2} and Johnson and
Cheret~\cite{JohnsonCheret}; see also~\cite{HamBla, Pierce, Muller}.

The wave-like nature of sound and the fact that it is caused by
vibrations of solid objects such as bells and musical instruments was
known by the ancients.  It was also known that these vibrations caused
corresponding vibrations in the air or other medium, which were then
sensed by the eardrum as sounds of various sorts.

The origin of modern theories of mechanics is regarded as Newton's
Principia, which consists of three books.  In Book I, Newton sets out
his famous Laws of Motion and uses his Calculus to solve many problems
of the dynamics of a single small body or point-mass.  Newton realized
that in order to obtain accurate results, he needed to take into
account resistive forces, which led him to begin a development of
fluid mechanics in Book II, and Book III revolved around his Law of
Gravitation.  In Book II, Newton attempted to describe both the
dynamics of continuous media and the propagation of sound waves, but
he was ultimately unable to get these quite right.

In the century after the appearance of the Principia, many of the
problems of continua and sound raised by Newton were understood, being
especially driven by the Bernoullis and Euler.  This led to
d'Alembert's derivation of the wave equation for a vibrating string in
1748, and to Euler's equation of 1751, expressing conservation of
linear momentum, which is
\begin{equation}
  \label{mom}
  \rho\,\frac{Du}{Dt} = -\nabla p + \rho\,b,
\end{equation}
in which $u$ is the velocity, $\rho$ the density, $p$ the pressure,
and $b$ is the body force per unit mass.  Euler correctly defined the
pressure and identified the internal force as the negative pressure
gradient, but his real triumph was to get the convective derivative
\[
  \frac{D}{Dt} = \frac{\del}{\del t} + u\cdot\nabla
\]
right.  In 1755, Euler coupled this with the continuity equation,
which expresses conservation of mass,
\begin{equation}
  \label{mass}
  \frac{D\rho}{Dt} + \rho\,\nabla\cdot u = 0.
\end{equation}
If the pressure $p(\rho)$ is given, \eqref{mom} and \eqref{mass} close
and together consist of \emph{Euler's equations} for a barotropic
fluid.  In 1759, Euler expressed these equations in a material
coordinate and linearized to get the wave equation of d'Alembert.  In
so doing, Euler provided a \emph{mechanical explanation} for the
propagation of sound waves.  The vibration of a solid object, such as
a bell, drum or string, causes the surrounding fluid (that is, air) to
vibrate, and these vibrations propagate through the fluid medium
according to \eqref{mom} and \eqref{mass}, and are in turn sensed at
the ear.  Given the central role that music played in Europe at that
time, this explanation of musical tones in terms of linear sinusoidal
oscillation and superposition, was one of the greatest intellectual
achievements of that era.

One of the big questions of the late 18-th and early 19-th centuries
was to understand the the constitutive relation $p=p(\rho)$, from
which the speed of sound could be calculated.  Poisson in 1808
developed what is now known as the method of characteristics, and
obtained implicit equations for the solution of \eqref{mom},
\eqref{mass} in one space dimension.  In 1848, Challis pointed out
that in some cases, Poisson's solution breaks down, which led Stokes,
also in 1848~\cite{Stokes}, to propose discontinuous solutions
containing shock waves.  In 1860, Earnshaw introduced simple waves,
and Riemann independently showed that \emph{any} non-trivial solution
to Euler's equations suffers gradient blowup; indeed,\footnote{from
  B. Riemann, ``The Propagation of Planar Air Waves of Finite
  Amplitude'', 1860,
  transl.~in~\cite{JohnsonCheret}}\\
\begin{center}
\begin{minipage}[h]{0.8\linewidth}
  \it The compression wave, that is, the portions of the wave where
  the density decreases in the direction of propagation, will
  accordingly become increasingly more narrow as it progresses, and
  finally goes over into compression shocks; but the width of the
  expansion or release wave grows proportional with time.
\end{minipage}
\end{center}
\medskip

A large part of the later 19-th century was spent in understanding
thermodynamics and the roles of energy and entropy, by which the
correct shock speeds could be found.  This culminated in the
development of the Rankine-Hugoniot conditions in the 1880's, which
allowed for the successful treatment of discontinuous solutions and
shock waves.  The upshot is that a third equation, namely conservation
of energy,
\begin{equation}
  \label{energy}
  \frac{\del}{\del t}\Big(\frac12\,\rho\,u^2 + \rho\,e\Big) +
  \nabla\cdot\Big(\frac12\,\rho\,u^2\,u + \rho\,e\,u + p\,u\Big) = 0,
\end{equation}
together with the Second Law of Thermodynamics,
\[
  \Theta\,ds = de + p\,dv, \qquad v= \frac1\rho,
\]
is needed to fully describe the system.  Here $\Theta$, $e$ and $s$ are the
temperature, specific internal energy, and specific entropy,
respectively, and for smooth solutions, \eqref{energy} is equivalent
to the \emph{entropy equation},
\begin{equation}
  \label{entropy}
  \frac{\del}{\del t}(\rho\,s) +
  \nabla\cdot(\rho\,s\,u) = 0.
\end{equation}
If discontinuities are present, \eqref{entropy} is replaced by an
inequality, which becomes a selection criterion for shock waves.

In the first decades of the twentieth century, researches investigated
more general constitutive laws and effects of viscosity, while the
1930's focussed attentin on shock waves.  In their renowned
monograph~\cite{CF}, Courant and Friedrichs collected most of the
results of shock wave theory which was developed up to and in the
Second World War, and which included early numerical simulations.

The modern theory of conservation laws was initiated by Peter Lax
in~\cite{Lax}, in which he considered the abstract system
\begin{equation}
  U_t + F(U)_x = 0, \qquad U\in\B R^n,\quad F:\B R^n\to\B R^n,
  \label{cl}
\end{equation}
for which he defined weak solutions and solved the Riemann problem.
In his celebrated 1965 paper~\cite{G}, Glimm proved the global
existence of weak solutions to~\eqref{cl}, provided the total
variation of the initial data $U_0$ is small enough.  Shortly
afterwards, Glimm and Lax showed that weak solutions of $2\times2$
systems decay as $1/t$ in~\cite{GL}, being driven by the decay of
shock waves.  Thereafter the prevailing view in the community was that
solutions to generic hyperbolic systems should always form shocks and
decay to a constant state.

For systems of more than two equations, F.~John, and later Tai-Ping
Liu, proved that shocks form given sufficiently compressive initial
data~\cite{J, TPLblowup}, and this was later improved by Geng Chen and
collaborators~\cite{Chen1, CYZ, CPZ1}.  The first indications that
there may be non-decaying solutions of the $3\times3$ Euler equations
arose from the method of weakly nonlinear geometric optics in the
mid-80s~\cite{MRS,P,HMR}, with further suggestions coming from
numerical studies~\cite{Shef,Vayn}, and exact solutions to certain
simplified model systems~\cite{Yper,Ysus}.  In~\cite{TY}, the authors
extended Glimm's existence theory to large time existence for the
$3\times3$ Euler equations, provided the initial data has (arbitrarily
large) finite total variation and small amplitude.  To our knowledge,
this is the largest class of data for which large-time existence has
been proved prior to the current work.

In~\cite{TYperStr}, the authors understood the physical phenomenon
that may prevent shock formation, namely an echoing effect due to
changes in entropy.  That is, when the (linearly degenerate) entropy
varies, nonlinear simple waves necessarily interact with this entropy
field, resulting in a partial transmission and partial reflection of
nonlinear waves.  These reflected waves, or echoes, are then (to
leading order) superimposed on all other nonlinear waves.  If this
superposition of waves fits a certain pattern, then compressions never
steepen completely, and periodic solutions could then ensue.
In~\cite{TYperLin,TYperG,TYperBif}, we described this structure in the
linearized Euler equations, and expressed the problem of finding
periodic solutions as a bifurcation problem, albeit with small
divisors and a loss of derivvative.  In~\cite{TYperEV,TYperNM}, we
then set up a Nash-Moser framework for solving this bifurcation
problem.  In~\cite{TYdiff1,TYdiff2}, we understood the effects of
derivative loss and small divisors in a scalar model problem, and
described a strategy to treat the small divisors uniformly.

\subsection{Structure of the paper}

We briefly outline the layout of the rest of the paper. In
Section~\ref{sec:sys}, we recall the nonlinear equations, and state
the periodicity problem as a projection for square wave entropy
profiles.  This mirrors the nonlinear effects and periodic structure
developed in our earlier papers~\cite{TYperStr, TYperLin, TYperBif}.
In Section~\ref{sec:lin} we linearize around a quiet state and
calculate the small divisors for this simplified profile.  In
Section~\ref{sec:onejump}, we carry out the bifurcation argument for
the lowest mode in the simplest case of a square wave entropy profile;
this establishes the first nontrivial periodic solution of the
compressible Euler equations.  In Section~\ref{sec:Tvar} we allow the
time period to vary, and repeat the bifurcation argument for all
nonresonant $k$-modes, thus providing an infinite number of distinct
periodic pure tone solutions for a nonresonant square wave entropy
field.  In Section~\ref{sec:pw} we extend the analysis to general
piecewise constant entropy profiles and show that the set of fully
nonresonant profiles has full measure and is residual.  This
establishes that generically, \emph{all} linearized $k$-modes perturb
to pure tone solutions of the nonlinear problem.  In
Section~\ref{sec:SL} we further generalize to \emph{arbitrary} $BV$
entropy profiles.  For this, we reduce the corresponding linear
problem to a Sturm-Liouville system and carry out the bifurcation
argument in this most general case.  This includes the statement that
generic profiles are again fully nonresonant.  Finally in
Section~\ref{sec:D2E}, we describe and solve the second derivative
$D^2\mc E$ of the nonlinear evolution operator, which requires
solution of an inhomogeneous linear system.  Although of interest in
its own right because it demonstrates that nonlinear evolution in
$H^s$ is twice differentiable, this provides an essential step in the
solution of the bifurcation problem used in Section~\ref{sec:SL}.

\section{The System}
\label{sec:sys}

Our starting point is the compressible Euler equations in a Lagrangian
frame, which is the $3\times3$ system in conservative form,
\begin{equation}
  \label{lagr}
  v_t - u_x = 0, \quad u_t + p_x = 0, \quad
  \big(\TS{\frac12}\,u^2 + e\big)_t + (u\,p)_x = 0.
\end{equation}
Here $x$ is the material coordinate and $u$ is the Eulerian
velocity, and the thermodynamic variables are specific volume $v$,
pressure $p$ and specific internal energy $e$.  The system is closed
by a \emph{constitutive relation} which satisfies the Second Law of
Thermodynamics,
\[
  de = \Theta\,ds - p\,dv, \com{so that}
  \Theta = e_s(s,v) \com{and} p = -e_v(s,v),
\]
where $\Theta$ is the temperature and $s$ the specific entropy.  It
follows that for classical smooth solutions, the third (energy)
equation can be replaced by the simpler \emph{entropy equation},
\[
  e_t + p\,v_t = 0,  \com{or}  s_t = 0.
\]
We initially consider the fundamental case of a polytropic ideal (or
$\gamma$-law) gas, which is described by
\begin{equation}
  \label{glaw}
  p = v^{-\gamma}\,e^{s/c_v}, \com{or}
  v = p^{-1/\gamma}\,e^{s/c_p},
\end{equation}
where $\gamma = c_p/c_v$ is the ratio of specific heats, assumed
constant; in Section 7 we generalize to arbitrary equation of state.
To start, we assume a piecewise constant entropy field, so that the
system is \emph{isentropic} on finite $x$-intervals, so we use only
the first two equations of \eqref{lagr}, which yields a closed system
in $u$ and $p$.  We treat $x$ as the evolution variable, so we write
the equation as
\begin{equation}
  \label{isen}
  u_x - \big(p^{-1/\gamma}\,e^{s/c_p}\big)_t = 0,
  \qquad p_x + u_t = 0,
\end{equation}
and regard $s$ as constant on each subinterval.

In our earlier development, we introduced a non-dimensionalization
which translates linear evolution into rotation, and thus finds the
essential geometry of the solutions of the linearized wave equation.
This allowed us to identify the simplest periodic wave structure, and
the nonlinear functional which imposes periodicity.  This rotation
structure allowed us to explicitly understand the resonances and small
divisors in the linearized operator.  In our later development using
Sturm-Liouville theory in Section~\ref{sec:SL} below, we again find
the correct angle variable in which the evolution can be interpreted
as rotation.

\subsection{Non-dimensionalization}

We briefly recall the non-dimensionalization of the isentropic system
\eqref{isen}, which was initially derived in \cite{TYperLin}.  The
advantage in non-dimensionalizing is that the nonlinear evolution
becomes independent of the constant value of the entropy.  In this
formulation the jump conditions in the rescaled variables are imposed
at each entropy jump.  The simplest weak solution which incorporates a
given entropy profile is one in which $(u,p)=(u_0,p_0)$ is constant.
By a Galilean transformation, we can take $u_0=0$ without loss of
generality, so our constant state is characterized by $p_0$.

\begin{lemma}
  \label{lem:nondim}
  On an interval on which the entropy $s$ is constant, the rescaling
  \[
    \begin{aligned}
    w &:= \Big(\frac {p-p_0}{p_0}\Big)^{\frac{\gamma-1}{2\gamma}},\\
    \st w &:= \frac{\gamma-1}{2\sqrt\gamma}\,
    e^{-s/2c_p}\,p_0^{-\frac{\gamma-1}{2\gamma}}\,u,\\
    X &:= \frac1{\sqrt\gamma}\,e^{s/2c_p}\,
    p_0^{-\frac{\gamma+1}{2\gamma}}\,x,
    \end{aligned}
  \]
  transforms the isentropic system \eqref{isen} into the non-dimensional
  system
  \begin{equation}
  \st w_X + (1+w)^{-\nu}\,w_t = 0, \qquad
  w_X + (1+w)^{-\nu}\,\st w_t = 0,
  \label{nondim}
\end{equation}
with $\nu=\frac{\gamma+1}{\gamma-1}$.  Conversely, any classical
solution $(w,\st w)$ of \eqref{nondim} yields a solution of the
isentropic system \eqref{isen}, in which the physical variables are
given by
\[
  \begin{aligned}
    p &:= p_0\,\big(w+1\big)^{\frac{2\gamma}{\gamma-1}},\\
    u &:= \frac{2\sqrt\gamma}{\gamma-1}\,
    e^{s/2c_p}\,p_0^{\frac{\gamma-1}{2\gamma}}\,\st w,\\
    x &:= \sqrt\gamma\,e^{-s/2c_p}\,
    p_0^{\frac{\gamma+1}{2\gamma}}\,X.
  \end{aligned}
\]
\end{lemma}

\begin{proof}
  We begin by introducing a thermodynamic parameter $h=h(p)$, so that
  $p=p(h)$ is determined in such a way that the nonlinear wavespeed
  appears the same in both equations.  Setting $p=p(h)$ in
  \eqref{isen} and manipulating yields
\[
  u_x +
  \frac1\gamma\,e^{s/c_p}\,p^{-\frac{\gamma+1}\gamma}\,p'(h)\,h_t = 0,
  \qquad
  h_x + \frac1{p'(h)}\,u_t = 0,
\]
so we choose $h$ such that
\[
  \frac1{p'(h)} = k_1\,p'(h)\,p^{-\frac{\gamma+1}\gamma},
  \com{or} h'(p) = k_2\,p^{-\frac{\gamma+1}{2\gamma}}.
\]
It suffices to take
\[
  h := p^{\frac{\gamma-1}{2\gamma}}, \com{so that}
  p = h^{\frac{2\gamma}{\gamma-1}},
\]
which yields a wavespeed of
\[
  \frac{k_3}{p'(h)} = k_4\,h^{-\nu}, \com{with}
  \nu := \frac{\gamma+1}{\gamma-1},
\]
where $k_i$ are appropriate constants.

To non-dimensionalize, we procced as follows: first, scale $h$ by
$h_0$, then scale $u$ by a constant (depending on $h_0$ and $s$), and
finally, rescale the material coordinate $x$ to get the simplest
possible nonlinear system.  Thus, given constant $p_0$, we set
\[
  \wh w := \Big(\frac p{p_0}\Big)^{\frac{\gamma-1}{2\gamma}},
  \com{so} p = p_0\,\wh w^{\frac{2\gamma}{\gamma-1}}, \quad
  v = e^{s/c_p}\,p_0^{-1/\gamma}\,\wh w^{-\frac2{\gamma-1}}.
\]
Plugging in to \eqref{isen}, for classical solutions we get
\[
  \begin{gathered}
    u_x +
    \frac2{\gamma-1}\,e^{s/c_p}\,p_0^{-1/\gamma}\,\wh w^{-\nu}\,\wh w_t =
    0, \\  \wh w_x +
    \frac{\gamma-1}{2\gamma}\,p_0^{-1}\,\wh w^{-\nu}\,u_t = 0,
  \end{gathered}
\]
which we rewrite as
\[
  \begin{gathered}
    \sqrt{\gamma}\,e^{-s/2c_p}\,p_0^{\frac{\gamma+1}{2\gamma}}\,
    \Big(\frac{\gamma-1}{2\sqrt\gamma}\,
    e^{-s/2c_p}\,p_0^{-\frac{\gamma-1}{2\gamma}}\,u\Big)_x + \wh w^{-\nu}\,\wh w_t = 0, \\
    \sqrt{\gamma}\,e^{-s/2c_p}\,p_0^{\frac{\gamma+1}{2\gamma}}\,
    \wh w_x + \wh w^{-\nu}\,\Big(\frac{\gamma-1}{2\sqrt\gamma}\,
    e^{-s/2c_p}\,p_0^{-\frac{\gamma-1}{2\gamma}}\,u\Big)_t = 0.
  \end{gathered}
\]
From this it is evident how we should rescale $u$ and $x$: namely, we
set
\begin{equation}
  \st w := \frac{\gamma-1}{2\sqrt\gamma}\,
    e^{-s/2c_p}\,p_0^{-\frac{\gamma-1}{2\gamma}}\,u, 
  \label{wst}
\end{equation}
and rescale the material variable by
\[
  x \mapsto X := \frac1{\sqrt\gamma}\,e^{s/2c_p}\,
  p_0^{-\frac{\gamma+1}{2\gamma}}\,x,
\]
and as a final simplification, we set $w:=\wh w-1$, so that our base
constant state is the origin in the $(w,\st w)$ coordinates.  Our
system then becomes \eqref{nondim}.
\end{proof}

For convenience, we revert to using $x$ rather than $X$ for the
rescaled material variable.  After this rescaling, on a region of
constant entropy, in a Lagrangian frame, the Euler equations
\eqref{nondim} can be written as the non-dimensional quasilinear \x2
system
\begin{equation}
  \del_xW + \wh\sg(w)\,H\,\del_tW = 0, \quad W(t,0) = W^0(t),
\label{evol}
\end{equation}
with
\[
  W = \(w\\\st w\), \quad H = \(0&1\\1&0\), \quad \wh\sg(w)=(1+w)^{-\nu},
\]
and we denote evolution of the data $W^0$ from $x=0$ to $x=\t$, by
$\wht E^\t(W^0)$.  Here $w$ and $\st w$ are the rescaled thermodynamic
variable and Eulerian velocity, respectively, and the system is
independent of the (constant) value of the entropy.  Because we want
to find solutions that are pure tones in time, we evolve in the
material variable, and the data $W^0$ is taken to be a periodic function
of time with period $T$, which we will initially take to be $2\pi$.

Since the entropy satisfies the equation $s_t=0$, any entropy jump is
stationary in the Lagrangian frame, and the effect of crossing the
jump is described by the Rankine-Hugoniot conditions, $[u] = [p] = 0$.
In non-dimensional coordinates, this becomes
\[
  w_R =  w_L,\quad \st w_R= J\,\st w_L, \com{or}
  W_R = \mc J\,W_L,
\]
where the linear operator $\mc J$ is defined by
\begin{equation}
  \label{jump}
  \mc J\,W := M(J)\,W, \com{where} M(J) = \(1&0\\0&J\),
\end{equation}
and where $J = e^{-[s]/2c_p}$, with $[s]:=s_R-s_L=s(x+)-s(x-)$
denoting the entropy jump.  We note that although this description is
exact for an ideal polytropic ($\gamma$-law) gas, given by
\eqref{glaw}, for a general constitutive law, the Rankine-Hugoniot
condition takes the form $p(v_R,s_R)=p(v_L,s_L)$.

In \cite{TYperStr,TYperLin}, the authors described the simplest
structure of possible space and time periodic solutions of the
compressible Euler equations which formally balance compression and
rarefaction.  These occur when there are exactly two entropy levels,
with separate jumps between them, and they should satisfy the
nonlinear equation
\begin{equation}
  \label{NU=U}
  (\mc N-\mc I)\,W = 0, \com{where}
  \mc N := \mc S^{T/2}\,\mc J^{-1}\,\wht E^{\ut\t}\,\mc J\,\wht E^{\wt\t},
\end{equation}
where the evolutions and jumps are as above, and $\mc S^{T/2}$ denotes
a half-period shift,
\[
  \mc S^{T/2}\,W(t) := W(t-T/2).
\]
In this paper we find space and time periodic solutions based on a
modification of \eqref{NU=U} that takes advantage of symmetries
preserved by nonlinear evolution in space, namely $p$ even and $u$ odd
as functions of $t$.  The key to managing the small divisors in
\eqref{NU=U} is based on a reformulation tailored to this more
symmetric class of allowable solutions.  We call this new approach
\emph{periodicity by projection}, while we refer to \eqref{NU=U} as
describing \emph{periodicity by periodic return}.

\subsection{Riemann Invariants and Symmetry}

It is convenient to use Riemann invariants to describe the system;
in terms of our rescaled variables, these are
\begin{equation}
  \begin{aligned}
    y &= w + \st w, \quad \st y = w - \st w, \com{so that}\\
    w &= \frac{y+\st y}2, \quad \st w = \frac{y-\st y}2.
  \end{aligned}
  \label{RIdef}
\end{equation}
In Riemann invariants, the evolution equations \eqref{evol} are
\begin{equation}
  \begin{aligned}
  \del_x\st y - \wh\sg\big(\frac{\st y+y}2\big)\,\del_t\st y &= 0, \\
  \del_xy + \wh\sg\big(\frac{\st y+y}2\big)\,\del_ty &= 0,
  \end{aligned}
\label{RI}
\end{equation}
and we write the state as $U = (\st y, y)^T$.  We let $\mc E^\t(U^0)$
denote evolution of the initial data $U^0$ from $x=0$ to $x=\t$.  It
follows that
\[
  \wht E^\t\mc Q = \mc Q\,\mc E^\t, \com{where}
  \mc Q = \( 1 & 1 \\ -1 & 1 \),
\]
and the two evolutions \eqref{evol} in $W$ and \eqref{RI} in $U$ are
equivalent.

Observe that the time symmetry
\begin{equation}
  \label{symm}
  w(x,\cdot) \text{ even}, \qquad \st w(x,\cdot) \text{ odd}  
\end{equation}
is respected by the nonlinear evolution \eqref{evol}.  That is, if it
holds at some $x_0$, it remains true at any other $x$.  Translating
this into Riemann invariants using \eqref{RIdef}, this is
\[
\begin{aligned}
  \st y(x,-t) &= w(x,-t) - \st w(x,-t)\\
  &= w(x,t) + \st w(x,t) = y(x,t).
\end{aligned}
\]
It is convenient to express this in terms of a reflection operator:
denote the action of reflecting time $t\to -t$ by $\mc R$, so that
\begin{equation}
  [\mc R\,v](t) := v(-t).
  \label{refl}
\end{equation}
In this notation, the $w$ even, $\st w$ odd symmetry \eqref{symm} becomes
the condition
\begin{equation}
  \st y(x,\cdot) = \mc R\, y(x,\cdot) \com{for any}x,
\label{rssymm}
\end{equation}
and imposing this at any point $x_0$ implies that it holds throughout
the evolution.  On the other hand, imposing \eqref{rssymm} on $y$ and
$\st y$ and defining $w$ and $\st w$ to be the even and odd parts of
$y$, respectively, implies the symmetry \eqref{symm}.  This means that we
can impose the symmetry \eqref{symm} in terms of Riemann invariants by
allowing $y^0$ to be arbitrary and setting $\st y^0 = \mc R\,y^0$.
Similarly, once $y(x,\cdot)$ is known, we use \eqref{rssymm} to get
$\st y(x,\cdot)$ and \eqref{RIdef} to reconstruct the full solution
$U$.

The upshot is that as long as the solutions of the \x2 system
\eqref{evol} remain regular, and the even/odd symmetry \eqref{symm} is
imposed in the data, then the system \eqref{evol} is equivalent to a
single \emph{nonlocal} and nonlinear scalar equation, namely
\begin{equation}
  \begin{gathered}
    \del_xy + \sg(y)\,\del_ty = 0, \com{where}\\
    \sg(y) := \wh\sg\big(\IR+y\big) = \big(1 + \IR+y\big)^{-\nu}.
  \end{gathered}
  \label{yevol}
\end{equation}
We can interpret this as a scalar transport equation, in which the
nonlinear wavespeed $\sg$ is nonlocal because $\IR+$ is a nonlocal
operator.  Note that \eqref{yevol} is the second equation of
\eqref{RI}; since $\mc R\del_t = -\del_t\mc R$, we recover the first
equation by simply applying reflection $\mc R$ to \eqref{yevol} and
using \eqref{rssymm}.

\subsection{Solution of the scalar equation at constant entropy}

We briefly describe the nonlinear evolution given by the nonlocal
scalar equation \eqref{yevol}.  We treat it as a simple nonlinear
transport problem, by solving along characteristics as usual.  These
are given by
\[
  \frac{dt}{dx} = \sg, \com{along which} \frac{dy}{dx} = 0,
\]
so $y$ is constant along characteristics, $y(x,t) = y^0(t_0)$.  Note
that in contrast to a scalar conservation law, we cannot conclude that
the characteristics are straight lines; this is due to the nonlocality
of the problem, because $\sg(x,t)$ is a function of both $y(x,t)$ and
$y(x,-t)$.  Note also that $\sg$ is always strictly positive.

Denote the characteristic through the reference point $(x_*,t_*)$,
parameterized by $x$, by $t = \tau_x := \tau_{x}(x_*,t_*)$.  Here the
subscript is a position locator rather than a partial derivative.  Then
$\tau_x$ satisfies the equation
\[
  \frac{d\tau_x}{dx} = \sg\big(y(x,\tau_x)\big), \quad
  \tau_{x_*} = t_*,
\]
and integrating gives
\begin{equation}
  \label{char}
  \tau_x(x_*,t_*) = t_* + \int_{x_*}^x \sg\big(y(\xi,\tau_\xi)\big)\;d\xi.
\end{equation}
This fully describes the characteristic field before gradient blowup,
which is the regime we are working in.  For later reference we record
the group property of the characteristic field,
\begin{equation}
  \label{group}
  \tau_{x_*}(x_*,t_*) = t_* \com{and}
  \tau_\xi\big(\eta,\tau_{\eta}(x_*,t_*)\big) = \tau_\xi(x_*,t_*).  
\end{equation}
The characteristic field yields the solution of the initial value
problem \eqref{yevol}, namely
\[
  y(x,t) = y^0\big(\tau_0(x,t)\big).
\]
This in turn defines the evolution operator $\mc E^\t(y^0)$, namely
\begin{equation}
  \label{et}
  \mc E^\t(y^0)(t) = y^0\big(\tau_0(\t,t)\big), \com{so that}
  \mc E^\t = \mc S_\phi,
\end{equation}
where the (linear) \emph{shift operator} $\mc S_\phi$ is defined by
\begin{equation}
  \label{Sphi}
  \mc S_\phi\big[y\big] := y\circ\phi,\qquad
  \mc S_\phi\big[y\big](t) = y\big(\phi(t)\big),
\end{equation}
and where  $\phi = \tau_0(\t,\cdot)$ is the (nonlinear) \emph{shift},
implicitly given by
\begin{equation}
  \label{shift}
  \phi(t) = \tau_0(\t,t) =
  t + \int_\t^0\sg\big(y(\xi,\tau_\xi)\big)\;d\xi.
\end{equation}

It follows that the nonlinear evolution
$\mc E^\t(y^0) = \mc S_\phi[y^0]$ can be regarded as acting linearly
\emph{once the characteristic field $\phi$ is known}, while the
nonlinearity (and nonlocality) is manifest in the determination of the
characteristics.  The above reasoning establishes the following lemma,
which describes the evolution under equation \eqref{yevol}.

\begin{lemma}
  \label{lem:evol}
  The evolution operator $\mc E^\t$, which evolves data $y^0$ through
  a spatial interval of width $\t$, is given by \eqref{et}, where the
  shift $\phi$ is a smooth function given implicitly by \eqref{shift}.
  This solution is valid and unique as long as the derivative
  $\del_ty$ remains finite, which holds as long as different
  characteristics don't intersect, that is, for values of $\t$ such
  that \eqref{char} can be solved uniquely for $t_*$, uniformly for
  $x\in[0,\t]$.
\end{lemma}

We note that the characteristics vary smoothly, because the wavespeed
$\sg(y)$ is a smooth function of $y$, and so can be expanded as
needed.  On the other hand, our prior work in~\cite{TYperEV} argues
that one cannot expand the linear (composition) $\mc S_\phi[y]$ in
$\phi$, because any such expansion up to $k$-th derivatives $y^{(k)}$
cannot be controlled in a Nash-Moser iteration~\cite{TYperNM}, because
the error is $O\big(y^{(k+1)}\big)$.

\subsection{Initial setup: finding the smallest tile}

By analyzing the linearization of \eqref{NU=U}, we characterized the
resonances and small divisors in the problem as eigenvalues of the
linearization of $\mc G(U) = \mc N(U)-U$.  We showed how a
Liapunov-Schmidt decomposition, coupled to a Nash-Moser iteration to
handle the small divisors, provides a consistent methodology for
solving \eqref{NU=U}, see \cite{TYperBif,TYperNM}.
In~\cite{TYdiff1,TYdiff2}, we further analyzed a scalar model
problem which isolated the problem of small divisors, and we proposed
a strategy for implementing a Nash-Moser iteration to solve
\eqref{NU=U}.

Preliminary numerical simulations indicated that solutions of
\eqref{NU=U} appear to all lie inside a smaller class of data which
satisfy restrictive symmetry properties.  The difficulty with
\eqref{NU=U} as stated is that this restricted class of solutions is
unidentifiable, because that formulation does not impose enough
symmetry that is respected by the nonlinear problem.  In order to
proceed further, we look for a modification of \eqref{NU=U} based on
symmetries that are respected by the \emph{nonlinear} evolution with
appropriate boundary conditions, which are also satisfied by the
linearized solutions.  In this more restricted class, the half-period
time shift is placed at axes of symmetry in (material) space $x$.

Our first restriction is to the space of solutions satisfying the
basic symmetry \eqref{symm}, so that $w$ is even and $\st w$ is odd as
functions of $t$ throughout the evolution.  Our second symmetry
follows by imposing the natural physical ``acoustic boundary
condition'' $u=0$ at $x=0$.  This further restricts the domain of the
nonlinear operator and imposes a spatial reflection property.  Our
proof ultimately shows that the periodic solutions we seek live within
this smaller domain.

The spatial reflection symmetry which follows from $u(0,\cdot)=0$ is
\begin{equation}
  w(-x,t) = w(x,t), \qquad
  \st w(-x,t) = -\st w(x,t),
\label{olsymm}
\end{equation}
and continuity at $x=0$ requires $\st w=0$ there, which is precisely
the acoustic boundary condition.  We thus impose the conditions
$\st w^0(t)=0$, $w^0$ even, and evolve this data forward in space to
$x=\ol\t:=\wt\t/2$.  We then obtain the solution on the entire entropy
level $-\ol\t<x<\ol\t$ by reflection through \eqref{olsymm}: by
uniqueness, these must coincide.  We note that the boundary symmetry
$\st w(x,\cdot)=0$ holds only at $x=0$, while the $w$ even, $\st w$
odd symmetry \eqref{symm} holds throughout the evolution.

It remains to impose a periodicity condition at the end of the
evolution.  For this we use the following principle: an even or odd
periodic function admits \emph{two} axes of symmetry, at both the
endpoint $0$ and midpoint $T/2$ of the interval, so that even or odd
periodic functions remain even or odd after a half-period shift,
respectively.  That is, if we set $\wt f(s):=f(s-T/2)$, then
\begin{equation}
  \label{reflper}
  \wt f(-s) = f(-s-T/2) = \pm f(s+T/2) = \pm f(s-T/2) = \pm\wt f(s).
\end{equation}
For periodicity,  it thus suffices to impose the additional \emph{shifted}
reflection symmetry at the end of the evolution, which becomes the shifted
symmetry axis.  This means that the practical interval of evolution is
actually only half the width of that of problem \eqref{NU=U}.  This
shifted reflection, analogous to \eqref{olsymm}, is
\begin{equation}
  \begin{aligned}
    w(\ell+x,t) &= w(\ell-x,t+T/2), \\
    \st w(\ell+x,t) &= -\st w(\ell-x,t+T/2),
  \end{aligned}
\label{ulsymm}
\end{equation}
and setting $x=0$, continuity at $\ell$ requires
\begin{equation}
  w(\ell,t+T/2)=w(\ell,t), \qquad
  \st w(\ell,t+T/2)=-\st w(\ell,t).
  \label{odddat}
\end{equation}
In other words, if $w(x,\cdot)$ and $\st w(x,\cdot)$ are defined for
say $0<x<\ell$, and satisfy the symmetry condition \eqref{odddat} at
$x=\ell$, then use of \eqref{ulsymm} allows us to extend the solution
to the interval $0<x<2\ell$.

Assuming these symmetries, it follows that we can generate a periodic
solution from a single tile evolving from the center of one entropy
level to the center of the other, with a single jump in between.  The
full periodic tile is then generated by the series of reflections
first \eqref{olsymm}, and then \eqref{ulsymm}, so the full spatial
period of the solution is $4(\ul\t+\ol\t)$.  We note that one
reflection converts the jump operator $\mc J$ to $\mc J^{-1}$, because
$\mc J$ is the jump from $U(x-)$ to $U(x+)$, and these are switched
under reflection in $x$.

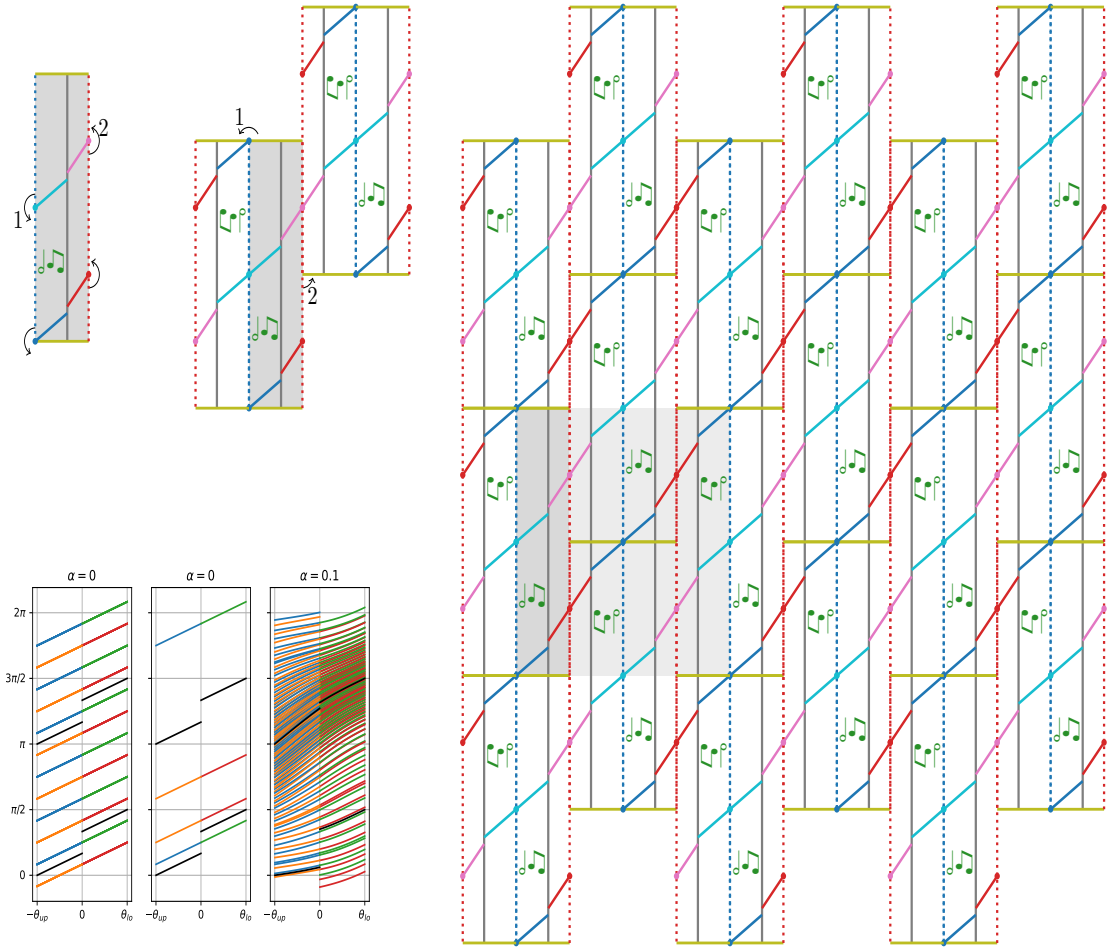
\begin{figure}[htb]
\def\olw{0.6}
\tikzset{
  onetile/.pic={
    \draw[style=very thick,color=C7] (\olw,0) -- (\olw,4);
    \draw[style=very thick,color=C8] (0,0) -- (1,0);
    \draw[style=very thick,color=C8] (0,4) -- (1,4);
    \draw[style=very thick,color=C0,style=dotted] (0,0) -- (0,4);
    \draw[style=very thick,color=C3,style=dotted] (1,0) -- (1,4);
    \draw[style=very thick,color=C0] (0,0) -- (\olw,0.7*\olw);
    \draw[style=very thick,color=C9] (0,2) -- (\olw,2+0.7*\olw);
    \draw[style=very thick,color=C3] (1,1) -- (\olw,1-0.8*\olw);
    \draw[style=very thick,color=C6] (1,3) -- (\olw,3-0.8*\olw);
    \node[color=C2!90!black] at (\olw/2,2*\olw) {\Large{\halfnote\twonotes}};
    \fill[color=C0] (0,0) circle (0.05);
    \fill[color=C9] (0,2) circle (0.05);
    \fill[color=C3] (1,1) circle (0.05);
    \fill[color=C6] (1,3) circle (0.05);
  }
}
\hspace*{-.6in}\resizebox{1.2\textwidth}{\textwidth}{%
  \begin{tikzpicture}
  \fill[color=C7!30] (-4,0) rectangle (-3,4);
  \draw [->] (-4,0.2) arc (90:245:0.2);
  \draw [->] (-4,2.2) arc (90:245:0.2) node[left] {1};
  \draw [->] (-3,2.8) arc (-90:75:0.2) node[right] {2};
  \draw [->] (-3,0.8) arc (-90:75:0.2);
  \pic at (-4,0) {onetile};
  \draw [->] (0.17,3.1) arc (30:150:0.2) node[above] {1};
  \draw [->] (1.,0.8) arc (-90:-15:0.2) node[below] {2};
  \fill[color=C7!30] (0,-1) rectangle (1,3);
  \pic at (0,-1) {onetile};
  \pic[rotate=180,transform shape] at (0,3) {onetile};
  \pic[rotate=180,transform shape] at (2,5) {onetile};
  \pic at (2,1) {onetile};

  \tikzset{shift={(5,-8)}}
  \fill[color=C7!30] (0,3) rectangle (1,7);
  \fill[color=C7!15] (1,3) rectangle (4,7);
  \foreach \nx in {0,4,8}
  {
    \foreach \ny in {0,4,8}
    {
      \pic at (\nx,\ny-1) {onetile};
      \pic[rotate=180,transform shape] at (\nx,\ny+3) {onetile};
      \pic[rotate=180,transform shape] at (\nx+2,\ny+5) {onetile};
      \pic at (\nx+2,\ny+1) {onetile};
    }
  }
  \draw (-6,2) node {\includegraphics[scale=0.5]{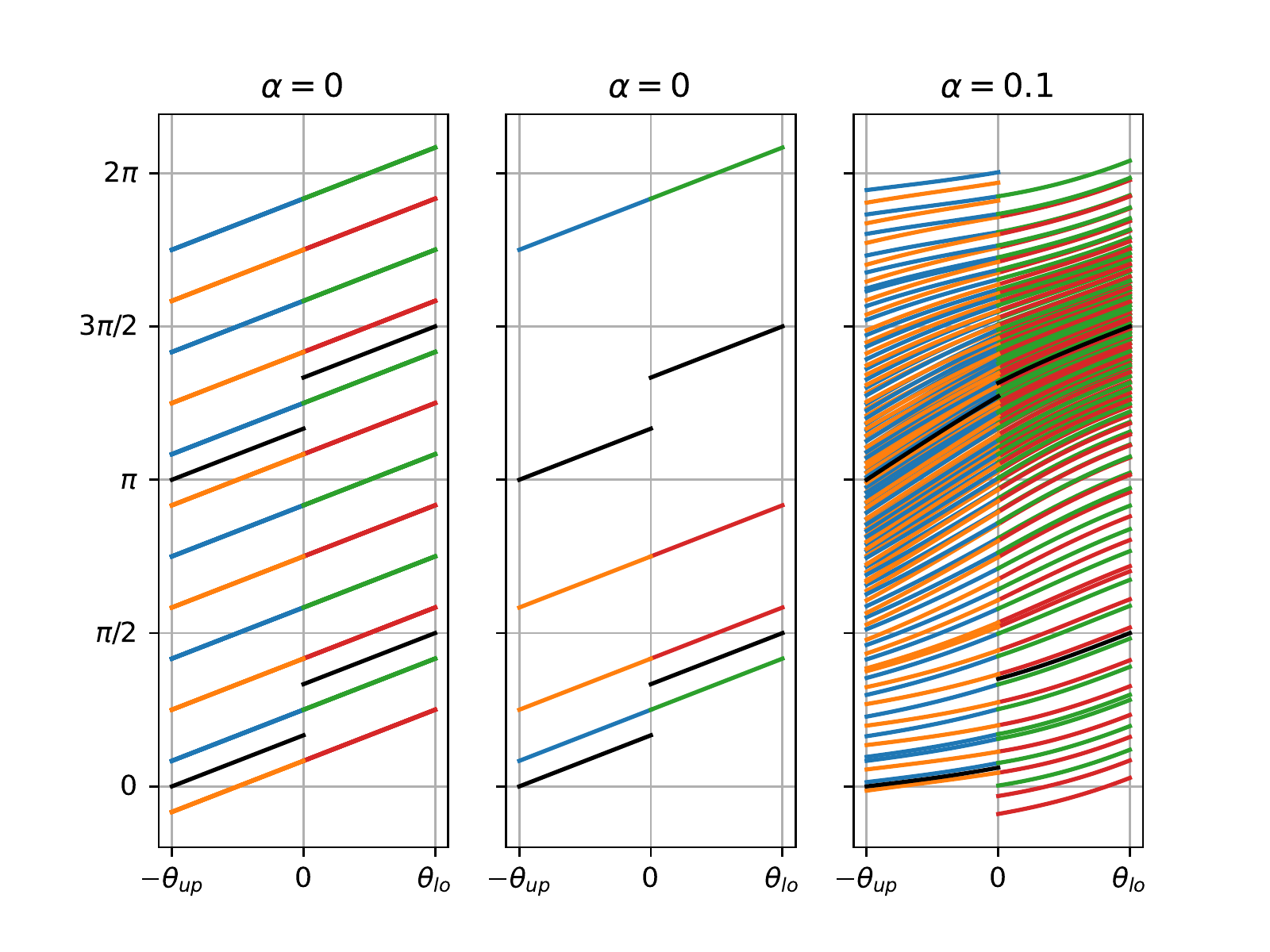}};
\end{tikzpicture}
}
\caption{Reflections of the generating tile.}
\label{fig:period}
\end{figure}

We illustrate the repeated reflections of the generating tile in
Figure \ref{fig:period}.  Recalling that two orthogonal reflections of
the plane yields a rotation of $\pi$ around the intersection point,
the standard reflection at $x=0$ is labelled `1', and the shifted
reflection at $x=\ell$ is labeled `2'.  The composition of these two
reflections is then translated periodically by $4\ell$ and $2\pi$ to
give the wave pattern on the plane.  A sketch of the perturbed tile,
showing the perturbed characteristics to leading order, is also shown.
The densityof sketched characteristics indicates the presence of
nonlinear rarefaction and compression.  Only the forward
characteristics are shown; the backward characteristics are obtained
by applying the time reflection $\mc R$.

In order to precisely write down the reduced nonlinear equation which
determines the minimal tile, we need to describe the double boundary
value problem which imposes periodicity.  We do this by
expressing the relevant operators in terms of the scalar Riemann
invariant $y$.  Assume the first entropy level $\ol s$ extends from
$-\ol\t<x<\ol\t$, and the second $\ul s$ extends from
$\ol\t<x<\ol\t+2\ul\t$.  We define a reduced nonlinear operator taking
data at $x=0$ to evolved solution at $x=\ol\t+\ul\t$ and impose
boundary conditions which imply periodicity.

Recall that the action of the jump $\mc J$ is $w \to w$,
$\st w\to J\st w$, ($\to$ indicating the jump from left to right)
where the scalar jump $J$ is a measure of how far apart the entropy
levels are.  Using \eqref{RIdef} and \eqref{rssymm}, the action of the
jump on $y$ is
\begin{equation}
y\to \mc Jy, \com{where}
\mc J := \IR+ + J\,\IR-,
\label{Jop}
\end{equation}
where $J$ is again the (scalar) size of the entropy jump, and
$\frac{\mc I\pm\mc R}2$ are projections onto the even and odd parts
$w$ and $\st w$ of $y$, respectively.

We next express the boundary conditions \eqref{olsymm} and
\eqref{ulsymm} in terms of the scalar Riemann invariant $y$.  Setting
$x=0$ in \eqref{olsymm} gives $\st w(0,t) = 0$, which is
\begin{equation}
   y(0,t) = \st y(0,t) = y(0,-t) \com{or} \ol y^0 = \mc R\,\ol y^0,
\label{ols}
\end{equation}
so that $\ol y^0 := y(0,\cdot)$ is even.

For the second entropy level, we set $x=\ell:=\ol\t+\ul\t$ in
\eqref{ulsymm}, so we get $w(\ell,t) = w(\ell,t+T/2)$ and
$\st w(\ell,t) = - \st w(\ell,t+T/2)$, or equivalently
\begin{equation}
\begin{gathered}
  y(\ell,t+T/2) = \st y(\ell,t) = y(\ell,-t), \com{or} \\
  y^\ell\big(\tau+\TS{\frac T4}\big)
  = y^\ell\big(-\tau+\TS{\frac T4}\big), \com{with} \tau = t+\TS{\frac T4}.
\end{gathered}
\label{shbd}
\end{equation}
It is convenient to define the shift or translation operator
$\mc S^\t$, which translates $v$ by a fixed amount $\t$, by
\begin{equation}
  [\mc S^\t\,v](t) := v(t-\t).
  \label{S}
\end{equation}
We use the convention that $\mc S^\t$ is a constant or uniform shift,
while $\mc S_\phi$ is a non-uniform shift given by \eqref{Sphi}.
These are then related by
\[
  \mc S^\t = \mc S_{\phi_\t}, \com{where}
  \phi_\t(t) := t-\t.
\]
Noting that
\[
  [\mc S^\t\,\mc R\,v](t) = v(-t+\t)
  = [\mc R\,\mc S^{-\t}\,v](t),
\]
our second boundary condition \eqref{shbd} can be written as
\begin{equation}
\mc S^{-T/4}\,y^\ell = \mc S^{T/4}\,\mc R\,y^\ell
= \mc R\,\mc S^{-T/4}\,y^\ell,
\label{uls}
\end{equation}
so that the target $\ut y^\ell$ is the quarter-period shift of an even
function,
\[
  \ut y^\ell = \mc S^{T/4}\,\ul y^\ell, \com{where}
  \ul y^\ell = \mc R\,\ul y^\ell,
\]
because $\ul y^\ell$ is even.

\subsection{The nonlinear equation}

We now give the reduced nonlinear operator whose solutions solve the
above double boundary value problem which generates a periodic
solution.  We treat the problem as an evolution in which we use the
first boundary condition \eqref{ols} to restrict the data of the
nonlinear operator, and the second boundary condition \eqref{uls} as
the target for the nonlinear evolution.

The boundary condition \eqref{ols} can be written $\IR+y^0=0$, so that
$y^0$ is even, and so we restrict our domain to (arbitrary periodic)
even data $y^0$.  This data is then evolved a distance $\ol\t$,
subjected to the jump \eqref{Jop}, and evolved a further distance
$\ul\t$.  We then impose the second boundary condition \eqref{uls},
namely $\IR-\mc S^{-T/4}y=0$.  Thus our periodic tile is generated by
the fully nonlinear (and nonlocal) reduced evolution equation
\begin{equation}
  \label{F=0}
  \IR-\,\mc S^{-T/4}\,\mc E^{\ul\t}\,\mc J\,\mc E^{\ol\t}\,y^0
  =: \mc F(y^0) = 0,
\end{equation}
where the data $y^0$ is even, and $T$-periodic.  At this stage, the
parameters $T$, $\ol\t$, $\ul\t$ and $J$ are yet to be determined;
each solution of \eqref{F=0} determines a periodic solution.

We note that we have replaced the subtraction $\mc N-\mc I$ in
\eqref{NU=U} by the linear projection $\IR-$ to impose periodicity.
In other words, we have replaced the \emph{periodic return} problem
\eqref{NU=U} with the \emph{periodicity by projection} problem
\eqref{F=0}.  Remarkably, this converts out-of-control small divisors
to uniform small divisors thus making the whole problem tractable.
Our re-expression of the nonlinear problem using projection allows a
factorization of the nonlinear operator, whereas the expression
\eqref{NU=U} has no obvious factorization.

\begin{theorem}
  \label{thm:tile}
  Any solution of the scalar equation \eqref{F=0} with even data $y^0$
  provides a minimal tile which generates by reflections, a solution of the
  compressible Euler equations which is $4\ell$-periodic in space and
  $T$-periodic in time.
\end{theorem}

By a \emph{minimal tile}, we mean a solution of \eqref{F=0} defined on
$[0,\ell]\times[0,T]$, which generates a space and time periodic
solution by this series of reflections and shifted reflections.

\begin{proof}
  Given a solution $y(x,t)$ defined on $[0,\ell]\times[0,T]$ and
  satisfying the stated conditions, we extend this to a full periodic
  solution of the system \eqref{nondim}, illustrated in
  Figure~\ref{fig:period}, as follows.  Recalling that $y$ is a
  Riemann invariant, we generate the state $(w,\st w)$ from $y$, and
  then reflect this $2\times2$ state using the boundary conditions.

  Recalling \eqref{rssymm} and \eqref{RIdef}, we define
  \[
    \begin{aligned}
      w(x,t) &:= \IR+y(x,t) = \frac{y(x,t)+y(x,-t)}2,\\
      \st w(x,t) &:= \IR-y(x,t) = \frac{y(x,t)-y(x,-t)}2,
    \end{aligned}
  \]
  for $(x,t)\in[0,\ell]\times[0,T]$.  We then extend $w$ using
  \eqref{ulsymm} and \eqref{olsymm} as
  \[
    w(x,t) :=
    \begin{cases}
      w(2\ell-x,t+T/2), &\ell\le x\le 2\ell,\\
      w(x-2\ell,t+T/2), & 2\ell\le x\le 3\ell,\\
      w(4\ell-x,t), & 3\ell\le x\le 4\ell,
    \end{cases}
  \]
  and similarly extend $\st w$ as
  \[
    \st w(x,t) :=
    \begin{cases}
      - \st w(2\ell-x,t+T/2), &\ell\le x\le 2\ell,\\
      \st w(x-2\ell,t+T/2), & 2\ell\le x\le 3\ell,\\
      - \st w(4\ell-x,t), & 3\ell\le x\le 4\ell.
    \end{cases}
  \]
  By construction, the solution so given is periodic and continuous at
  all boundaries $x=j\ell$, $j=0,\dots 4$.  Note that this uses the
  extra symmetry point of funtions that are both periodic and even or
  odd, \eqref{reflper}.
\end{proof}

\section{Linearization and Small Divisors}
\label{sec:lin}

Our first task is to linearize our main equation \eqref{F=0} around
$y=0$, which is itself a solution of the equation.  By our
non-dimensionalization, $y=0$ corresponds to the quiet state solution
$p=p_0$, $u=0$ of \eqref{lagr} over a variable entropy profile.  We
note that the jump $\mc J$, shift $\mc S^{-T/4}$ and projection onto
even, $\IR+$ are already linear operators, so it suffices to linearize
the evolution $\mc E^\t$.  Since $\sg(0)=1$, the linearization of the
nonlinear transport equation \eqref{yevol} around $y=0$ is just the
linear transport equation
\[
  \del_xY + \del_tY = 0,
\]
which has solution
\[
  Y(x,t) = Y^0(t-x) = \mc S^x\big[Y^0\big](t),
\]
where $\mc S^x$ is translation by $x$, as in \eqref{S}.  It follows
that the linearization of the evolution operator around $y=0$ through
$\t$ is
\begin{equation}
  \label{L}
  \mc L^\t := D\mc E^\t(0), \com{so}
  \mc L^\t[Y] = D\mc E^\t(0)[Y] = \mc S^\t Y.
\end{equation}
Here and throughout the paper, we will use square brackets $[\cdot]$
for inputs of (multi-)linear operators, and parentheses $(\cdot)$ to
denote inputs to nonlinear operators, and we will use upper case to
refer to arguments of (multi-)linear operators when convenient.  It
follows that the linearization of $\mc F$ around $y=0$, acting on $Y$,
is
\begin{equation}
  \label{DF}
  D\mc F(0)[Y] =
  \IR-\,\mc S^{-T/4}\,\mc L^{\ul\t}\,\mc J\,\mc L^{\ol\t}[Y]
  =: \IR-\,\mc L_0[Y],
\end{equation}
and because we have restricted the domain, we take $Y$ even, and we
have set
\begin{equation}
  \mc L_0 := \mc S^{-T/4}\,\mc L^{\ul\t}\,\mc J\,\mc L^{\ol\t},
  \label{L0}
\end{equation}  
and we note that each factor of $\mc L_0$ is invertible.

Our goal is to fully understand the kernel of
$D\mc F(0)=\IR-\,\mc L_0$, and choose parameters so that the kernel
consists of constant states and isolated modes.  This will allow us to
perturb the kernel to get a nontrivial solution of \eqref{F=0}.
Because $\mc L_0$ consists of translations and jumps, it respects
Fourier $k$-modes in $t$; this in turn allows us to explicitly
calculate the kernel and small divisors of $D\mc F(0)$.

We use the following notation: for a given (reference) time period
$T$, denote the $k$-mode by the $1\times2$ matrix
\begin{equation}
  \mc T_k = \mc T_k(T) := \(\c(k\frac{2\pi}Tt) & \s(k\frac{2\pi}Tt)\),
  \label{Tk}
\end{equation}
where $\c/\s$ abbreviate the trigonometric functions $\cos/\sin$.  It
then follows that any $T$-periodic function can be represented as
\[
  f(t) = \sum_{k\ge 0} \mc T_k\(a_k\\b_k\)
  = \sum a_k\,\text{c}(k\TS{\frac{2\pi}T}t) + b_k\,\text{s}(k\TS{\frac{2\pi}T}t),
\]
uniquely with $b_0=0$.  Identifying a mode as a $k$-mode requires
us to fix the reference period $T$: clearly
$\mc T_{jk}(jT) = \mc T_k(T)$ for any $j\ge 1$.

This notation allows us to express \eqref{L} in simple matrix terms,
as stated in the following lemma, which can readily be verified.

\begin{lemma}
  \label{lem:Tk}
  The above linear operators act on $\mc T_k$, as follows:
  \begin{align*}
    \mc R\,\mc T_k &= \mc T_k\,M(-1),\\
    \mc J\,\mc T_k &= \mc T_k\,M(J),\\
    \mc S^{-T/4}\,\mc T_k &= \mc T_k\,P^{-k},\\
    \IR-\,\mc T_k &= \s(k\TS{\frac{2\pi}T}\t)\(0&1\), \com{and}\\
    \mc L^\t\,\mc T_k &= \mc S^\t\,\mc T_k
         = \mc T_k\,R(k\TS{\frac{2\pi}T}\t),
  \end{align*}
  where $R(\t)$ is the usual rotation matrix, $M(\cdot)$ is the jump
  matrix from \eqref{jump}, and $P = R(\pi/2)$, that is
  \[
    R(\t) = \(\c\t&-\s\t\\\s\t&\c\t\), \quad
    M(J) = \(1&0\\0&J\), \com{and}
    P = \(0&-1\\1&0\).
  \]
\end{lemma}

\begin{corollary}
  \label{cor:DF0}
  The action of $D\mc F(0)$ on any even $Y=\sum a_k\,\c(k\piot t)$ is
  \[
    D\mc F(0)\Big[\TS{\sum}a_k\,\c(k\piot t)\Big]
    = \sum a_k\,\d_k\,\s(k\piot t),
  \]
  where the \emph{$k$-th divisor}
  $\d_k = \d_k\big(\ol\t,\ul\t,J;T\big)$ is the number
  \begin{equation}
    \label{dk}
    \d_k = \(0&1\)\,P^{-k}\,R(k\piot\ul\t)\,M(J)\,R(k\piot\ol\t)\(1\\0\).
  \end{equation}
\end{corollary}

\begin{proof}
  We write the data as
  $Y=\sum a_k\,\c(k\piot t) = \sum a_k\,\mc T_k\(1\\0\)$, and use
  Lemma~\ref{lem:Tk} in \eqref{DF}, to get the matrix expression for
  $\d_k$.  This in turn can be easily multiplied out to get an
  explicit expression.
\end{proof}

\subsection{The Kernel}

Because $D\mc F(0)$ respects $k$-modes, it follows that it has a
$k$-mode kernel if and only if $\d_k=0$.  There is always a 0-mode
kernel, because constant states solve the linearized (and nonlinear)
equation.  Our initial strategy, to construct the simplest solution as
in \cite{TYperLin}, is to choose parameters so that there is a 1-mode
kernel, $\d_1=0$, but no higher modes appear in the kernel, $\d_k\ne0$
for each $k>1$.  We call this the \emph{nonresonant case}.  Because
the 1-mode kernel is then isolated, our methods will show that it will
perturb to a solution of the nonlinear problem \eqref{F=0}.  To ease
notation, until we choose otherwise, we will assume a time period of
$T=2\pi$; putting the factor $\piot$ back in is straight-forward.

\begin{lemma}
\label{lem:dk}
  The linearization $D\mc F(0)$ has a 1-mode kernel if and only if 
  \begin{equation}
    \label{Jth}
    J = \frac{\c\ul\t}{\s\ul\t}\,\frac{\c\ol\t}{\s\ol\t}.
  \end{equation}
  Moreover, for almost every pair $(\ul\t,\ol\t)\in\B R^2$, $\d_k\ne0$
  for all $k\ge 2$, so that there are no other $k$-modes in the
  kernel.

  In the diagonal case $\ul\t=\ol\t=:\t$, some $\d_k=0$ if and only if
  $\t$ is a rational multiple of $\pi$, and if $\t\notin\pi\B Q$
  satisfies the diophantine condition
  \begin{equation}
    \label{dio}
    \Big|\frac\t\pi-\frac pq\Big| \ge \frac C{q^r},
  \end{equation}
  then the divisors $\d_k$ satisfy the explicit bound
  \begin{equation}
    \big|\d_k\big| \ge \frac{K}{k^{2(r-1)}}.
    \label{lowbd}
  \end{equation}
\end{lemma}

\begin{proof}
The proof requires us to calculate each $\d_k$.  First, we calculate
$\d_1$ to be 
\[
  \begin{aligned}
    \d_1 &= \(-1&0\)R(\ul\t)\,M(J)\,R(\ol\t)\(1\\0\)\\
    &= \( -\c\ul\t & \s\ul\t\)\(\c\ol\t\\J\,\s\ol\t\)\\
    &= J\,\s\ul\t\,\s\ol\t - \c\ul\t\,\c\ol\t,
  \end{aligned}
\]
and there is a 1-mode kernel if and only if $\d_1=0$, which is
equivalent to \eqref{Jth}.  This in turn allows us to write
\[
  M(J) = \frac1{\s\ul\t\,\s\ol\t}
  \(\s\ul\t&0\\0&\c\ul\t\)\(\s\ol\t&0\\0&\c\ol\t\),
\]
and we use this in the calculation of $\d_k$.

To calculate $\d_k$, at the $\ol\t$ level, we have
\[
  \begin{aligned}
  \(\s\ol\t&0\\0&\c\ol\t\)&R(k\ol\t)\(1\\0\)
  = \(\s\ol\t\,\c(k\ol\t)\\\c\ol\t\,\s(k\ol\t)\)\\[2pt]
  &= \frac{\s\big((k+1)\ol\t)}{2}\(1\\1\)
     + \frac{\s\big((k-1)\ol\t)}{2}\(-1\\1\).
  \end{aligned}
\]
Similarly, at the $\ul\t$ level, for $k$ even, we get
\[
  \begin{aligned}
    (-1)^{k/2}\(0&1\)&\,P^{-k}\,R(k\ul\t)\(\s\ul\t&0\\0&\c\ul\t\)\\
    &= \(\s(k\ul\t)\s\ul\t &
    \c(k\ul\t)\c\ul\t\)\\
    &= \frac{\c\big((k+1)\ul\t\big)}{2}\(-1&1\)+
    \frac{\c\big((k-1)\ul\t\big)}{2}\(1&1\),
  \end{aligned}
\]
while for $k$ odd,
\[
  \begin{aligned}
    (-1)^{(k+1)/2}\(0&1\)&\,P^{-k}\,R(k\ul\t)\(\s\ul\t&0\\0&\c\ul\t\)\\
    &= \(\c(k\ul\t)\s\ul\t &
    -\s(k\ul\t)\c\ul\t\)\\
    &= \frac{\s\big((k+1)\ul\t\big)}{2}\(1&-1\)+
    \frac{\s\big((k-1)\ul\t\big)}{2}\(-1&-1\).
  \end{aligned}
\]
It now follows from \eqref{dk} that for $k$ even, we have
\[
  \d_k = \frac{(-1)^{k/2}}{2\,\s\ul\t\,\s\ol\t}\,\Big(
  \c\big((k+1)\ul\t\big)\,\s\big((k-1)\ol\t\big)
  + \c\big((k-1)\ul\t\big)\,\s\big((k+1)\ol\t\big) \Big),
\]
while for $k$ odd,
\[
  \d_k = \frac{(-1)^{(k-1)/2}}{2\,\s\ul\t\,\s\ol\t}\,\Big(
  \s\big((k+1)\ul\t\big)\,\s\big((k-1)\ol\t\big)
  + \s\big((k-1)\ul\t\big)\,\s\big((k+1)\ol\t\big) \Big).
\]

For each fixed $k\ge 2$, the equation $\d_k=0$ holds at the zeros of a
nontrivial periodic function, so is a countable set of regular
solution curves $(\ul\t,\ol\t)\in\B R^2$, each of which has measure 0
in $\B R^2$.  Thus, being a countable union of measure zero sets, the
full ``resonant'' set
\[
  \Big\{(\ul\t,\ol\t)\;|\;\d_k=0 \text{ for some }k\ge2\Big\}
\]
also has measure 0 in $\B R^2$.

Restricting to the diagonal case $\ul\t=\ol\t=:\t$, we get
\begin{equation}
  \begin{aligned}
  \d_k &= \frac{(-1)^{k/2}}{2\,\s^2\t}\,\s(2k\t) \com{or}\\
  \d_k &= \frac{(-1)^{(k-1)/2}}{\s^2\t}\,
  \s\big((k+1)\t\big)\,\s\big((k-1)\t\big),
  \end{aligned}
  \label{dk1}
\end{equation}
for $k$ even or odd, respectively, so that $\d_k=0$ for some $k\ge 2$
if and only if $\t$ is some rational multiple of $\pi$.

Since $\s(\vartheta)\ge 2\vartheta/\pi$ for $0\le \vartheta\le \pi/2$,
it follows that for all $\vartheta$, we have the lower bound
\[
  \big|\s(\vartheta)\big| \ge
  \frac2\pi\,\min_{j\in\B Z}\big|\vartheta-\pi\,j\big|,
\]
so, provided $\t$ satisfies \eqref{dio}, we have for $q\ge 2$,
\[
  \big|\s(q\t)\big|\ge \frac2\pi\,\min\big|q\,\t-\pi\,j\big|
  \ge 2q\,\min\Big|\frac\t\pi-\frac jq\Big| \ge
  \frac{2\,C}{q^{r-1}}.
\]
Using this estimate in \eqref{dk1} now yields \eqref{lowbd}.
\end{proof}

\section{Factorization of the Nonlinear Operator}
\label{sec:onejump}

Our use of symmetry replaces the periodic return problem \eqref{NU=U}
with the periodicity by projection problem \eqref{F=0}.  In this
section we show that the fully nonlinear operator $\mc F$ factors
into the fixed linearized operator $\IR-\,\mc L_0$ times a regular
invertible nonlinear factor.  In this sense, by exploiting symmetry
and restricting to a smaller domain, we have been able to
``shrink-wrap'' the problem by reducing it while retaining the
fundamental nonlinear principle that leads to the existence of
periodic solutions, which is the balance of compression and
rarefaction.

We thus focus on the reduced equation \eqref{F=0}, namely
\[
  \mc F(y^0) = \IR-\,\mc S^{-T/4}\,\mc E^{\ul\t}\,
  \mc J\,\mc E^{\ol\t}\,y^0 = 0,
\]
for $y^0$ even.  Our first observation is that, according to
Corollary~\ref{cor:DF0} and Lemma~\ref{lem:dk}, the small divisors are
a fundamental effect of the leading order linearization around the
base state $y=0$, which will persist when perturbing to a nonlinear
solution.  By factoring the nonlinear operator, we show that the small
divisors remain uniform under perturbation.  Indeed, because of the
simple structure of $\mc F$ as a composition of operators, and since
each linear and nonlinear evolution is invertible by backwards
evolution, we are able to factor the linearization $\mc L_0$, which
generates the small divisors, out of the fully nonlinear operator
$\mc F$, as follows.

\begin{theorem}
  \label{thm:fact}
  The nonlinear operator $\mc F$ given in \eqref{F=0} can be factored
  as
  \begin{equation}
    \label{fact}
    \mc F = \IR-\,\mc L_0\,\mc N, \com{where}
    \mc N := \ulc{N}\,\olc{N}.
  \end{equation}
  Here $\mc L_0$ is given by \eqref{L0}, and for $y = O(\a)$, $\ulc{N}$ and
  $\olc{N}$ are regular nonlinear operators, with
  \begin{equation}
    \label{NIa}
    D\olc N = \mc I + O(\a), \qquad D\ulc N = \mc I + O(\a).
  \end{equation}
  It follows that the fully nonlinear equation \eqref{F=0} can be
  rewritten as
  \begin{equation}
    \mc F(y^0) = 0 \com{iff}
    \mc N(y^0) \in \ker\big\{\IR-\,\mc L_0\big\}.
    \label{Neq}
  \end{equation}    
\end{theorem}

The upshot of this factorization is that the small divisors, which are
determined by the leading order linearization $\IR-\mc L_0$, have been
explicitly factored out, and so are necessarily uniform.  In addition,
\eqref{Neq} can be regarded as containing only a projection with no
small divisors, and in addition $\mc N = \ulc{N}\,\olc{N}$ is a
regular perturbation of the identity.  By this, we are able to
solve \eqref{Neq} using the standard implicit function theorem,
without having to resort to a more technical Nash-Moser iteration.

\begin{proof}
  The proof is a direct computation, using the fact that the
  linearized evolutions $\mc L^\t$ and jump $\mc J$ are invertible.
  Using \eqref{F=0} and \eqref{L0}, we write
  \[
    \begin{aligned}
      \mc F(y^0) &= \IR-\,\big\{\mc L^{\ul\t}\mc J\mc L^{\ol\t}
      \,{\mc L^{\ol\t}}^{-1}\mc J^{-1}
      {\mc L^{\ul\t}}^{-1}\big\}\,\mc E^{\ul\t}
      \,\mc J\mc E^{\ol\t}\,y^0\\
      &= \IR-\,\mc L_0\,\big(\mc J\mc L^{\ol\t}\big)^{-1}\,
      {\mc L^{\ul\t}}^{-1}\mc E^{\ul\t}
      \big(\mc J\mc L^{\ol\t}\big)\,{\mc L^{\ol\t}}^{-1}
      \mc E^{\ol\t}\,y^0,
    \end{aligned}
  \]
  where $\{\square\}=\mc I$.  This is \eqref{fact}, provided we define
  \begin{equation}
    \ulc{N}:=\big(\mc J\mc L^{\ol\t}\big)^{-1}\,
    {\mc L^{\ul\t}}^{-1}\mc E^{\ul\t}
    \big(\mc J\mc L^{\ol\t}\big) \com{and}
    \olc{N}:={\mc L^{\ol\t}}^{-1}\mc E^{\ol\t}.
    \label{Ndef}
  \end{equation}   
  Recall that, as long as derivatives remain bounded, $\mc E^\t$ is a
  regular operator which propagates the value of $y^0$ along
  characteristics.  Also, if $y^0=0$, the characteristics for
  $\mc L^\t$ and $\mc E^\t$ coincide, and $\mc L^\t = D\mc E^\t(0)$.
  From this it follows that $\mc N^\t:={\mc L^\t}^{-1}\mc E^\t$ is a
  regular perturbation of the identity, and moreover
  \[
    D\mc N^\t(y) = {\mc L^\t}^{-1}D\mc E^\t(y) = \mc I + O(\a) \com{if}
    y = O(\a),
  \]
  because $\mc L^\t=D\mc E^\t(0)$.  Thus \eqref{NIa} follows for
  $\olc{N}$, and it follows for $\ulc{N}$ since this property is
  preserved under conjugation by (fixed) linear operators.
\end{proof}

\subsection{Removal of the small divisors}

Our fully nonlinear equation now has the factored form \eqref{fact},
namely
\[
  \IR-\,\mc L_0\,\mc N\,y^0 = 0, \quad y^0 \text{ even},
\]
where the nonlinear part $\mc N$ is bounded invertible for small
regular data.  We now introduce the Hilbert spaces which allow us to
find solutions which are perturbations of the constant state $y=0$.
To leading order, these have the form
$\a\,\c(t)\in\ker\{\IR-\mc L_0\}$, and are parameterized by the
amplitude $\a$ provided \eqref{Jth} holds.  Our program is to
construct solutions of the form
\begin{equation}
  \label{y0}
  y^0(t) = \a\,\c(t) + z + \sum_{j>1}a_j\,\c(jt),
\end{equation}
where the 0-mode $z$ (also in $\ker\{\IR-\mc L_0\}$), and higher mode
coefficents $a_j$ of order $O(\a^2)$ are unknowns to be found.

Observe first that there is one free variable and one equation
corresponding to each mode because the operator \eqref{fact} takes
even modes to odd modes.  That is, after projection by $\IR-$, we have
one equation for each coefficient of $\s(jt)$ in
$\mc L_0\,\mc N\,y^0$, $j\ge 1$.  On the other hand, the free
parameters are the $a_j$, $j>1$, and 0-mode $z$, so each equation
corresponds uniquely to an unknown.  Thus formally we expect to get a
solution of \eqref{fact} of the form \eqref{y0} for any $\a$
sufficiently small in the nonresonant case.

The nonresonant case is characterized by the conditions $\d_1=0$,
$\d_j\ne0$ for all $j>1$.  According to Lemma~\ref{lem:dk}, this is in
turn implied for almost every pair $(\ul\t,\ol\t)$, provided $J$ is
chosen according to \eqref{Jth}.  We thus fix a nonresonant pair
$(\ul\t,\ol\t)$.  The following development is an explicit version of
the Liapunov-Schmidt decomposition of the nonlinear operator $\mc F$
into the auxiliary equation and corresponding bifurcation equation in
the nonresonant case.

As a roadmap, we briefly describe the abstract bifurcation problem and
its solution.  Consider the problem of solving equations of the form
\[
  f(\a,z,w) = 0, \com{with} f(0,0,0)=0,
\]
for solutions parameterized by the small amplitude parameter $\a$ (for
us, $\mc F$ plays the role of $f$).  We cannot do this with a direct
application of the implicit function theorem because the gradient
$\nabla_{(z,w)}f\big|_0$ is not invertible.  Here we assume
\[
  z\in\ker\Big\{\nabla_{(z,w)}f\big|_0\Big\} \com{and}
  w\in\ker\Big\{\nabla_{(z,w)}f\big|_0\Big\}^\perp.
\]
Assuming $\frac{\del f}{\del w}\big|_0$ is invertible, we first find
\[
  w(\a,z) \com{so that}
  f\big(\a,z,w(\a,z)\big)=0;
\]
this is known as the \emph{auxiliary equation}.  Next, we wish to
solve for
\[
  z=z(\a) \com{such that} f\big(\a,z(\a),w(\a,z(\a))\big) = 0,
\]
In our problem, $\frac{\del f}{\del z}\big|_0$ is not invertible, so
we replace $f$ by the equivalent function
\[
  g(\a,z) :=
  \begin{cases}
    \frac1\a\,f\big(\a,z,w(\a,z)\big), &\a\ne 0,\\[3pt]
    \frac{\del f}{\del \a}\big(0,z,w(0,z)\big), &\a = 0,
  \end{cases}
\]
and the equation $g(\a,z)=0$ is known as the \emph{bifurcation
  equation}.  This bifurcation equation can be solved by the implicit
function theorem, provided
\[
  \frac{\del g}{\del z}\Big|_{(0,0)} \equiv
  \frac{\del^2 f}{\del z\,\del\a}\Big|_{(0,0,0)} \ne 0.
\]
The decomposition of the domain of the function $f$ into a direct sum
of the kernel and its orthogonal complement is known as the
Liapunov-Schmidt decomposition~\cite{Golush,TYperBif}.

In our application, the (infinite dimensional) gradient, although
invertible, has unbounded inverse because of the presence of small
divisors $\d_k$.  However, the factorization \eqref{fact} means that
these small divisors are uniform, and so we can handle them by an
appropriate adjustment of the associated Hilbert space norms.  It is
remarkable that by factoring the nonlinear operator, we are able to
avoid difficult technical issues, like diophantine estimates such as
\eqref{dio}, \eqref{lowbd}, which are common in Nash-Moser iterations.
When solving the bifurcation equation, we must calculate the second
derivative of an infinite dimensional nonlinear evolution operator,
which is an important technical part of the overall argument.

Assume $y^0$ lies in the Sobolev space $H^s$, so that for small data
$y^0$, the evolution $\mc E^x(y^0)$ stays in $H^s$, by the local
existence theory~\cite{Majda,Taylor}.  To apply the implicit function
theorem, define
\begin{equation}
  \label{Hsplit}
  \begin{aligned}
    \mc H_1 &:= \big\{z + \a\,\c(t)\;\big|\;z, \a\in\B R\big\}
    \com{and}\\
    \mc H_2 &:= \Big\{ \sum_{j>1}a_j\,\c(jt)\;\Big|
    \;\sum a_j^2\,j^{2s}<\infty\Big\},    
  \end{aligned}
\end{equation}
so that $y^0\in\mc H_1\oplus\mc H_2$, and define
\[
  \begin{gathered}
    \whc F:\mc H_1\times\mc H_2\to H^s \com{by}\\
    \whc F(y_1,y_2) := \mc F(y_1+y_2) =
    \IR-\,\mc L_0\,\mc N\,(y_1+y_2),
  \end{gathered}
\]
so that $\whc F$ is a continuous (nonlinear) operator.  It follows that
the partial Frechet derivative
\[
  D_{y_2}\whc F(0,0):\mc H_2\to H^s \com{is}
  D_{y_2}\whc F(0,0) = \IR-\,\mc L_0,
\]
and according to Corollary~\ref{cor:DF0}, this acts as
\[
  D_{y_2}\whc F(0,0)\Big[\sum_{j>1} a_j\,\c(jt)\Big]
  = \sum_{j>1} a_j\,\d_j\,\s(jt).
\]
By our choice of parameters, we have $\d_j\ne0$ for $j>1$, so that
$D_{y_2}\whc F(0,0)$ is injective, but the inverse is not bounded as a map
$H^s\to H^s$ because of the presence of the small divisors $\d_k$.
However, because the small divisors are uniform, we can define a new
norm on the target space $H^s$ so that $D_{y_2}\whc F(0,0)$ becomes an
isometry, and in particular is bounded invertible on its range, as in
\cite{TYdiff2}.  Thus we set
\begin{equation}
  \begin{aligned}
  \mc H_+ &:= \Big\{ y = \sum_{j>1}c_j\,\s(jt)\;\Big|
  \;\|y\|<\infty\Big\}, \com{and}\\
  \mc H &:= \big\{ \beta\,\s(t) \big\} \oplus \mc H_+, \com{with norm}\\
  \|y\|^2 &:= \beta^2 + \sum_{j>1} c_j^2\,\d_j^{-2}\,j^{2s},
  \end{aligned}
  \label{Hplus}
\end{equation}
where for convenience we set $c_1:= \beta$, which isolates the
1-mode kernel.  Referring to \eqref{dk}, we see that each divisor
$\d_j$ is bounded above, $\d_j\le\max\{1,J\}$, so that
\[
  \|y\|^2_{H^s} \le \frac1{\max\{1,J^2\}}\,\|y\|^2,
  \com{and} \mc H\subset H^s,
\]
and it is clear that $\mc H$ is a Hilbert space.  Finally, let $\Pi$
denote the projection 
\[
  \Pi:\mc H\to \mc H_+, \quad
  \Pi\Big[\beta\,\s(t)+\sum_{j>1}a_j\,\s(jt)\Big] := \sum_{j>1}a_j\,\s(jt).
\]
Note that we have constructed these spaces so that
\[
  \ker\{D_{y_2}\whc F(0,0)\} = \mc H_1, \quad
  \text{ran}\{D_{y_2}\whc F(0,0)\} = \mc H_+,
\]
so that $\Pi D_{y_2}\whc F(0,0) = D_{y_2}\whc F(0,0)$, and moreover
$D_{y_2}\whc F(0,0):\mc H_2\to\mc H_+$ is an isometry, and thus bounded
invertible on its range $\mc H_+$.  The invertibility of
$D_{y_2}\whc F(0,0)$ allows a regular application of the classical
implicit function theorem on Hilbert spaces.

\begin{lemma}
  \label{lem:himodes}
  There is a neighborhood $\mc U\subset\mc H_1$ of the origin and a
  unique $C^1$ map
  \[
    W:\mc U\to\mc H_2, \com{written}
    W\big(z+\a\,\c(t)\big) =: W(\a,z) \in\mc H_2,
  \]
  such that, for all $z+\a\,\c(t)\in\mc U$, we have
  \begin{equation}
    \Pi\,\whc F\big(z+\a\,\c(t),W(\a,z)\big)  =
    \Pi\,F\big(z+\a\,\c(t)+W(\a,z)\big) = 0.
    \label{Weq}
  \end{equation}
\end{lemma}

\begin{proof}
  This is a direct application of the implicit function theorem.
  Recall that this states that if $\mc H_1$, $\mc H_2$ and $\mc H_+$
  are Hilbert spaces, and
  \[
    \mc G := \Pi\whc F:\Omega\subset\mc H_1\times\mc H_2\to\mc H_+
  \]
  is a continuously differentiable map defined on an open neighborhood
  $\Omega$ of $(0,0)$, and satisfying $\mc G(0,0)=0$, and if the
  linear (partial derivative) map $D_{y_2}\mc G(0,0):\mc H_2\to\mc H_+$ is
  bounded invertible, then there is an open neighborhood
  $\mc U_1\subset \mc H_1$ of 0 and a unique differentiable map
  $W:\mc U_1\to\mc H_2$, such that $\mc G\big(x,W(x)\big)=0$ for all
  $x\in\mc U_1$, see e.g.~\cite{Jost}.  Because we have built our
  Hilbert spaces so that $D_{y_2}\mc G$ is an isometry, the result follows
  immediately.
\end{proof}

It is remarkable that we do \emph{not} require any estimates on the
decay rate of the small divisors, because we have reframed the problem
as the vanishing of a composition of operators, \eqref{Neq}.  Indeed,
the faster the small divisors decay, the smoother the corresponding
periodic solution must be, as seen by the norm \eqref{Hplus}.  
By splitting the Hilbert spaces into orthogonal complements in
\eqref{Hsplit} and \eqref{Hplus}, we have explicitly carried out the
Liapunov-Schmidt decomposition of the nonlinear operator $\mc F$
around 0, as anticipated in \cite{TYperBif}.

\subsection{Solution of the bifurcation equation}

To complete the solution of equation \eqref{F=0}, \eqref{Neq}, it
remains to show that we can ensure, after use of \eqref{Weq}, that the
remaining component of $\mc F(y^0)$, which is the component orthogonal
to the range of $D_{y_2}\whc F(0,0)$, also vanishes.  From the
decomposition \eqref{Hplus}, this is the (scalar) \emph{bifurcation
  equation},
\begin{equation}
  \label{bif}
  f(\a,z) :=
  \Big\langle \s(t), \whc F\big(z+\a\,\c(t),W(\a,z)\big)\Big\rangle = 0.
\end{equation}
Here $\a$ is the amplitude of the linearized solution, and $z$ is a
0-mode adjustment which can be regarded as bringing compression and
rarefaction back into balance, as described in \cite{TYdiff2}.  We
will see presently that the existence of a solution of \eqref{bif} is
a consequence of the genuine nonlinearity of the system, which states
that the nonlinear wavespeed depends nontrivially on the state, and is
therefore controlled by $z$.

The scalar function $f(\a,z)$ given by \eqref{bif} is defined on the
neighborhood $\mc U$, which with a slight abuse of notation can be
regarded as $\mc U\subset\B R^2$, so we write
\[
  f:\mc U\subset\B R^2\to\B R, \com{with} f(0,0)=0.
\]
As in our description of the Liapunov-Schmidt method, we would like to
apply the implicit function theorem to $f$, to get a curve $z=z(\a)$
on which $f\big(\a,z(\a)\big)=0$.  We cannot apply this directly,
because $\frac{\del f}{\del z}\big|_{(0,0)} = 0$, since all 0-modes
are killed by the projection $\IR-$.  Thus we consider the second
derivative $\frac{\del^2f}{\del z\del\a}\big|_{(0,0)}$, and if this is
nonzero, we can conclude the existence of a solution.

One way of effectively calculating the second derivative is to replace
$f$ with the function
\begin{equation}
  \label{gdef}
  \begin{aligned}
    g(\a,z) &:= \frac1\a\, f(\a,z), \quad \a\ne0,\\[2pt]
    g(0,z) &:=  \frac{\del f}{\del\a}(0,z), 
  \end{aligned}
\end{equation}
which is consistently defined because $W(0,z)=0$, and so also
$f(0,z) = 0$, for all $z$ near 0.  It is then clear that
\[
  f(\a,z) = 0 \com{iff} g(\a,z) = 0 \com{for} \a \ne 0,
\]
and it suffices to apply the implicit function theorem to $g$.
To begin, we first show that $W$ does not essentially affect the
argument.

\begin{lemma}
  \label{lem:W}
  The map $W(\a,z)$ found in Lemma~\ref{lem:himodes} satisfies the
  estimate
  \begin{equation}
    \label{West}
    W(\a,z) = o(|\a|), \com{so that}
    \frac{\partial W}{\partial\a}\to 0,
  \end{equation}
  uniformly for $z$ in a neighborhood of 0.
\end{lemma}

\begin{proof}
  Since $\mc F(z)=0$, we have
  \[
    W(\a,z) \in \mc H_2, \com{with} W(0,z) = 0,
  \]
  and, setting $y=z+\a\,\c(t) + W(\a,z)$, we have by \eqref{Weq}
  \[
    0 = \Pi\, \mc F(y) =
    \Pi\,\IR-\,\mc L_0\,\mc N\big(z+\a\,\c(t) + W(\a,z)\big).
  \]
  Differentiating with respect to $\a$ and setting $\a=0$,
  we get
  \[
    \begin{aligned}
      0 = \frac{\del\Pi\mc F(y)}{\del\a}\Big|_{\a=0}
      &= \Pi\,\IR-\,\mc L_0\,D\mc N(z)\Big[\c(t) +
      \frac{\del W}{\del\a}\Big|_{\a=0}\Big] \\
      &= \Pi\,\IR-\,\mc L_0\,D\mc N(z)\Big[\frac{\del W}{\del\a}\Big|_{\a=0}\Big],
    \end{aligned}
  \]
  since, in addition to $\mc L_0$, $D\mc N(z)$ respects modes,
  although for $z\ne0$ the linear wavespeed is changed.  Since
  $\Pi\,\IR-\,\mc L_0:\mc H_2\to\mc H_+$ is invertible, and
  $D\mc N(z) = \mc I + O(z)$, the result follows.
\end{proof}

To calculate $\del g/\del z$ at $(0,0)$, we first calculate $\del\mc
F/\del\a$ and set $\a=0$.  As above, we have
\[
  \begin{aligned}
    \frac{\del\mc F(y)}{\del\a}\Big|_{\a=0}
    &= \IR-\,\mc L_0\,D\mc N\big(z+\a\,\c(t)+W(\a,z)\big)
    \Big[\c(t) +
    \frac{\del W}{\del\a}\Big]\Big|_{\a=0}\\
    &= \IR-\,\mc L_0\,D\mc N\big(z\big)\big[\c(t)\big],
  \end{aligned}
\]
where we have used Lemma~\ref{lem:W}, and $D\mc N(z)$ is diagonal.
Differentiating this in $z$, setting $z=0$, and using \eqref{gdef} and
\eqref{bif}, we get
\begin{equation}
  \label{dgdz1}
  \frac{\del g}{\del z}\Big|_{(0,0)} = \Big\langle \s(t),
  \IR-\,\mc L_0\,D^2\mc N\big(0\big)\big[1,\c(t)\big]\Big\rangle.
\end{equation}
From \eqref{fact}, we have $\mc N = \ulc{N}\,\olc{N}$, so by the chain
rule we have
\[
  D\mc N(y)[Y] = D\ulc N(\olc Ny)\big[D\olc N(y)[Y]\big],
\]
and this in turn implies
\[
  \begin{aligned}
    D^2\mc N(y)\big[Y_1,Y_2\big]
    = D^2&\ulc N(\olc Ny)\big[D\olc N(y)[Y_1],D\olc N(y)[Y_2]\big]\\
    &{}+ D\ulc N(\olc Ny)\big[D^2\olc N(y)[Y_1,Y_2]\big],
  \end{aligned}
\]
and since $D\mc N = \mc I + O(\a)$, when we set $y=0$, $Y_1=1$ and
$Y_2=\c(t)$, \eqref{dgdz1} becomes
\begin{equation}
  \label{dgdz}
  \frac{\del g}{\del z}\Big|_{(0,0)} = \Big\langle \s(t),
  \mc L_0\,D^2\ulc N(0)\big[1,\c(t)\big]+
  \mc L_0\,D^2\olc N(0)\big[1,\c(t)\big]\Big\rangle,
\end{equation}
where we have dropped $\IR-$ because $\s(t)$ is odd.  The final step
in the proof of existence of space and time periodic solution by the
Liapunov-Schmidt method is to show that this derivative \eqref{dgdz}
is nonzero,
\begin{equation}
  \label{dgdz0}
  \frac{\del g}{\del z}\Big|_{(0,0)}  \ne 0.
\end{equation}

In order to establish \eqref{dgdz0}, referring to \eqref{Ndef}, we
need to calculate the second Frechet derivative of the evolution
operator, $D^2\mc E^\t(y^0)\big[Y_1^0,Y_2^0\big]$.  To do this, we use
the solution formula \eqref{et}, \eqref{Sphi}, \eqref{shift}.

\begin{lemma}
  \label{lem:D2E}
  The linearization (or Frechet derivative) of the evolution operator
  $\mc E^\t$ given by \eqref{et} at $y^0$ in the direction $Y^0$ is
  given by
  \begin{equation}
    \label{DE}
    D\mc E^\t(y^0)\big[Y^0\big] =
    \mc S_\phi\big[Y^0\big] + \mc S_\phi\Big[\frac d{dt}y^0\Big]\cdot
    D\phi(y^0)\big[Y^0\big],
  \end{equation}
  where the linearization of the shift is given by
  \begin{equation}
    \label{Dphi}
    D\phi(y^0)\big[Y^0\big] = 
    \sg'\big(y^0(\tau_0)\big)\,\int_\t^0\IR+\Big(Y(\xi,\tau_\xi)
    + \frac{\del y}{\del t}\Big|_{(\xi,\tau_\xi)}\,D\tau_\xi[Y^0]\Big)
    \;d\xi,
  \end{equation}
  and where the linearized solution $Y(\xi,\tau_\xi)$ and
  characteristic field $D\tau_x$ are given in \eqref{DY} and
  \eqref{Dtau} below, respectively.  The second Frechet derivative of
  $\mc E^\t$ is
  \begin{align}
    D^2\mc E^\t(y^0)\big[Y^0_1,Y^0_2\big] ={}&
      \mc S_\phi\Big[\frac d{dt}Y^0_2\Big]\cdot
      D\phi(y^0)\big[Y^0_1\big] +
      \mc S_\phi\Big[\frac d{dt}Y^0_1\Big]\cdot
      D\phi(y^0)\big[Y^0_2\big] \nonumber\\
      &+ \mc S_\phi\Big[\frac{d^2}{dt^2}y^0\Big]\cdot
      D\phi(y^0)\big[Y^0_1\big]\cdot D\phi(y^0)\big[Y^0_2\big]\nonumber\\
      &{}+ \mc S_\phi\Big[\frac{d}{dt}y^0\Big]\cdot
        D^2\phi(y^0)\big[Y^0_1,Y^0_2\big],
        \label{D2E}
  \end{align}
  where $\cdot$ denotes multiplication.  For initial data
  $y^0\in H^{s+1}$, the evolved solution $y(x)\in H^{s+1}$ as long as
  the derivative $\frac{dy}{dt}$ remains finite.  In this case, the
  linearizations \eqref{DE} and \eqref{Dphi} are bounded $H^s\to H^s$,
  and the bilinear second derivative is bounded
  $H^s\times H^s\to H^{s-1}$.
\end{lemma}

\begin{proof}
  In order to compute the linearization at fixed $x=\t$, we need to
  understand the nonlinear evolution throughout the interval
  $0<x<\t$.  We thus apply \eqref{et} at $x$, to get
  \begin{equation}
    \label{etx}
    \mc E^x(y^0)(t) = y^0\big(\tau_0(x,t)\big), \com{so that}
    \mc E^x(y^0) = \mc S_{\tau_0(x,\cdot)},
  \end{equation}
  where $t_\xi(x,t)$ describes the characteristic field, and is given
  by \eqref{char}, namely
  \begin{equation}
    \label{charxi}
    \tau_\xi(x,t) = t + \int_{x}^\xi
    \sg\big(y(\eta,\tau_\eta)\big)\;d\eta,
    \qquad \tau_\eta = \tau_\eta(x,t).
  \end{equation}

  Perturb the data $y^0$ by $Y^0$, and let $\tau_\xi$ and
  $\tau_\xi+T_\xi$ denote the characteristic fields of $y^0$ and the
  perturbed solution $y^0+Y^0$, respectively.  Then by \eqref{etx}, we
  have
  \[
    \mc E^x(y^0) = \mc S_{\tau_0}[y^0]  \com{and}
    \mc E^x\big(y^0+Y^0\big) = \mc S_{\tau_0+T_0}\big[y^0+Y^0\big],
  \]
  where the reference point $(x,t)$ is understood.  Then subtracting
  and rearranging gives
  \[
    \begin{aligned}
      \mc E^x\big(y^0+Y^0\big) - \mc E^x(y^0)
      &= \mc S_{\tau_0+T_0}\big[Y^0\big] +
      \big(\mc S_{\tau_0+T_0}-\mc S_{\tau_0}\big)\big[y^0\big]\\
      &= \mc S_{\tau_0}\big[Y^0\big] +
      T_0\cdot \mc S_{\tau_0}\Big[\frac d{dt}y^0\Big]
      + O(\|Y^0\|^2),
    \end{aligned}
  \]
  where we have taken $T_0=O(\|Y^0\|)$ because the characteristic
  field \eqref{charxi} is regular.  Taking the limit of small
  $\|Y^0\|$ then yields the linearization
  \begin{equation}
    \label{DEx}
    D\mc E^x(y^0)\big[Y^0\big] = \mc S_{\tau_0}\big[Y^0\big]
    + D\tau_0(y^0)\big[Y^0\big]\cdot
    \mc S_{\tau_0}\Big[\frac d{dt}y^0\Big],
  \end{equation}
  where again the reference point for $\tau_0$ and $D\tau_0$ is
  $(x,t)$.

  To linearize the characteristic field, fix the reference point
  $(x,t)$, and let $\tau_\xi+T_\xi$ denote the perturbed
  characteristic field \eqref{charxi}, so that
  \[
    \tau_\xi + T_\xi = t + \int_{x}^\xi
    \sg\big((y+Y)(\eta,\tau_\eta+T_\eta)\big)\;d\eta,
  \]
  and subtracting \eqref{charxi} gives
  \[
    T_\xi(x,t) = \int_{x}^\xi\sg\big((y+Y)(\eta,\tau_\eta+T_\eta)\big)
    -\sg\big(y(\eta,\tau_\eta)\big)\;d\eta,
  \]
  so we must linearize the integrand.  By \eqref{evol}, we have
  $\sg(y)=\big(1+\IR+y\big)^{-\nu}$, and so
  \[
    \begin{aligned}
      \sg\big((y&+Y)(\eta,\tau_\eta+T_\eta)\big)
      -\sg\big(y(\eta,\tau_\eta)\big) \\ 
      &\approx \sg'\big(y(\eta,\tau_\eta)\big)\,\IR+\Big(
      Y(\eta,\tau_\eta+T_\eta) +
      y(\eta,\tau_\eta+T_\eta)-y(\eta,\tau_\eta)\Big)
    \end{aligned}
  \]
  with $\sg'(y) := -\nu\,\big(1+\IR+y\big)^{-\nu-1}$.  To linearize we
  drop higher order terms, which results in
  \begin{equation}
    \label{slin}
    \approx \sg'\big(y(\eta,\tau_\eta)\big)\,\IR+\Big(Y(\eta,\tau_\eta)
    + \frac{\del y}{\del t}\Big|_{(\eta,\tau_\eta)}\,D\tau_\eta[Y^0]\Big),
  \end{equation}
  where $Y$ is the linearized evolution of $Y^0$, namely
  \begin{align}
    Y(\eta,\tau_\eta) &:= D\mc E^\eta(y^0)\big[Y^0\big](\tau_\eta)
    \nonumber \\
      &= \mc S_{\tau_0(\eta,\tau_\eta)}(y^0)\big[Y^0\big]
      + D\tau_0(\eta,\tau_\eta)(y^0)\big[Y^0\big] \nonumber\\
      &= \mc S_{\tau_0}(y^0)\big[Y^0\big] + D\tau_0(y^0)\big[Y^0\big].
        \label{DY}
  \end{align}
  Here $\tau_\eta$ has the reference point $(x,t)$, and we have used
  the important group property \eqref{group}, so that
  \[
    \tau_0\big(\eta,\tau_\eta(x,t)\big) = \tau_0(x,t),
  \]
  so $\tau_0$ has simplified reference point $(x,t)$.  We use this
  again in \eqref{slin}, to get the first term
  \[
    \sg'\big(y(\eta,\tau_\eta)\big) =
    \sg'\big(y^0\big(\tau_0(\eta,\tau_\eta(x,t))\big)\big)
    = \sg'\big(y^0\big(\tau_0(x,t))\big);
  \]
  although we can similarly simplify the higher order
  $\frac{\del y}{\del t}$ term, this derivative is not preserved along
  characteristics, and we have no need to do so.

  It now follows from \eqref{slin} that the linearization of the
  characteristic field satisfies the linear integral equation
  \begin{equation}
    \label{Dtau}
    D\tau_\xi[Y^0] = \sg'\big(y^0(\tau_0)\big)\,\int_{x}^\xi
    \IR+\Big(Y(\eta,\tau_\eta)
    + \frac{\del y}{\del t}\Big|_{(\eta,\tau_\eta)}\,D\tau_\eta[Y^0]\Big)
    \;d\eta,
  \end{equation}
  with $Y$ given by \eqref{DY}, and since $\phi(t)=\tau_0(\t,t)$,
  substituting in gives \eqref{Dphi}.

  We now wish to differentiate $D\mc E^\t$ a second time.  We fix
  $Y_2^0$ and perturb $y^0$ by $Y^0_1$, and again let $\phi+\Phi$
  denote the perturbed shift.  Then by \eqref{DE} we have
  \[
    \begin{aligned}
      D\mc E^\t(y^0&+Y^0_1)\big[Y^0_2\big] =
      \mc S_{\phi+\Phi}\big[Y^0_2\big]\\
      &{}+ \mc S_{\phi+\Phi}\Big[\frac d{dt}\big(y^0
      +Y^0_1\big)\Big]\, D(\phi+\Phi)(y^0+Y^0_1)\big[Y^0_2\big],
    \end{aligned}
   \]
   and subtracting \eqref{DE} (evaluated at $Y^0_2$) yields, after
   rearranging,
  \[
    \begin{aligned}
      D\mc E^\t(y^0&+Y^0_1)\big[Y^0_2\big] =
      \mc S_{\phi+\Phi}\big[Y^0_2\big]-\mc S_\phi\big[Y^0_2\big]\\
      &{}+ \mc S_{\phi+\Phi}\Big[\frac d{dt}Y^0_1\Big]\,
      D(\phi+\Phi)(y^0+Y^0_1)\big[Y^0_2\big]\\
      &{}+ \Big(\mc S_{\phi+\Phi}-\mc S_\phi\Big)
      \Big[\frac d{dt}y^0\Big]\,
      D(\phi+\Phi)(y^0+Y^0_1)\big[Y^0_2\big]\\
      &{}+\mc S_\phi\Big[\frac d{dt}y^0\Big]\,\Big(D(\phi+\Phi)
      (y^0+Y^0_1)\big[Y^0_2\big]-D\phi(y^0)\big[Y^0_2\big]\Big).
    \end{aligned}
  \]
  Now again taking the limit of small $\|Y^0_1\|$, we retain only
  linear terms (or terms bilinear in $[Y^0_1,Y^0_2]$) to get
  \[
    \begin{aligned}
      D^2\mc E^\t(y^0)\big[Y^0_1,Y^0_2\big] =&
      \mc S_\phi\Big[\frac d{dt}Y^0_2\Big]\,
      D\phi(y^0)\big[Y^0_1\big] +
      \mc S_\phi\Big[\frac d{dt}Y^0_1\Big]\,
      D\phi(y^0)\big[Y^0_2\big] \\
      &+ \mc S_\phi\Big[\frac{d^2}{dt^2}y^0\Big]\,
      D\phi(y^0)\big[Y^0_1\big]\,D\phi(y^0)\big[Y^0_2\big]\\
      &{}+ \mc S_\phi\Big[\frac{d}{dt}y^0\Big]\,
      D^2\phi(y^0)\big[Y^0_1,Y^0_2\big],
    \end{aligned}
  \]
  which is \eqref{D2E}.
\end{proof}

To complete the solution of the bifurcation equation, we now isolate
the terms appearing in equations \eqref{dgdz1} and \eqref{dgdz}.  We
use a slight generalization in order to apply it later on in a more
general context.  The first argument 1 in the bilinear operator below
corresponds to differentiation of the 0-mode, that is
$\frac\partial{\partial z}$.

\begin{corollary}
  \label{cor:D2N}
  Suppose, as in \eqref{Ndef}, that the nonlinear operator $\mc N^\t$
  is defined by
  \begin{equation}
    \label{Nt}
    \mc N^\t(y) := {\mc L^\t}^{-1}\,\mc E^\t(y), \com{where}
    \mc L^\t := D\mc E^\t(0) = \mc S^\t,
  \end{equation}
  and let $\mc B$ be any fixed constant invertible linear operator which
  preserves $k$-modes.  Then $\mc B^{-1}\,\mc N^\t\,\mc B$ is twice
  Frechet differentiable at $0$, with
  \begin{equation}
    \label{DN}
    \begin{aligned}
      D\big(\mc B^{-1}\,\mc N^\t\,\mc B\big)(0) 
      &= \mc I, \com{and}\\
      D^2\big(\mc B^{-1}\,\mc N^\t\,\mc B\big)(0)[1,Y^0]
      &= \nu\,\t\,\mc B1\,\frac{d}{dt}Y^0.
    \end{aligned}
  \end{equation}
\end{corollary}

Note that in all contexts in which we apply this corollary, including
\eqref{Ndef}, where $\mc B = \mc J\,\mc L^{\ol\t}$, the operator $\mc
B$ is a combination of linear evolutions $\mc L$ and jumps $\mc J$,
for which we have $\mc B1=1$.

\begin{proof}
  Setting $y^0=0$ in \eqref{charxi}, \eqref{shift}, we get
  \[
    \tau_\xi(x,t) = t + \xi - x \com{and}
    \phi(t) = t-\t,
  \]
  and using \eqref{DE}, we get
  \[
    D\mc E^\t(0) = \mc S_\phi = \mc S^\t = \mc L^\t,
    \com{so also} D\mc N^\t(0) = \mc I,
  \]
  and the first part of \eqref{DN} follows.  Next, since $\mc B$ is
  fixed, taking $y^0=0$, $Y^0_1=1$ and $Y^0_2=Y^0$ in \eqref{D2E},
  only the first term persists, and we get
  \[
    \begin{aligned}
      D^2\big(\mc B^{-1}\,\mc N^\t\,\mc B\big)(0)\big[1,Y^0\big]
      &= \mc B^{-1}\,{\mc L^\t}^{-1}\,\Big\{
      \mc S_\phi\Big[\frac d{dt}\mc B\,Y^0\big]\cdot
      D\phi(\mc B0)[\mc B1]\Big\}\\
      &= \frac{d}{dt}Y^0\cdot D\phi(0)[\mc B1],
    \end{aligned}
  \]
  since $\mc B$ and $\mc L^\t$ preserve 0-modes, and $\mc B\square$ is
  the result of operation by $\mc B$, and $\mc L^\t=\mc S_\phi$.
  Finally, setting $y^0=0$ and $Y^0_1=1$ in \eqref{Dphi}, we get
  \[
    D\phi(0)[\mc B1] = \sg'(0)\int_\t^0\IR+\mc B 1\;d\xi = \nu\,\t\,\mc B 1,
  \]
  which yields \eqref{DN}.
\end{proof}

We can now state our first theorem on the simplest space and
time periodic solutions of the compressible Euler equations.

\begin{theorem}
  \label{thm:mainex}
  Let $(\ul\t,\ol\t)\in\B R^2$ be given such that the divisors $\d_k$
  given by \eqref{dk}, with $J$ defined by \eqref{Jth}, are nonzero
  for all $k\ge 2$.  Then there is a number $\ol\a_1>0$ and $C^2$
  function $z = z(\a)$ satisfying $z(0)=0$, such that for each
  $\a\in(-\ol\a_1,\ol\a_1)$, the even function
  \[
    y^0(t) = \a\,\c(t) + z(\a) + W\big(\a,z(\a)\big)
    \in \mc H_1\oplus\mc H_2,
  \]
  is a solution of the equation
  \[
    \mc F(y^0) = \IR-\,\mc L_0\,\ulc N\,\olc N\,y^0 = 0.
  \]
  By reflection symmetry this defines a minimal tile whch generates a
  classical space and time periodic solution of the compressible
  Euler equations with stationary square wave entropy profile.
\end{theorem}

Thus we have a one-parameter family of solutions, which are
perturbations of a cosine 1-mode, parameterized by the amplitude $\a$
of the 1-mode of the solution.  Our identification and use of symmetry
has turned the problem into a regular bifurcation problem which has
been handled by a modified Liapunov-Schmidt reduction, which has been
explicitly carried out here.  This has led both to a proof that is
easier and a result that is more robust than that which we would get
by applying the Nash-Moser method, which would require an expunging of
resonant parameter values to get diophantine conditions, rather than
the simpler irrationality condition, cf.~\cite{TYdiff2}.

\begin{proof}
  We use Lemma~\ref{lem:himodes} to find the function $W$ which solves
  the auxiliary equation, and in order to complete the proof we must
  solve the bifurcation equation \eqref{bif}.  To do so it suffices to
  solve $g(\a,z)=0$, where $g$ is given by \eqref{gdef}.  Since
  $g(0,0)=0$, and the partial derivative $g_z(0,0)$ is given by
  \eqref{dgdz}, we get a function $z(\a)$ from the implicit function
  theorem if we can show that $g_z(0,0)\ne 0$.  Using \eqref{DN} with
  $Y^0=\c(t)$ and $\mc B1=1$ in \eqref{dgdz}, we get
  \begin{equation}
    \frac{\del g}{\del z}\Big|_{(0,0)} =
    - \nu\,\big(\ul\t+\ol\t\big)\,\big\langle \s(t),
    \mc L_0\,\s(t)\big\rangle \ne 0,
    \label{dgdzss}
    \end{equation}
  because $\big\langle \s(t),\mc L_0\,\c(t)\big\rangle=0$ and
  $\mc L_0$ is invertible on 1-modes, which verifies condition
  \eqref{dgdz0} and completes the Liapunov-Schmidt argument.  In fact,
  using Lemma~\ref{lem:Tk} it is easy to calculate
  \[
    \begin{aligned}
    \big\langle \s(t),\mc L_0\,\s(t)\big\rangle &=
    \(-1&0\)R(\ul\t)\,M(J)\,R(\ol\t)\(0\\1\)\\
    &= \c\ul\t\,\s\ol\t + J\,\s\ul\t\,\c\ol\t = \frac{\c\ul\t}{\s\ol\t}\ne0,
    \end{aligned}
  \]
  by our choice of $(\ul\t,\ol\t)$, and where we have used
  \eqref{Jth}.  Thus the function $z(\a)$ is determined and the proof
  is complete.
\end{proof}

Theorem~\ref{thm:mainex} provides the first proof of the existence of
a space and time periodic solution of the compressible Euler
equations exhibiting sustained nonlinear interactions.  This completes
the author's initial program proposed in \cite{TYperStr,TYperLin}.
Indeed, to the author's knowledge, this also represents the first
global existence theorem for a non-monotone solution of the $3\times3$
compressible Euler equations having large (spatial) total variation.

\section{Higher Mode Periodic Solutions}
\label{sec:Tvar}

We now consider the divisors $\d_k(\ol\t,\ul\t,J;T)$ as a function of
time period $T$, while fixing the entropy field parameters $\ol\t$,
$\ul\t$, and $J$.  The divisor $\d_k$ is again given by \eqref{dk},
but the functional dependence is more complicated because the variable
$T$ appears in each evolution component $R(k\piot \t)$.  

We thus consider the
divisors
\begin{equation}
  \label{dkT}
  \d_k(\Theta;T) = \(0&1\)\,P^{-k}\,R(k\,\piot\,\ul\t)\,
  M(J)\,R(k\,\piot\,\ol\t)\(1\\0\),
\end{equation}
in which we regard $\Theta:=(\ul\t, J, \ol\t)$ as fixed, and look for
\emph{all} values of $T>0$ for which $\d_k=0$.  We begin by recalling
the redundancy in our basis vectors \eqref{Tk}: that is,
\[
  \mc T_k(T) = \( \c\big(k\piot t\big) & \s\big(k\piot t\big) \)
  = \mc T_{jk}(j T),
\]
which immediately yields (since $P^2=-I$)
\[
  \d_k\big(\Theta;T\big) = -\d_{k-2}\Big(\Theta;\frac{k-2}k\,T\Big),
\]
and which in turn implies
\begin{equation}
  \label{d2j}
  \begin{aligned}
    \d_{2j-1}(\Theta;T\big) &= (-1)^{j-1}\,
    \d_1\Big(\Theta;\frac1{2j-1}\,T\Big), \com{and}\\
    \d_{2j}(\Theta;T\big) &= (-1)^j\,
    \d_2\Big(\Theta;\frac1{j}\,T\Big),
  \end{aligned}
\end{equation}
for odd and even indices, respectively.  It thus suffices to find all
solutions to $\d_1=0$ and $\d_2=0$, and scaling then yields all
solutions of $\d_k=0$.

Because each $2\times2$ matrix in \eqref{dkT} is invertible, the
intermediate vectors in the composition \eqref{dkT} never vanish, and
it suffices to consider the \emph{angle} made by these $2\times2$
vectors: these are
\[
  0 \to k\,\piot\,\ol\t \to \g \to \g + k\,\piot\,\ul\t,
\]
where $\g$ is the angle obtained after applying the jump corresponding
to $M(J)$ to the angle $k\,\piot\,\ol\t$.  We simplify by making the
substitutions
\begin{equation}
  \label{thsub}
  \t := \frac{\ol\t}{\ul\t+\ol\t}, \quad
  \w := k\,\piot\,(\ul\t+\ol\t),
\end{equation}
which yields
\[
  k\,\piot\,\ol\t = \w\,\t \com{and}
  k\,\piot\,\ul\t = \w\,(1-\t),
\]
and defining the function
\[
  h:\B R_+\times\B R\to\B R, \com{by}
  h(J,x) := \arctan\big(J\,\tan x\big),
\]
chosen so that $h$ is smooth.  The angles then vary from
\[
  0 \to \w\,\t \to \g \to \g + \w\,(1-\t),
  \com{with} \g := h(J,\w\,\t).
\]

To be precise, for $J>0$, we define
\begin{equation}
  \label{hdef}
  h(J,x) :=
  \begin{cases}
    \text{Arctan}(J\,\tan x\big) + k\,\pi, &-\pi/2<x-k\,\pi<\pi/2,\\
    x &x = k\,\pi \pm\pi/2,
  \end{cases}
\end{equation}
where $\text{Arctan}(\square)\in(-\pi/2,\pi/2)$ is the principal
branch.  Geometrically, $h$ gives the angle that results from a vector
with angle $x$ being acted upon by $M(J)$; in particular, $h(J,x)$ is
always in the same quadrant as $x$, and $h(J,x)=x$ for $x$ on the
coordinate axes.  We can now give the angle associated with rotations
$x_1$ and $x_2$ separated by jump $J$, namely
\begin{equation}
  \label{angle}
  R(x_2)\,M(J)\,R(x_1)\(1\\0\) \com{is represented by}
  x_2 + h(J,x_1).
\end{equation}
This is illustrated in Figure~\ref{fig:freqs} below for a more general
piecewise constant entropy profile.

Using \eqref{thsub}, and referring to \eqref{dkT}, we now see that for
any $\T$, $T$ provides a solution of $\d_k(\T;T)=0$ if the angle
$\g+\w\,(1-\t)$ points along the $k$-th coordinate axis, counting
anti-clockwise: that is, if
\[
  \w\,(1-\t) + \g = k\,\frac\pi2, \com{with} \g = h(J,\w\,\t).
\]
More precisely, we define the \emph{$k$-th base frequency}
$\w^{(k)}=\w^{(k)}(\T)$ to be the solution of the equation
\begin{equation}
  \label{weq1}
  \w^{(k)}\,(1-\t) + h\big(J,\w^{(k)}\,\t\big) = k\,\frac\pi2.
\end{equation}
Note that the parity of $k$ determines which axis is met; this is
consistent with \eqref{d2j}.

\begin{lemma}
  \label{lem:Tj}
  For any fixed positive step entropy parameter
  $\Theta = (\ul\t,J,\ol\t)$, and for any $k\ge 1$, there is a unique
  base reference period $T^{(k)}$, such that the $k$-mode
  $\mc T_k\big(T^{(k)}\big)$ has vanishing multiplier,
  \[
    \d_k\big(\Theta;T^{(k)}\big)=0.
  \]
  If in addition this $k$-mode is nonresonant, that is
  \begin{equation}
    \label{djnr}
    \d_j\big(\T;T^{(k)}\big) \ne 0
    \com{for all} j\ge 1,\quad j\ne k,
  \end{equation}
  then the corresponding $k$-mode solution of the linearized equation
  perturbs to a pure tone solution of the nonlinear Euler equations
  with reference period $T=T^{(k)}$.
\end{lemma}

\begin{proof}
  We begin with the observation that the function $h$ given by
  \eqref{hdef} is a smooth function, with
  \begin{equation}
    \label{hder}
    \frac{\del h}{\del x} =
    \frac{J\,(1+\tf^2(x))}{1 + J^2\,\tf^2(x)} > 0, \qquad
    \frac{\del h}{\del J} = \frac{\tf(x)}{1 + J^2\,\tf^2(x)},
  \end{equation}
  where $\tf:=\tan$.  In particular $h$ is strictly increasing,
  one-to-one and onto, and so smooth invertible, as a function of $x$,
  and invertible as a function of $J$ as long as
  $\tf(x)\ne0,\ \pm\infty$, that is as long as $x\ne j\frac\pi2$.
  Considering \eqref{weq1}, we define
  \[
    f\big(\w,J,\t\big) := \w\,(1-\t) + h(J,\w\,\t),
  \]
  so $\w^{(k)}$ solves $f(\w,J,\t) = k\,\frac\pi2$.  Since $0<\t<1$,
  $f$ is a sum of two smooth strictly monotone increasing and
  surjective functions, so there is a unique positive solution
  $\w^{(k)}$ of \eqref{weq1}.  Using \eqref{thsub}, it follows that
  our desired reference period is
  \begin{equation}
    \label{Tkw}
    T^{(k)} := k\,\frac{2\pi}{\w^{(k)}}\,(\ul\t+\ol\t).
  \end{equation}
  The proof that nonresonant modes perturb follows in the same manner as
  Theorem~\ref{thm:mainex} above.  We omit the details because we
  treat a more general case in Theorem~\ref{thm:pwent} below.
\end{proof}

We now examine the nonresonance condition \eqref{djnr} in more detail.
The reference period $T^{(k)}$ is defined by \eqref{Tkw},
\eqref{weq1}, and satisfies $\d_k\big(\T;T^{(k)}\big)=0$.  The $j$-th
mode is resonant with the given $k$-mode, if also
$\d_j\big(\T;T^{(k)}\big)=0$.  According to \eqref{dkT},
\eqref{thsub}, \eqref{angle}, it follows that the angle associated
with the $j$-mode, namely
\[
  j\,\frac{2\pi}{T^{(k)}}\,\ol\t = \frac jk\,\w^{(k)}\,\t, \qquad
  j\,\frac{2\pi}{T^{(k)}}\,\ul\t = \frac jk\,\w^{(k)}\,(1-\t),
\]
must also be a solution of \eqref{weq1}, say
\begin{equation}
  \label{rescon}
  q\,\w^{(k)}=\w^{(p)}, \com{with} q := \frac jk\in\B Q_+.
\end{equation}
This states that the $j$-th and $k$-th modes (with respect to
reference period $T^{(k)}$) are in resonance, if and only if the
distinct roots $\w^{(k)}$ and $\w^{(p)}$ are rationally related by
\eqref{rescon}.  Finally, note that because there are countably many
solutions $\{\w^{(k)}\}$ of \eqref{weq1}, rational combinations of
these span a ``small'' set over the whole solution space, and so for a
generic $\Theta$, all modes will be nonresonant and so perturb to a
finite amplitude nonlinear solution.

\begin{theorem}
  \label{thm:T}
  Let the single step function entropy profile be parameterized by
  $\T=(\ul\t,J,\ol\t)\in\B R^3_+$.  Generically, such entropy profiles
  are \emph{totally nonresonant}, in that \emph{all} $k$-modes with
  associated reference period $T^{(k)}$ are nonresonant, and so
  perturb to finite amplitude smooth solutions of \eqref{F=0}, which
  in turn generate pure tone periodic solutions of the nonlinear
  system \eqref{lagr}.  That is, there is a set
  $\wh{\mc Z}\subset\B R^3_+$, which is both measure zero and meagre,
  such that any $\T\notin\wh{\mc Z}$ is totally nonresonant.
\end{theorem}

\begin{proof}
  We work with the reduced parameters $(\w,J,\t)$, given by
  \eqref{thsub}, for which redundancies due to scale-invariance of the
  system \eqref{lagr} have been removed.  The condition that the
  $j$-th mode be resonant with the $k$-th mode, with respect to
  reference period $T^{(k)}$, is \eqref{rescon}.  Thus for fixed
  integers $j$, $k$ and $p$, we define the set
  \[
    \mc Z_{k,j,p} := \Big\{(J,\t)\;\Big|\;
    \w^{(p)} = q\,\w^{(k)}\Big\}, \quad q := \frac jk,
  \]
  where $\w^{(k)}=\w^{(k)}(J,\t)$ is regarded as a known function
  defined by \eqref{weq1}.

  We show that $\mc Z_{k,j,p}$ is a smooth one-dimensional manifold in
  $\B R^2_+$.  We first make our description of $\w^{(k)}$ more
  explicit, by noting that we can write $\w=\w^{(k)}$ in \eqref{weq1}
  as
  \begin{equation}
    \label{wf}
    \w = x - h(J,x) + k\,\frac\pi2, \qquad \w\,\t = x,
  \end{equation}
  and treating $x$ as a variable with $(J,\t)$ fixed.  Thus $(x,\w)$
  is the point of intersection of two explicit, regular curves.
  The slopes of these curves are
  \[
    1-h_x<1 \com{and} \frac1\t>1,
  \]
  respectively, so there is a unique solution, as expected.

  From \eqref{wf} we write
  \[
    \begin{aligned}
      \w^{(k)} = x^{(k)} - h(J,x^{(k)}) + k\,\frac\pi2,
      &\qquad \w^{(k)}\,\t = x^{(k)},\\
      \w^{(p)} = x^{(p)} - h(J,x^{(p)}) + p\,\frac\pi2,
      &\qquad \w^{(p)}\,\t = x^{(p)},
    \end{aligned}
  \]
  and so $\mc Z_{k,j,p}$ becomes the set
  \[
    \w^{(p)} = q\,\w^{(k)}, \qquad
    x^{(p)} = q\,x^{(k)},
  \]
  and
  \[
    g(J,x) := h(J,q\,x) - q\,h(J,x) = (p-j)\frac\pi2,
  \]
  where we have written $x:=x^{(k)}$, $q\,x = x^{(p)}$.  To show that
  $\mc Z_{k,j,p}$ is a submanifold, it suffices to show that $\nabla
  g\ne0$ on $\big\{g = (p-j)\frac\pi2\big\}$.

  From \eqref{hder}, we calculate
  \[
    \begin{aligned}
      \frac{\del g}{\del x}
      &= \frac{J\,(1+\tf^2(qx))}{1 + J^2\,\tf^2(qx)}\,q
        - q\,\frac{J\,(1+\tf^2(x))}{1 + J^2\,\tf^2(x)}\\
      &= \frac{q\,J\,(J^2-1)\,\big(\tf^2(x)-\tf^2(qx)\big)}
        {\big(1 + J^2\,\tf^2(qx)\big)\,\big(1 + J^2\,\tf^2(x)\big)},
    \end{aligned}
  \]
  and so
  \[
    \frac{\del g}{\del x} = 0
    \com{if and only if}
    \tf(qx) = \pm\tf(x),
  \]
  unless $J=1$, which is the degenerate isentropic case.  Similarly,
  \[
    \begin{aligned}
      \frac{\del g}{\del J}
      &= \frac{\tf(qx)}{1 + J^2\,\tf^2(qx)}
        - q\,\frac{\tf(x)}{1 + J^2\,\tf^2(x)}\\
      &= \frac{\tf(qx)-q\tf(x) + J^2\,\tf(qx)\,\tf(x)\,
        \big(\tf(x)-q\,\tf(qx)\big)}
        {\big(1 + J^2\,\tf^2(qx)\big)\,\big(1 + J^2\,\tf^2(x)\big)},
    \end{aligned}
  \]
  and so if $\frac{\del g}{\del J}=0$ \emph{and}
  $\frac{\del g}{\del x} = 0$, we must have $\tf(qx)=\pm\tf(x)$ and,
  after simplifying,
  \[
    \tf(x)\,(\pm1-q)\,\big(1+J^2\tf^2(x)\big) = 0,
    \com{so that}
    \tf(qx)=\tf(x)=0.
  \]
  This condition (or $J=1$) then implies that $g=0\ne(p-j)\frac\pi2$, so
  that
  \[
    \nabla g(J,\t)\ne0 \com{for all}
    (J,\t)\in\mc Z_{k,j,p}.
  \]
  
  The implicit function theorem now implies that $\mc Z_{k,j,p}$ is a
  smooth submanifold of codimension one.  This implies that
  $\mc Z_{k,j,p}$ is both a measure zero and nowhere dense set in
  $\B R^2_+$.  Since the integers $k$, $j$ and $p$ were arbitrary, it
  follows that the resonant set
  \[
    \mc Z := \bigcup_{k,p,j}\mc Z_{k,p,j}
  \]
  being a countable union, is also measure zero and meagre.  Thus the
  complement $\mc Z^c$, which is the set of \emph{fully nonresonant}
  parameters, is generic in both the measure and Baire category
  senses.  Finally, we use scale invariance of the system and
  \eqref{thsub} to transform from $(J,\t)$ back to
  $\T = (\ul\t,J,\ol\t)$ coordinates, which preserves the smallness
  condition and thus completes the proof.
\end{proof}

\section{Piecewise constant entropy}
\label{sec:pw}

Having understood the evolution and its linearization for two constant
enropy levels separated by a single jump, we now extend the argument
to a general piecewise constant entropy profile.  We again use the
non-dimensional system \eqref{nondim}, and consider perturbations of
the equilibrium state $(w,\st w)=(0,0)$, which corresponds to the
quiet state stationary solution $(p,u)=(p_0,0)$.  Let the entropy
profile $s(x)$ be a piecewise constant function on $[0,\ell]$,
continuous at the endpoints $x=0$ and $x=\ell$.  To be precise, we
introduce $n+1$ \emph{entropy widths} $\t_m>0$, $m=0,\dots,n$,
separated by $n$ \emph{entropy jumps} $J_m=e^{-[s]_m/2c_p}\ne 1$, as
above.  Assume the same symmetries as before, namely $w$, $\st w$
even/odd in $t$ and time periodic with period $T$, together with the
spatial/material reflection symmetries \eqref{olsymm} at $x=0$ and
\eqref{ulsymm} at $x=\ell$.  Then our generalized setup becomes: $x=0$
is the center of the $\t_0$ entropy level, and boundary condition
\eqref{olsymm} is imposed there, and that initial data is evolved
through $\t_0$ by the nonlinear evolution \eqref{evol}.  Then
inductively, at $x=\sum_{j=0}^{m-1}\t_j$, we place entropy jump
$\mc J_m$ of size $J_m$, given by \eqref{jump}, and evolve this from
$x$ to $x+\t_m$ by \eqref{evol}.  After the final jump $J_n$, we
evolve the amount $\t_n$ to $x = \ell = \sum_{j=0}^n\t_j$, where we
impose the shifted boundary condition \eqref{ulsymm} which generates a
minimal tile.

As above, after non-dimensionalization we can leverage the even/odd
symmetry to reduce to the nonlocal and nonlinear scalar
evolution~\eqref{yevol} together with jumps \eqref{Jop}, and our
boundary conditions again become \eqref{ols} and \eqref{uls},
respectively.  Since the entropy field is stationary, the
non-dimensionalized jumps $\mc J_m$ are the same for the nonlinear
evolution.  Thus our nonlinear equation for $y^0$, given arbitrary
piecewise constant entropy profile, which is a direct generalization
of \eqref{F=0}, becomes
\begin{equation}
  \label{F=0pw}
  \begin{gathered}
    \mc F(y^0) = 0,  \quad  y^0\text{ even, where} \\
    \mc F(y^0) :=
    \IR-\,\mc S^{-T/4}\,\mc E^{\t_n}\,\mc J_n\,\mc E^{\t_{n-1}}\,
    \dots\,\mc J_1\,\mc E^{\t_0}\,y^0.
  \end{gathered}
\end{equation}
Our goal is to show that generically, $k$-mode solutions of the
corresponding linearized equation are nonresonant, and perturb to
finite amplitude pure tone solutions of the nonlinear system.

Because the structure is very similar to that of
Sections~\ref{sec:lin} and~\ref{sec:onejump}, we can solve the
equation \eqref{F=0pw} in much the same way.  In particular,
Lemmas~\ref{lem:Tk} and \ref{lem:D2E} and Corollary~\ref{cor:D2N}
continue to hold unchanged, and the only essential difference is in
the base linear operator $\mc L_0$, which, however, has a similar
structure.

\begin{lemma}
  \label{lem:pws}
  For a given reference period $T$, and parameters
  \[
    J := \big(J_1,\dots,J_n\big)\in\B R^n_+ \com{and}
    \Theta := \big(\t_0,\dots,\t_n\big)\in\B R^{n+1}_+,
  \]
  let $\mc F$ be given by \eqref{F=0pw}.  Then $\mc F(0)=0$, and the
  Frechet derivative of $\mc F$ at $0$ is
  \[
    D\mc F(0) = \IR-\,\mc L_0, \com{that is}
    D\mc F(0)\big[Y^0\big] = \IR-\,\mc L_0\big[Y^0\big],
  \]
  where $\mc L_0$ is the invertible linear operator
  \begin{equation}
    \label{Lpw}
    \mc L_0 := \mc S^{-T/4}\,\mc L^{\t_n}\,\mc J_n\,\mc L^{\t_{n-1}}\,
    \dots\,\mc J_1\,\mc L^{\t_0},
  \end{equation}
  and the input $Y^0$ is an even function of $t$ with period $T$.
  The nonlinear operator \eqref{F=0pw} can again be factored,
  \begin{equation}
    \label{facpw}
    \mc F = \IR-\,\mc L_0\,\mc N, \com{where}
    \mc N := \mc N_n\,\mc N_{n-1}\cdots\mc N_0,
  \end{equation}
  and each $\mc N_m$ is a nonlinear operator given by
  \[
    \mc N_m = \mc B_m^{-1}\,{\mc L^{\t_m}}^{-1}\,
    \mc E^{\t_m}\,\mc B_m,
  \]
  where $\mc B_m$ is a bounded linear operator that respects
  $k$-modes, with $\mc B1=1$.

  The operators $\mc L_0$ and $D\mc F(0)$ respect $k$-modes, with
  \[
    D\mc F(0)\Big[\c\big(k\piot t\big)\Big] =
    \IR-\,\mc L_0\Big[\c\big(k\piot t\big)\Big]
    = \d_k(J,\Theta;T)\,\s\big(k\piot t\big),
  \]
  where the divisors are the scalars
  \begin{equation}
    \label{dkpw}
    \begin{aligned}
      \d_k(J,\Theta;T) := \(0&1\)&P^{-k}\,
      R(k\piot\t_n)\,M(J_n)\,\dots\\&\dots
      R(k\piot\t_1)\,M(J_1)\,R(k\piot\t_0)\(1\\0\),
    \end{aligned}
  \end{equation}
  and the $2\times2$ matrices $P$, $R(\cdot)$ and $M(\cdot)$ are
  defined in Lemma~\ref{lem:Tk}.
\end{lemma}

\begin{proof}
  It is clear that $\mc F(0)=0$, and \eqref{Lpw} and \eqref{dkpw}
  follow directly from \eqref{L} and Lemma~\ref{lem:Tk}, as in
  Corollary~\ref{cor:DF0}.  We show the factorization \eqref{facpw} by
  induction: to begin, set
  \[
    \mc B_0 := \mc I, \quad
    \mc N_0 := {\mc L^{\t_0}}^{-1}\,\mc E^{\t_0}, \com{so that}
    \mc E^{\t_0} = \mc L^{\t_0}\,\mc N_0.
  \]
  Next, assume inductively that
  \[
        \mc N_j = \mc B_j^{-1}\,{\mc L^{\t_j}}^{-1}\,
    \mc E^{\t_j}\,\mc B_j \com{for} j < m,
  \]
  which implies
  \begin{equation}
    \label{Nprod}
    \mc N_j\dots\mc N_0 = \mc B_j^{-1}\,{\mc L^{\t_j}}^{-1}\,
    \mc E^{\t_j}\,\mc B_j\,\mc B_{j-1}^{-1}\,{\mc L^{\t_{j-1}}}^{-1}\,
    \mc E^{\t_{j-1}}\,\mc B_{j-1}\dots{\mc L^{\t_0}}^{-1}\,\mc E^{\t_0},
  \end{equation}
  and comparing this to $\mc F$, we require each $\mc B_j$ to satisfy
  \[
    \mc B_j\,\mc B_{j-1}^{-1}\,{\mc L^{\t_{j-1}}}^{-1} = \mc J_j.
  \]
  Thus, at the inductive step, we choose
  \[
    \mc B_m := \mc J_m\,\mc L^{\t_{m-1}}\,\mc B_{m-1} \com{and}
     \mc N_m = \mc B_m^{-1}\,{\mc L^{\t_m}}^{-1}\,
    \mc E^{\t_m}\,\mc B_m.
  \]
  This implies that for each $m\le n$, we have
  \[
    \mc B_m = \mc J_m\,\mc L^{\t_{m-1}}\,\mc J_{m-1}\dots
    \mc L^{\t_1}\,\mc J_1\,\mc L^{\t_0},
  \]
  and \eqref{Nprod} then yields
  \[
    \begin{aligned}    
    \mc S^{-T/4}\,\mc E^{\t_n}\,\mc J_n\,\mc E^{\t_{n-1}}\,
    \dots\,\mc J_1\,\mc E^{\t_0} &= \mc S^{-T/4}\,\mc L^{\t_n}\,
    \mc B_n\,\mc N_n\,\dots\mc N_0\\
    &= \mc L_0\,\mc N_n\,\dots\mc N_0,
    \end{aligned}
  \]
  and \eqref{facpw} follows, completing the proof.
\end{proof}

It follows that, for fixed reference period $T$, the base linearized
operator $D\mc F(0) = \IR-\,\mc L_0$ has a $k$-mode kernel if and only
if $\d_k(J,\Theta,T)=0$.  As above, our purpose is to perturb this to a
nontrivial solution $y^0$ of $\mc F(y^0)=0$ with finite amplitude
$\a$.  We proceed as in the earlier case.

For fixed $T$, denote the zero set of $\d_k$ by
$\Lambda_k=\Lambda_k(T)$, that is, set
\begin{equation}
  \label{lambda}
  \Lambda_k := \big\{ (J,\T)\in\B R^{2n+1}\;\big|\;
  \d_k\big(J,\T;T\big) = 0 \big\}, \quad k\ge 1,
\end{equation}
and we declare a parameter $(J,\T)\in\Lambda_k$ to be
\emph{nonresonant} if it is in no other zero set $\Lambda_j$,
\begin{equation}
  \label{nonres}
  (J,\T) \in \Lambda_k, \com{but}
  (J,\T) \notin\Lambda_j, \quad j\ge1,\ j\ne k.
\end{equation}
Then, as above, any nonresonant $(J,\T)\in \Lambda_k$ provides a
$k$-mode that perturbs to a solution of the nonlinear problem.

\begin{theorem}
  \label{thm:pwent}
  Suppose that $T$ is fixed, and let $(J,\T)\in\Lambda_k(T)$ be
  nonresonant.  Then the $k$-mode $\a\,\c(k\piot t)$ perturbs to a
  pure tone solution of the nonlinear problem $\mc F(y^0) = 0$.  More
  precisely, there is an $\ol\a_k>0$ and functions
  \[
    W(\a,z) = \sum_{j\ge 1,j\ne k}a_j\,\c(j\piot t) \com{and}
    z = z(\a) \in \B R,
  \]
  defined for $|\a|<\ol\a_k$ and some interval $z\in(-\ol z_k,\ol z_k)$, such
  that $W(0,z)=0$, $z(0)=0$, and
  \[
    \mc F\Big(\a\,\c\big(k\piot t\big) + z(\a) +
    W\big(\a,z(\a)\big)\Big) = 0,
  \]
  and this generates a space and time periodic solution of the
  compressible Euler equations.
\end{theorem}

\begin{proof}
  The proof proceeds exactly as in that of Theorem~\ref{thm:mainex}.
  The Hilbert spaces are defined as in \eqref{Hsplit} and
  \eqref{Hplus}, with the $k$-mode being substituted for the 1-mode.
  The analog of Lemma~\ref{lem:himodes} then follows exactly as above
  to yield the solution $W(\a,z)$ of the auxiliary equation.
  The bifurcation equation \eqref{bif} is replaced by
  \[
    f(\a,z) :=
    \Big\langle \s(k\piot t), \whc F\big(z+\a\,\c(k\piot t)+
    W(\a,z)\big)\Big\rangle = 0,
  \]
  and we again define $g$ by \eqref{gdef}.  The analog of \eqref{dgdz}
  is
  \begin{equation}
    \label{dgdzpw}
    \frac{\del g}{\del z}\Big|_{(0,0)} =
    \Big\langle \s\big(k\piot t\big),
    \sum_{m=0}^n\mc L_0\,D^2\mc N_m(0)\big[1,\c(k\piot t)\big]\Big\rangle,
  \end{equation}
  and Lemmas~\ref{lem:D2E} and Corollary~\ref{cor:D2N} continue to
  hold.  Since
  \[
    D^2\mc N_m(0)\big[1,Y^0\big] =
    \nu\,\t_m\,\frac{d}{dt}Y^0,
  \]
  the analog of \eqref{dgdzss} is
  \[
    \frac{\del g}{\del z}\Big|_{(0,0)} = - \nu\,k\,\piot\,
        \Big\langle \s\big(k\piot t\big),
        \mc L_0\,\s\big(k\piot t\big)\Big\rangle\sum_{m=0}^n\t_m\ne 0,
  \]
  where again the coefficient is nonzero because
  \[
    \Big\langle \s\big(k\piot t\big),
    \mc L_0\,\c\big(k\piot t\big)\Big\rangle = 0,
  \]
  and $\mc L_0$ is invertible on $k$-modes, spanned by
  $\c\big(k\piot t\big)$ and $\s\big(k\piot t\big)$.
\end{proof}

\subsection{Structure of the sets $\Lambda_k(T)$}

In Theorem \ref{thm:pwent}, we showed that nonresonant modes perturb,
but we have not yet shown the existence of such modes.  We now study
the structure of the sets $\Lambda_k(T)$, and show that a small
perturbation of any piecewise constant entropy profile yields the
existence of some $k$ for which $(J,\T)\in\Lambda_k$, that is
$\d_k(J,\Theta;T)=0$ by \eqref{lambda}.  Moreover, we can choose this
perturbed entropy profile to be nonresonant, so that the corresponding
$k$-mode solution indeed perturbs.

\begin{lemma}
  \label{lem:lamk}
  For each $k\ge 1$, and fixed $T>0$, the set $\Lambda_k(T)$ is a
  $C^\infty$ submanifold of $\B R^{2n+1}$ of codimension 1, and for
  each $j\ne k$, the intersection $\Lambda_k\cap\Lambda_j$ is a
  submanifold of $\B R^{2n+1}$ of codimension 2 if it is nonempty.
\end{lemma}

\begin{proof}
  Without loss of generality we can assume $T=2\pi$.  The set
  $\Lambda_k$ is defined in \eqref{lambda} as the zero set of the
  single $C^\infty$ scalar equation $\d_k(J,\Theta)=0$, given by
  \eqref{dkpw}.  Thus to show it is a manifold of codimension 1, it
  suffices to show that the gradient
  \[
    \nabla\d_k \ne 0, \qquad \nabla:=\nabla_{(J,\T)},
  \]
  everywhere on $\Lambda_k$.  To isolate the dependence of
  $\d_k(J,\Theta)$ on the localized triple $(\t_m, J_m, \t_{m-1})$,
  $m\in\{1,\dots,n\}$, note that for each such $m$, \eqref{dkpw} can
  be written as
  \begin{equation}
    \label{dkz}
    \d_k(J,\Theta;T) = \ul{z_m^{(k)}}^T\,M(J_m)\,\ol{z_m^{(k)}},
  \end{equation}
  where the vectors $\ol{z_m^{(k)}}$ and $\ul{z_m^{(k)}}$ are defined
  inductively by
  \[
    \begin{aligned}
      \ol{z_0^{(k)}} &:= \(1\\0\),\\
      \ol{z_{m+1}^{(k)}} &:= R(k\t_m)\,M(J_m)\,\ol{z_m^{(k)}},
      \quad m = 0,\dots,n-1,
    \end{aligned}   
  \]
  and 
  \[
    \begin{aligned}
      \ul{z_n^{(k)}} &:= R(-k\t_n)\,P^{k}\,\(1\\0\),\\
      \ul{z_{m-1}^{(k)}} &:= R(-k\t_{m-1})\,M(J_m)\,\ul{z_m^{(k)}},
      \quad m = n,\dots,1,
    \end{aligned}   
  \]
  respectively.  Each of $\ol{z_m^{(k)}}$ and $\ul{z_m^{(k)}}$ depends
  only on those variables that determine the entropy profile to the
  left or right of the jump $J_m$, respectively.

  For any fixed $m$, we further simplify by noting that since
  rotations and jumps are invertible, the $z_m^{(k)}$'s never vanish,
  so we can write them as
  \begin{equation}
    \label{rvp}
    \ul{z_m^{(k)}} := \ul{r_k}
    \(\c\big(\ul{\vp_k}\big)\\\s\big(\ul{\vp_k}\big) \)
    \com{and}
    \ol{z_m^{(k)}} := \ol{r_k}
    \(\c\big(\ol{\vp_k}\big)\\\s\big(\ol{\vp_k}\big) \),
  \end{equation}
  for some given angles $\ul{\vp_k}$ and $\ol{\vp_k}$, and scalar
  amplitudes $\ul{r_k}$ and $\ol{r_k}$, where we have simplified the
  subscripts for readability; recall that each of these is defined for
  each $m$.

  We calculate the components of the gradient by differentiating
  \eqref{dkz} directly: it is immediate that
  \[
    \frac{\del\d_k}{\del J_m} =
    \(0&1\)\ul{z_m^{(k)}}\,\(0&1\)\ol{z_m^{(k)}},
  \]
  while since $\frac{\del}{\del\t}R(\t) = PR(\t) = R(\t)P$, we
  have
  \[
    \frac{\del\d_k}{\del\t_m} =
    \ul{z_m^{(k)}}^T\,kP\,M(J_m)\,\ol{z_m^{(k)}}, \qquad
    \frac{\del\d_k}{\del\t_{m-1}} =
    \ul{z_m^{(k)}}^T\,M(J_m)\,kP\,\ol{z_m^{(k)}}.
  \]
  If we use \eqref{rvp}, we get the explicit formulas
  \begin{equation}
    \label{graddk}
    \begin{aligned}
      \frac{\del\d_k}{\del J_m} &=
      \ul{r_k}\,\ol{r_k}\,
      \s(\ul{\vp_k})\,\s(\ol{\vp_k}),\\
      \frac{\del\d_k}{\del\t_m} &=
      \ul{r_k}\,\ol{r_k}\,k\,
      \Big(\s(\ul{\vp_k})\,\c(\ol{\vp_k})
      -J_m\,\c(\ul{\vp_k})\,\s(\ol{\vp_k})\Big),\\
      \frac{\del\d_k}{\del\t_{m-1}} &=
      \ul{r_k}\,\ol{r_k}\,k\,
      \Big(J_m\,\s(\ul{\vp_k})\,\c(\ol{\vp_k})
      -\c(\ul{\vp_k})\,\s(\ol{\vp_k})\Big),
    \end{aligned}
  \end{equation}
  for each $m$.

  To show that $\Lambda_k$ is a manifold, we first suppose that
  $\nabla\d_k=0$, so that the right hand sides in
  \eqref{graddk} all vanish at every $m$.  However, if all three
  vanish for just one $m$, the first equation implies either
  $\s(\ul{\vp_k})=0$ or $\s(\ol{\vp_k})=0$, and the second equation
  then implies that both of these are zero.  It then follows that
  $\c(\ul{\vp_k})=\pm1$ and $\c(\ol{\vp_k})=\pm1$, and plugging these
  in to \eqref{dkz}, we must have $\d_k(J,\Theta)\ne0$, so
  $(J,\T)\notin\Lambda_k$.

  We similarly show that for $j\ne k$, if $\Lambda_k\cap\Lambda_j$ is
  nonempty, it is a manifold in $\B R^{2n+1}$ of codimension 2, which
  will imply that it is a submanifold of $\Lambda_k$ of codimension 1.
  This will follow if we show that the gradients of $\d_k$ and $\d_j$
  are independent at each
  \[
    (J,\T)\in\Lambda_k\cap\Lambda_j = \Big\{(J,\T)\;\Big|\;
    \d_k(J,\Theta)=\d_j(J,\Theta)=0\Big\}.
  \]
  We thus suppose that the gradients $\nabla\d_k$ and
  $\nabla\d_j$ are \emph{dependent}, so that
  \[
    \nabla\d_j = c\,\nabla\d_k, \com{and set}
    C = c\,\frac{\ul{r_k}\,\ol{r_k}}{\ul{r_j}\,\ol{r_j}}.
  \]
  By comparing the gradients, it follows from \eqref{graddk} that we
  have
  \[
    \begin{aligned}
      \s(\ul{\vp_j})\,\s(\ol{\vp_j})
      &= C\,\s(\ul{\vp_k})\,\s(\ol{\vp_k}),\\
      \s(\ul{\vp_j})\,\c(\ol{\vp_j})
      -J_m\,\c(\ul{\vp_j})\,\s(\ol{\vp_j})
      &= C\,\frac kj\,\Big(\s(\ul{\vp_k})\,\c(\ol{\vp_k})
      -J_m\,\c(\ul{\vp_k})\,\s(\ol{\vp_k})\Big),\\
      J_m\,\s(\ul{\vp_j})\,\c(\ol{\vp_j})
      -\c(\ul{\vp_j})\,\s(\ol{\vp_j})
      &= C\,\frac kj\,\Big(J_m\,\s(\ul{\vp_k})\,\c(\ol{\vp_k})
      -\c(\ul{\vp_k})\,\s(\ol{\vp_k})\Big).
    \end{aligned}
  \]
  Eliminating $J_m$, the last two equations simplify, so we have
  \[
    \begin{aligned}
      \s(\ul{\vp_j})\,\s(\ol{\vp_j})
      &= C\,\s(\ul{\vp_k})\,\s(\ol{\vp_k}),\\
      \s(\ul{\vp_j})\,\c(\ol{\vp_j})
      &= C\,\frac kj\,\s(\ul{\vp_k})\,\c(\ol{\vp_k}),\\
      \c(\ul{\vp_j})\,\s(\ol{\vp_j})
      &= C\,\frac kj\,\c(\ul{\vp_k})\,\s(\ol{\vp_k}),
    \end{aligned}
  \]
  and multiplying the last two equations and dividing by the first
  also yields
  \[
    \c(\ul{\vp_j})\,\c(\ol{\vp_j})
    = C\,\frac{k^2}{j^2}\,\c(\ul{\vp_k})\,\c(\ol{\vp_k}),
  \]
  unless $\s(\ul{\vp_k})\,\s(\ol{\vp_k})=0$.  It now follows from
  \eqref{dkz} that $\d_k\ne \d_j$ unless possibly
  $\s(\ul{\vp_k})\,\s(\ol{\vp_k})=0$; however, as above, in this case
  neither $\d_k$ or $\d_j$ vanishes, and the proof is complete.
\end{proof}

The piecewise constant entropy profile can be viewed as being fully
parameterized by $(J,\T) \in \B R^{2n+1}_+$, although this degenerates
if some $J_m=1$.  The following lemma follows directly from this
parameterization of the piecewise constant entropy profile, denoted
$s(J,\T)$, recalling that the jump is given by $J = e^{-[s]/2c_p}$ in
\eqref{jump}.

\begin{lemma}
  \label{lem:dkpw}
  The parameter change
  \[
    \begin{gathered}
      \T \to \wt\T(\eta,m) \com{given by}\\
      \t_m\to \t_m+\eta,\quad \t_{m-1}\to \t_{m-1}-\eta,
    \end{gathered}
  \]
  has the effect of shifting the position of the jump $J_m$ to the
  right by an amount $\eta$, while for $m<n$, the change
  \[
    \begin{gathered}
      J \to \wt J(h,m) \com{given by}\\
      J_{m+1}\to J_{m+1}\,e^{h},\quad J_m\to J_m\,e^{-h},
    \end{gathered}
  \]
  increases the value of the entropy at level $\t_m$ by $2c_p\,h$.  The
  corresponding $L^1$ norm of the change in entropy is
  \[
    \begin{aligned}
      \big\|s\big(J,\wt\T\big)-s(J,\T)\big\|_{L^1}
      &= 2c_p\,|\log J_m|\,|\eta|, \com{or}\\
      \big\|s\big(\wt J,\T\big)-s(J,\T)\big\|_{L^1}
      &= 2c_p\,\t_m\,|h|,
    \end{aligned}
  \]
  respectively.
\end{lemma}

A consequence of these lemmas is the following theorem, which states
that generically, zero divisors are nonresonant, so generate nonlinear
pure tones of finite amplitude.  Moreover, nonresonant profiles are
dense in the set of measurable profiles.  This means that given any
entropy profile $s(x)$ on $(0,\ell)$, an arbitrarily small
perturbation of $s(x)$ supports finite amplitude pure tone space and
time periodic solutions of the compressible Euler equations.

\begin{theorem}
  \label{thm:genpw}
  Given any $T$, we define the \emph{resonant set} of $k$-modes
  to be
  \[
    \Lambda^{res}_k(T) := \Big\{(J,\T)\in\Lambda_k\;\Big|\;
    \d_j(J,\T)=0 \text{ for some }j\ne k\Big\} \subset \Lambda_k(T).
  \]
  Then the set of \emph{nonresonant} $k$-modes is generic in the
  $2n$-dimensional manifold $\Lambda_k(T)$, in that the set
  $\Lambda^{res}_k(T)$ is both meagre (or first Baire category) in
  $\Lambda_k$ and has $2n$-dimensional Hausdorff measure 0.

  Moreover, given any entropy profile, an arbitrarily small $L^1$
  perturbation yields a piecewise constant profile with infinitely
  many nonresonant zero divisors, each of which which generates a
  corresponding periodic solution.  That is, the set of entropy
  profiles which support infinitely many finite amplitude periodic
  solutions is dense in the set of all (measurable) entropy profiles.
\end{theorem}

\begin{proof}
  We write
  \[
    \Lambda^{res}_k(T) = 
    \bigcup_{j\ne k}\big(\Lambda_k\cap\Lambda_j\big),
  \]
  and by Lemma \ref{lem:lamk}, each of these intersections is both
  nowhere dense and measure zero with respect to $2n$-dimensional
  Hausdorff measure in $\B R^{2n+1}$, and the first part of the
  theorem follows.

  For the second part, we note that the equation $\d_k(J,\Theta)=0$
  can be regarded as an equation for the unknown $\t_m$ (or
  $k\piot\t_m$), which is $2\pi$-periodic, namely
  \begin{equation}
    \label{tmeq}
    \d_k = \ul z\,R(k\piot\t_m)\,\ol z = 0,
  \end{equation}
  similar to \eqref{dkz}, for some nonzero vectors $\ul z$ and
  $\ol z$.  Note that as $k$ increases, solutions of \eqref{tmeq} get
  closer together.  Given an arbitrary measurable entropy
  profile $S(x)$ and $\epsilon>0$, we can approximate the entropy
  profile by a piecewise constant profile $s(J,\T)$, with
  \[
    \big\|S-s(J,\T)\big\|_{L^1} < \epsilon/3.
  \]
  Next, by choosing $k$ large enough, we find a small perturbation
  $\wh\Theta$ of $\Theta$ with $\theta_m$ a solution of \eqref{tmeq},
  so that
  \[
    (J,\wh\T)\in\Lambda_k \com{with}
    \big\|s(J,\wh\T)-s(J,\T)\big\|_{L^1} < \epsilon/3.
  \]
  Finally, by genericity, we find a small perturbation
  $(\wt J,\wt\T)\in\Lambda_k$ of $(J,\wh\T)$ which is nonresonant, and
  for which
  \[
    \big\|s(\wt J,\wt \T)-s(J,\wh\T)\big\|_{L^1} < \epsilon/3;
  \]
  combining these completes the proof.
\end{proof}

\subsection{Dependence on Reference Period}

As in Section~\ref{sec:Tvar}, we again allow the piecewise constant
entropy field, parameterized by $(J,\T)$, to be arbitrary, and
consider the dependence of the divisors $\d_k(J,\T;T)$ on the
reference period $T$.  The total width of the evolution (in the
material frame) is $\sum_m\t_m = \ell$, and by scale invariance of our
system \eqref{lagr}, we can without loss of generality take $\ell=1$.
This means that our parameter space which fully describes any
piecewise constant entropy profile on $[0,\ell]=[0,1]$ is
\[
  J \in\B R^n_+, \qquad
  \T \in \Delta^n \subset \B R^{n+1}_+,
\]
where $\Delta^n$ is the $n$-simplex,
\[
  \Delta^n := \Big\{ \Theta\in\B R^{n+1}_+\;\Big|\;
  \sum_{m=0}^n\t_m = 1 \Big\},
\]
which is the convex hull of the standard coordinate vectors.  It will
be convenient to parameterize $\Delta^n$ by the open solid simplex in
$\B R^n$,
\[
  \Delta_\circ^n := \Big\{ \Theta^\circ\in\B R^n_+\;\Big|\;
  \sum_{m=0}^{n-1}\t_m < 1 \Big\}, \com{with}
  \t_n := 1 - \sum_{m=0}^{n-1}\t_m > 0;
\]
here we write
\[
  \T^\circ := (\t_0,\dots,\t_{n-1}) \com{and}
  \T := (\t_0,\dots,\t_n) = (\T^\circ,\t_n).
\]
Note that whenever any $J_m=1$, this description becomes degenerate,
and in fact the actual number of entropy jumps is $\#\big\{J_m\ne1\big\}$.

For a given $T$, the basis vectors $\mc T_k(T)$ and divisors
$\d_k(J,\T;T)$ are given by \eqref{Tk} and \eqref{dkpw}, respectively,
and so again satisfy the degeneracy \eqref{d2j}.  As in our earlier
development, we write $\w = k\piot$, recalling we have assumed
$\ell=1$, and we express each $\d_k(J,\T;T)$ in terms of $\w$, as
\[
  \begin{aligned}
    \d_k(J,\T;T) = \(0&1\)&P^{-k}\,
    R(\w\,\t_n)\,M(J_n)\,\dots\\&\dots
    R(\w\,\t_1)\,M(J_1)\,R(\w\,\t_0)\,\(1\\0\).
  \end{aligned}
\]
Since we are interested in those values of $T$ for which $\d_k$
vanishes, and since the matrices $R(\w\,\t_m)$ and $M(J_m)$ are
invertible, we again use the function $h$ defined in \eqref{hdef} to
describe the angles at each jump.

Proceeding inductively, we thus define
\begin{equation}
  \label{gamdef}
  \begin{aligned}
    \g_1(\w) &:= h(J_1,\w\t_0), \com{and}\\
    \g_{m+1}(\w) &:= h\big(J_{m+1},\w\t_m+\g_m), \quad m = 1,\dots,n-1.
  \end{aligned}
\end{equation}
Thus $\g_n(\w)$ is the angle obtained after the final jump $J_n$, and
our condition that $\d_k=0$ is again that $\g_n+\w\t_n$ lie exactly on
a coordinate axis.  Thus, in analogy with \eqref{weq1}, we
again define the $k$-th \emph{base frequency}
$\w^{(k)}=\w^{(k)}(J,\T)$ by the implicit equation
\begin{equation}
  \label{weq}
  \w^{(k)}\t_n + \g_n\big(\w^{(k)}\big) = k\,\frac\pi2.
\end{equation}
Here we have suppressed dependence on the entropy profile parameter
$(J,\T)$.  Note, however, that $\g_m$ depends only on the
initial part of the parameter,
\[
  \g_m = 
  \g_m\big(\w;J_1,\dots,J_m,\t_0,\dots,\t_{m-1}\big) =: \g_m(\w).
\]
To illustrate, in Figure~\ref{fig:freqs}, we show the corresponding
rotations and jumps (scaled for visibility) for the first four nonzero
eigenfrequencies.  Here the entropy field consists of four constant
states separated by three jumps.  The circular arcs represent linear
evolution of the $k$-mode through the entropy level $\t_m$ so are
rotations by $\w^{(k)}\t_m$, and the vertical segments represent the
action of the jumps, in which the second component is scaled but the
first remains constant.  The four curves are color coded, and some
arcs are labeled: for example, $\w^{(1)}\t_1$ (blue) is the evolution
of the 1-mode between jumps $J_1$ and $J_2$.

\begin{figure}[thb]
  \begin{tikzpicture}[>=stealth]
  \draw[->] (-4.5,0) -- (4.5,0);
  \draw[->] (0,-4.5) -- (0,5);

\def\th0{22}

\foreach \rad/\mult/\C in
  {3.5/1/C0,2.15/2.208/C2,3.1/4.6425/C1,1.7/3.655/C3}{%
\fill[color=\C] (\rad, 0) circle (2 pt);
\draw[very thick,color=\C,>->,shorten <=2.4pt,shorten >=2.5pt]
(\rad0, 0) arc (0:\mult*\th0:\rad) coordinate (ee);
  \foreach \J/\th in {1.3/20,1.5/15,0.8/20}
  {%
    \draw[very thick,color=C7] 
    let \p2 = (ee) in (ee) -- (\x2,\J*\y2) coordinate (ff);
    \fill[color=\C] (ee) circle (2 pt);
    \fill[color=\C] (ff) circle (2 pt);
    \draw[very thick,color=\C,>->,shorten <=2.4pt,shorten >=2.5pt]
    let \p1 = (ff), \n1 = {veclen(\x1,\y1)},
    \n2 = {atan(\y1/\x1)+180*(\x1<0)}
    in (ff) arc [radius=\n1, start angle=\n2, delta angle=\mult*\th]
    coordinate (ee);
    \fill[color=\C] (ee) circle (2 pt);
  }
}

\node[C0] at (3.6,2.3) {$\w^{(1)}\theta_1$};
\node[C2] at (-1.5,3) {$\w^{(2)}\theta_2$};
\node[C1] at (-3.55,2.8) {$\w^{(4)}\theta_1$};
\node[C3] at (-1.15,-1.5) {$\w^{(3)}\theta_3$};
\node[C1] at (2.3,-2.85) {$\w^{(4)}\theta_3$};
\end{tikzpicture}
\caption{Rotations and jumps generating frequencies}
\label{fig:freqs}
\end{figure}
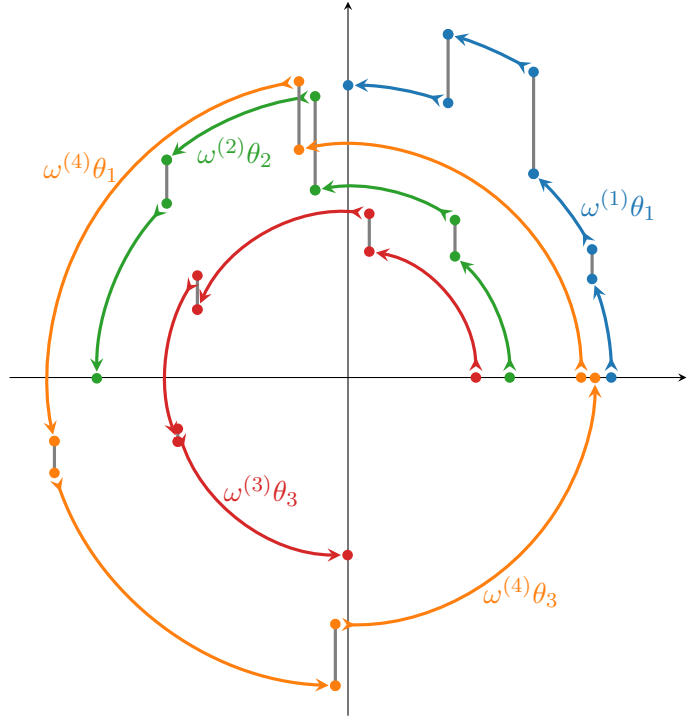

\begin{lemma}
  \label{lem:wsol}
  For any given $J\in\B R^n_+$, $\T\in\Delta^n$, there are infinitely
  many bounded time periods
  \[
    T^{(k)} := k\,\frac{2\pi}{w^{(k)}}\,\ell, \com{or}
    T^{(k)} := k\,\frac{2\pi}{w^{(k)}}, \quad \ell=1,
  \]
  such that the $k$-mode with corresponding reference period
  $T^{(k)}$, that is $\mc T_k\big(T^{(k)}\big)$, has a zero divisor,
  $\d_k(J,\T;T^{(k)})=0$.
\end{lemma}

\begin{proof}
  We have seen that $\d_k(J,\T;T)=0$ if and only if \eqref{weq} holds
  for appropriate $k$.  To show that this uniquely defines $\w^{(k)}$,
  we differentiate \eqref{gamdef}: inductively, we have
  \[
    \begin{aligned}
      \frac{\del\g_1}{\del\w}
      &= \frac{\del h}{\del x}\Big|_{(J_1,\w\t_0)}\,\t_0 \com{and}\\
      \frac{\del\g_{m+1}}{\del\w}
      &= \frac{\del h}{\del x}\Big|_{(J_{m+1},\w\t_m+\g_m)}\,
        \Big(\t_m + \frac{\del\g_m}{\del\w}\Big).
    \end{aligned}
  \]
  Now by \eqref{hder}, each of these terms is positive, and so
  differentiating the left hand side of \eqref{weq}, we get
  \[
    \frac{\del}{\del\w}\Big(\g_n(\w)+\w\,\t_n\Big) > \t_n > 0,
  \]
  so by the implicit function theorem there exists a unique positive
  solution $\w^{(k)}$ for each $k\ge1$.

  We now show that $\w^{(k)}$ grows like $k$ for $k$ large.  Note from
  \eqref{hdef} that $h(J,x)$ is always in the same quadrant as $x$,
  and $|h(J,x)-x\big|<\frac\pi2$.  Applying this to \eqref{gamdef}, we
  get
  \[
    \big|\g_1 - \w\t_0\big| < \frac\pi2, \com{and}
    \big|\g_{m+1} - (\w\t_m+\g_n)\big| < \frac\pi2,
  \]
  for $m=1,\dots,n-1$.  Telescoping and using the triangle inequality
  we get
  \[
    \Big|\g_n - \w\sum_{m=0}^{n-1}\t_m\Big| < n\,\frac\pi2,
  \]
  and finally using \eqref{weq} to eliminate $\g_n$ yields
  \[
    \Big|k\,\frac\pi2 - \w^{(k)}\sum_{m=0}^n\t_m\Big|
    < n\,\frac\pi2,
  \]
  and since $\sum_m\t_m=1$, this yields
  \[
    (k-n)\,\frac\pi2 < \w^{(k)} < (k+n)\,\frac\pi2,
  \]
  and the boundedness of reference periods $T^{(k)}$ follows.
\end{proof}

Having described the frequencies $\w^{(k)}$ and corresponding periods
$T^{(k)}$, we again look for the conditions that determine resonance
of distinct frequencies.  We have shown above that each $\w^{(k)}$ can
be regarded as a single-valued, non-degenerate smooth function
\[
  \w^{(k)}:\B R^n_+\times\Delta^n\to\B R, \qquad
  \w^{(k)} = \w^{(k)}\big(J,\T\big),
\]
in which the simplex $\Delta^n\subset\B R^{n+1}_+$ is independently
parameterized by the $\T^\circ\in\Delta^n_\circ\subset\B R^n$.
Because the addition of extra entropy intervals introduces more
degrees of freedom, the following generalization of
Theorem~\ref{thm:T} to a general step entropy profile is unsurprising.

\begin{theorem}
  \label{thm:T2}
  There is a meagre subset $\mc Z\subset\Delta^n\times\B R^n$ of
  ($2n$-dimensional) measure zero, such that, for any
  $(J,\T)\notin\mc Z$, \emph{every} frequency $\w^{(k)}$ is
  non-resonant.  For each of the corresponding piecewise constant
  entropy profiles, the $k$-mode with reference period $T^{(k)}$,
  $k\ge 1$, is non-resonant, and so perturbs to a finite amplitude
  solution of the nonlinear equation \eqref{F=0pw}.  This in turn
  generates a space and time periodic solution of the nonlinear
  Euler equations.
\end{theorem}

\begin{proof}
  We proceed as in the proof of Theorem~\ref{thm:T}.  Thus we fix a
  $k$-mode, and assume that the $j$ mode resonates with this $k$ mode.
  The $k$-mode determines the reference period
  $T^{(k)}=k\,2\pi/\w^{(k)}$, and resonance of the $j$-mode means that
  also $\d_j\big(J,\T;T^{(k)}\big)=0$.  This in turn corresponds to
  $\w = j\,2\pi/T^{(k)}$ being a solution of \eqref{weq}, say
  \begin{equation}
    \label{res}
    \w = j\,\frac{2\pi}{T^{(k)}} = \w^{(p)}, \com{so that}
    \w^{(p)} = q\,\w^{(k)}, \quad q :=\frac jk\in\B Q_+.
  \end{equation}
  Thus, fixing positive integers $k$, $j\ne k$ and $p$, we define
  \[
    \mc Z_{k,p,j} := \Big\{(J,\T)\in\B R^n_+\times\Delta^n\;\Big|\;
    \w^{(p)} = q\,\w^{(k)}\Big\},
  \]
  where $\w^{(k)}$ and $\w^{(p)}$ are functions of $(J,\T)$, regarded
  as known.  We wish to show that this is a non-degenerate constraint
  on the parameters, which will imply that $\mc Z_{k,p,j}$ is a
  codimension one submanifold, and so is small, in both the measure
  and Baire senses.

  We again get a more explicit version of the angles, and so also
  $\w^{(k)}$, by using the fact that $\t_n = 1 - \sum_{m<n}\t_m$ in
  \eqref{weq}, which gives, with $\w=\w^{(k)}$,
  \begin{equation}
    \label{woft}
    \w = \sum_{m=0}^{n-1} \w\t_m - \g_n(\w) + k\,\frac\pi2.
  \end{equation}
  Now whenever $\t_m$ appears in \eqref{gamdef}, it is scaled by $\w$,
  and this is the only context in which it appears.  We thus make the
  substitution $\w\t_m\to x_m$ for $m<n$.  That is, we use $\T^\circ$
  to parameterize the simplex, and scale this up to a vector
  \[
    x = (x_0,\dots,x_{n-1}) \in\B R^n, \com{by}
    x_m = \w\,\t_m, \com{or} x = \w\,\T^\circ.
  \]
  We now use \eqref{gamdef} to define $\g_m$ as a function of $x$,
  namely
  \begin{equation}
    \label{gx}
    \begin{aligned}
      \g_1(x) &:= h(J_1,x_0), \com{and}\\
      \g_{m+1}(x) &:= h\big(J_{m+1},x_m+\g_m\big), \quad m = 1,\dots,n-1,
    \end{aligned}    
  \end{equation}
  which now yields the explicit function $\g_n = \g_n(J,x)$.  

  We can combine this function and the scaling $x=\w\,\T^\circ$ to get
  a geometric description of $w^{(k)}$, as follows.  Referring to
  \eqref{woft}, we see that \eqref{weq} is equivalent to the coupled
  expressions
  \begin{equation}
    \label{wxk}
    \w^{(k)} = \sum_{m=0}^{n-1}x^{(k)}_m - \g_n(x^{(k)}) + k\,\frac\pi2, \qquad
    x^{(k)}_m = \w^{(k)}\,\t_m.
  \end{equation}
  We thus interpret $\w^{(k)}$ as being determined by the intersection
  of a ray $\w\,\T^\circ$ and the graph of the explicit function
  $\sum x_m+\g_n(x)+k\frac\pi2$.

  We now use \eqref{wxk} to re-express the resonance condition
  \eqref{res} as a single equation: that is, we have
  \[
    \w^{(p)} = q\,\w^{(k)} \com{if and only if}
    g(J,x) = (p-j)\,\frac\pi2,
  \]
  where we have defined the function
  \begin{equation}
    \label{gJx}
    g:\B R^{2n}_+\to\B R \com{by}
    g(J,x) := \g_n(q\,x) - q\,\g_n(x),
  \end{equation}
  and written
  \[
    x := x^{(k)}, \quad \w := \w^{(k)}, \quad
    \w^{(p)} = q\,\w, \com{and} x^{(p)} = q\,x,
  \]
  with $q=j/k\in\B Q_+$ fixed and $\g_n$ is given by \eqref{gx}.  Here
  the vector identity $x^{(p)} = q\,x^{(k)}$ follows from the second
  equation in \eqref{wxk}.  Thus we have effectively described the
  resonant set as a level surface of the explicit function $g$, namely
  \[
    \mc Z_{j,k,p} = \Big\{(J,x)\in\B R^{2n}\;\Big|\;
    g(J,x) = (p-j)\,\frac\pi2\Big\},
  \]
  and we recover the original coordinates $(J,\T)\in\B
  R^n_+\times\Delta^n$ by setting
  \[
    \T^\circ := \frac x{\w^{(k)}}, \com{and}
    \t_n = 1 - \sum_{m=0}^{n-1}\t_m,
  \]
  in which $\w^{(k)}$ is given explicitly by \eqref{wxk}.

  We will show that $g$ is non-degenerate, that is
  \[
    \nabla_{(J,x)}g \ne 0 \com{whenever} g(J,x) = b,
  \]
  for any nonzero constant $b\ne0$.  From the definition \eqref{gJx}, and
  using \eqref{hder}, we calculate
  \[
    \begin{aligned}
      \frac{\del g}{\del x_{n-1}}
      &= \frac{\del h}{\del x}
      \bigg|_{\big(J_n,q\,x_{n-1}+\g_{n-1}(qx)\big)}\,q
      - q\,\frac{\del h}{\del x}
      \bigg|_{\big(J_n,x_{n-1}+\g_{n-1}(x)\big)}\\[3pt]
      &= \frac{q\,J_n\,(1+\G_{n-1}^2(qx))}{1 + J^2\,\G_{n-1}^2(qx)}
        - \frac{q\,J\,(1+\G_{n-1}^2(x))}{1 + J^2\,\G_{n-1}^2(x)}\\[3pt]
      &= \frac{q\,J\,(J^2-1)\,\big(\G_{n-1}^2(x)-\G_{n-1}^2(qx)\big)}
        {\big(1 + J^2\,\G_{n-1}^2(qx)\big)\,\big(1 + J^2\,\G_{n-1}^2(x)\big)},
    \end{aligned}
  \]
  where we have set
  \[
    \G_{n-1}(x) := \tf\big(x_{n-1}+\g_{n-1}(x)\big).
  \]
  It follows that
  \[
    \frac{\del g}{\del x_{n-1}} = 0
    \com{iff}
    \G_{n-1}(qx) = \pm\G_{n-1}(x),
  \]
  unless $J_n=1$.  Similarly,
  \[
    \begin{aligned}
      \frac{\del g}{\del J_n}
      &= \frac{\del h}{\del J}
      \bigg|_{\big(J_n,q\,x_{n-1}+\g_{n-1}(qx)\big)}
      - q\,\frac{\del h}{\del J}
      \bigg|_{\big(J_n,x_{n-1}+\g_{n-1}(x)\big)}\\
      &= \frac{\G_{n-1}(qx)}{1 + J^2\,\G_{n-1}^2(qx)}
        - q\,\frac{\G_{n-1}(x)}{1 + J^2\,\G_{n-1}^2(x)}.
    \end{aligned}
  \]
  Thus, if we assume $J_n\ne 1$, and assume both
  \[
    \frac{\del g}{\del x_{n-1}}\bigg|_{(J,x)} = 0 \com{and}
    \frac{\del g}{\del x_{n-1}}\bigg|_{(J,x)} = 0,
  \]
  then we must have both
  \[
    \G_{n-1}(qx) = \pm\G_{n-1}(x) \com{and}
    \G_{n-1}(qx) - q\,\G_{n-1}(x) = 0,
  \]
  which in turn implies $\G_{n-1}(qx) = \G_{n-1}(x) = 0$, so that, by
  \eqref{gJx}, \eqref{gx} and \eqref{hdef}, we get
  \[
    g(J,x) = \arctan\big(J_n\,\G_{n-1}(qx)\big)
    - q\,\arctan\big(J_n\,\G_{n-1}(x)\big) = 0.
  \]

  On the other hand, if $J_n=1$, which is the degenerate case, we
  carry out the same calculation for $\frac{\del g}{\del x_{n-2}}$ and
  $\frac{\del g}{\del J_{n-1}}$, and continue by backward induction as
  necessary.  If all $J_m=1$, which is the most degenerate isentropic
  case, then since $h(1,x)=x$, we get $g = 0$ identically.

  We have thus shown that, as long as at least one $J_m\ne 1$,
  \[
    \nabla_{(J,x)} g = 0 \com{implies} g(J,x) = 0.
  \]
  The implicit function theorem now tells us that non-zero level sets
  of $g$, and in particular $\mc Z_{k,p,j}$, are codimension one
  submanifolds of $\B R^{2n}$, and so these are again both measure
  zero and nowhere dense.

  As before, we now set
  \[
    \mc Z := \bigcup_{k,p,j}\mc Z_{k,p,j},
  \]
  and since this is a countable union, $\mc Z$ has measure zero and is
  meagre.  Since $\mc Z$ contains all possible resonances, the proof
  is complete.
\end{proof}

\section{Generalizations}
\label{sec:SL}

We return now to the compressible Euler equations \eqref{lagr} with a
general constitutive law, $p = p(v,s)$, in which we solve for $v$, so
that $v=v(p,s)$.  For classical solutions, in a Lagrangian frame, we
write the system as
\begin{equation}
  \label{genls}
  p_x + u_t = 0, \qquad
  u_x - v(p,s)_t = 0,
\end{equation}
or, in quasilinear form,
\[
  p_x + u_t = 0, \qquad
  u_x - v_p(p,s)\,p_t = 0,
\]
where we are again evolving in $x$, and since $s_t=0$, we regard the
entropy as a given fixed function $s(x)$.  As above we make the
observation that the nonlinear equations respect the symmetry $p$
even, $u$ odd, as functions of $t$.  That is, if we specify
$p(x_0,\cdot)$ even, $u(x_0,\cdot)$ odd at one point $x_0$, then this
is satisfied throughout the interval of classical existence, that is
as long as gradients remain finite.

\subsection{Minimal Tile for Periodicity and Boundary Conditions}

We again build a tile which generates a periodic solution by assuming
the data is time periodic and imposing extra symmetry conditions at
the ends $x=0$ and $x=\ell$ of the evolution.  As in \eqref{olsymm},
\eqref{ols}, we impose a regular reflective boundary condition at
$x=0$: that is, we impose the \emph{acoustic reflection boundary
  condition} $u(0,\cdot)=0$, with $p(0,\cdot)$ even and periodic.
Physically, this acoustic boundary condition describes pure (lossless)
reflection of a sound wave off a wall.  On the other hand, in
\eqref{ulsymm}, \eqref{uls}, the boundary condition which is a shifted
reflection, provides a mechanism for the generation of periodic tiles,
but is not realizable as a simple physical boundary condition.  This
is because when reflecting the profile around $x=\ell$ to generate a
(half) periodic tile, the velocity need not actually vanish at
$x=\ell$.  This can be explained by noting the effect of the
quarter-period shift $S^{T/4}$ on a $k$-mode: we have
\[
  \mc S^{T/4}\s\big(k\piot t\big) =
  \s\big(k\piot t-k{\TS{\frac\pi2}}\big) =
  \begin{cases}
    \pm\c\big(k\piot t\big), & k \text{ odd},\\
    \pm\s\big(k\piot t\big), & k \text{ even},
  \end{cases}
\]
and similarly for $\mc S^{T/4}\c\big(k\piot t\big)$.  It follows that,
for $k$-modes with $k$ even,
\[
  \IR-\,\mc S^{T/4}\,\mc T_{k} = 0 \com{iff}
  \IR-\,\mc T_{k} = 0,
\]
while this fails for odd $k$-modes.  As a consequence, at least at the
linear level, if we double the time frequency of tones, then we can
drop the $1/4$-period shift in posing the right boundary condition,
and this then becomes an acoustic reflection boundary condition at
$x=\ell$, exactly as at $x=0$.

It follows that we can pose two problems that effectively have the
same method of solution:  as in previous sections, we assume $p$ even
and $u$ odd as functions of $t$, and set
\begin{equation}
  \label{ypu}
  y(x,t) := p(x,t) + u(x,t),
\end{equation}
and define a nonlocal scalar flux by
\[
  g(y) := u + v(p,s) = \IR-y + v\big(\IR+y,s\big),
\]
so that $y$ solves the nonlocal scalar conservation law
\begin{equation}
  \label{ycl}
  y_x + g(y)_t = 0.
\end{equation}
Note that this is distinct from our earlier treatment of piecewise
constant entropy, in which we used the isentropic equations at each
entropy level together with the rescaling of $x$, whereas here we use
the full equations and unscaled $x$ directly.

The two problems we then pose are:
\begin{itemize}
\item the \emph{periodic tile problem}, namely
  \begin{equation}
  \label{yper}
    \mc F_P(y^0) := \IR-\,\mc S^{-T/4}\,\mc E^\ell\,y^0 = 0;
  \end{equation}
\item the \emph{acoustic boundary value problem}, namely
  \begin{equation}
  \label{yab}
    \mc F_A(y^0) := \IR-\,\mc E^\ell\,y^0 = 0.
  \end{equation}
\end{itemize}
In both of these problems, as above, $\mc E^\ell$ denotes nonlinear
evolution through the varying entropy profile from $x=0$ to $x=\ell$,
and the data $y^0=y^0(t)$ is again assumed to be even and $T$-periodic.

For a given entropy profile $s(x)$, each solution of \eqref{yper}
generates a space and time periodic solution, generated by a
reflection at the left boundary and a shifted reflection on the right,
so the space period of the solution is $4\ell$.  On the other hand, a
solution of \eqref{yab} generates a periodic solution by just one
reflection in $x$, so has space period $2\ell$.  The two problems are
related as follows: any $k$-mode solution of the linearization of
\eqref{yper} with $k$ even is also a solution of that of \eqref{yab},
and conversely an even $2j$-mode linearized solution of \eqref{yab}
can also be realized as a solution of the linearization of
\eqref{yper}.  Moreover, if these linearized solutions perturb, then
by uniqueness they will coincide as solutions of the corresponding
nonlinear problems.

It appears that \eqref{yper} generates more solutions than
\eqref{yab}, but we regard \eqref{yab} as a more physically relevant
problem.  This is because the acoustic boundary condition \eqref{yab}
is simply a reflection off a wall, so the system models sound waves
bouncing between two walls which bound an unrestricted varying entropy
profile on $[0,\ell]$.  On the other hand, the velocity $u$ need not
vanish at $x=\ell$ with \eqref{yper}, and it is hard to ascribe
\eqref{yper} to a physical condition, because we expect that
controlling an entropy profile to be perfectly symmetric between two
walls (at $x=0$ and $x=2\ell$) is practically infeasible.  However, we
regard periodic solutions in which compression and rarefaction are in
perfect balance as being of fundamental importance.

\subsection{Linearization and Sturm-Liouville systems}

We now analyze \eqref{genls} following the same steps as before.  We
regard the entropy field $s(x)$ as a given piecewise continuous
function on the interval $x\in[0,\ell]$, continuous at the endpoints,
and assume a constitutive law
\[
  v = v(p,s), \com{with} v_p<0,
\]
and we will make further assumptions as necessary.

We begin by noting that the quiet state $(\ol p,0)$ is a solution of the
Euler equations satisfying the boundary conditions; although it is a
constant solution of \eqref{genls}, it is part of a non-constant
standing wave solution of \eqref{lagr}.  Linearizing \eqref{genls}
around $(\ol p,0)$ gives the system
\begin{equation}
  \label{UP}
  \begin{gathered}
    P_x + U_t = 0, \qquad
    U_x + \sg^2\,P_t = 0,\\ \com{where}
    \sg = \sg(x) := \sqrt{-v_p(\ol p,s)},    
  \end{gathered}
\end{equation}
in which we use the convention that $(P,U)$ solve linearized
equations, while $(p,u)$ solve the nonlinear system.

We use separation of variables to solve \eqref{UP}, in which we look
for $T$-periodic solutions of \eqref{UP} in which $P$ and $U$ are even
and odd, respectively, so we set
\begin{equation}
  \label{UPsub}
  \begin{aligned}
    P(x,t) &:= \sum_{n\ge 0} \vp_n(x)\,\c\big(n\piot t\big), \\
    U(x,t) &:= \sum_{n> 0} \psi_n(x)\,\s\big(n\piot t\big).
  \end{aligned}
\end{equation}
Plugging in to \eqref{UP} and simplifying, we get the ODE system
\begin{equation}
  \label{SL1}
  \dot\vp_n + \w\,\psi_n = 0, \qquad
  \dot\psi_n - \sg^2\,\w\,\vp_n = 0,  \qquad
  \w := n\piot,
\end{equation}
where $\dot\square$ denotes $\frac{d}{dx}\square$.

We denote nonlinear evolution through $x$ according to \eqref{ycl} by
$\mc E^x$, and we denote the linearization of this around the constant
quiet state $(\ol p,0)$ by $\mc L^x := D\mc E^x(\ol p,0)$.  That is,
$\mc L^x$ denotes evolution by \eqref{UP}, or more precisely by the
linear nonlocal scalar equation for $Y := P+U$,
\begin{equation}
  \label{Ylin}
  Y_x + \big(\IR-Y+\sg^2\,\IR+Y\big)_t = 0.
\end{equation}
We again use the notation \eqref{Tk}, so that \eqref{UPsub} can be
written
\begin{equation}
  \label{YasT}
  Y(x,t) = P + U = \sum_{n\ge 0} \mc T_n\,\(\vp_n(x)\\\psi_n(x)\),
\end{equation}
and we obtain an analogue of Lemma~\ref{lem:Tk}, whose proof is
omitted.

\begin{lemma}
  \label{lem:Lx}
  The linearized evolution operator $\mc L^x = D\mc E^x(\ol p,0)$ acts
  on the $k$-mode row vector $\mc T_k = \mc T_k(T)$ given by
  \eqref{Tk}, by
  \[
    \mc L^x\,\mc T_k = \mc T_k\,\Psi\big(x;k\piot\big), \com{so}
    \mc L^x\,\mc T_k\(a\\b\) = \mc T_k\,\Psi\(a\\b\),
  \]
  where $\Psi = \Psi(x;\w)$ is the fundamental solution of
  \eqref{SL1}, so satisfies
  \begin{equation}
    \label{Psi}
    \dot\Psi(x;\w) = \w\,
    \(0&-1\\\sg^2(x)&0\)\,\Psi(x;\w),
    \qquad \Psi(0;\w) = I.
  \end{equation}
\end{lemma}

Alternatively, we can express the linearized system \eqref{UP} as a
wave equation with varying speed, namely
\begin{equation}
  \label{Peq}
  P_{xx} - \sg^2\,P_{tt} = 0,
\end{equation}
and after use of \eqref{UPsub}, by separation of variables, we get the
second-order linear system
\begin{equation}
  \label{SL2}
  \ddot \vp_n + \big(n\piot\big)^2\,\sg^2\,\vp_n = 0,
\end{equation}
which is equivalent to \eqref{SL1}.

We now consider boundary values for these ODE systems.  In both
\eqref{yper} and \eqref{yab}, the data $y^0$ (and so also $Y^0$) posed
at $x=0$ is even, so using \eqref{ypu} and \eqref{UPsub}, this becomes
the condition
\begin{equation}
  \label{bc0}
  U(0,\cdot) = 0, \com{equivalently} \psi_n(0) = \dot\vp_n(0) = 0.
\end{equation}
If we are solving \eqref{yab}, we get the same condition at $x=\ell$,
namely
\begin{equation}
  \label{bcab}
  U(\ell,\cdot) = 0, \com{equivalently} \psi_n(\ell) = \dot\vp_n(\ell) = 0.
\end{equation}
On the other hand, if we are solving \eqref{yper}, the boundary
condition is
\[
  \begin{aligned}
  0 &= \IR-\,\mc S^{-T/4}\big(P(\ell,\cdot)+U(\ell,\cdot)\big)\\
  &= \IR-\sum\Big(\vp_n(\ell)\,\c\big(n\piot t-n{\TS\frac\pi2}\big)
  +\psi_n(\ell)\,\s\big(n\piot t-n{\TS\frac\pi2}\big)\Big),
  \end{aligned}
\]
which yields the conditions
\begin{equation}
  \label{bcper}
  \begin{gathered}
    \dot\vp_n(\ell) = \psi_n(\ell) = 0, \quad n \text{ even},\\
    \vp_n(\ell) = \dot\psi_n(\ell) = 0, \quad n \text{ odd}.
  \end{gathered}
\end{equation}

Our boundary conditions can now be expressed succinctly using the
fundamental solution \eqref{Psi}, in analogy with Lemma~\ref{lem:pws},
as follows.

\begin{lemma}
  \label{lem:fund}
  Suppose that constant ambient pressure $\ol p$, entropy profile
  $s(x)$ and reference period $T$ are given.  Then the constant quiet
  state $y^0=\ol p$ solves \eqref{yper} and \eqref{yab}, and the
  respective linearizations around $\ol p$ are
  \[
    D\mc F_P(\ol p) = \IR-\,\mc S^{-T/4}\,\mc L^\ell \com{and}
    D\mc F_A(\ol p) = \IR-\,\mc L^\ell,
  \]
  and moreover the nonlinear functionals $\mc F_P$ and $\mc F_A$
  factor as
  \begin{equation}
    \label{Ngenl}
    \mc F_P = \IR-\,\mc S^{-T/4}\,\mc L^\ell\,\mc N^\ell \com{and}
    \mc F_A = \IR-\,\mc L^\ell\,\mc N^\ell,
  \end{equation}
  respectively, where
  $\mc N^\ell := \big(\mc L^\ell\big)^{-1}\mc E^\ell$ in both cases.
  The linearized operators respect $k$-modes, and we get
  \[
    \begin{aligned}
      D\mc F_P(\ol p)\Big[\c\big(k\piot t\big)\Big]
      &= \d_{P,k}(T)\,\s\big(k\piot t\big) \com{and}\\
      D\mc F_A(\ol p)\Big[\c\big(k\piot t\big)\Big]
      &= \d_{A,k}(T)\,\s\big(k\piot t\big),
    \end{aligned}
  \]
  where the $k$-th divisor is given by
  \begin{equation}
    \label{dkcts}
    \begin{aligned}
      \d_{P,k}(T) &:= \(0&1\)\,P^{-k}\,
      \Psi\big(\ell;k\piot\big)\(1\\0\) \com{or}\\
      \d_{A,k}(T) &:= \(0&1\)\,
      \Psi\big(\ell;k\piot\big)\(1\\0\),    
    \end{aligned}
  \end{equation}
  for the periodic or acoustic boundary conditions, respectively.
\end{lemma}

It follows that the single mode $\c\big(k\piot t\big)$ provides a
solution given by $\mc T_k\,\Psi\big(x;k\piot\big)$ of the linearized
problem satisfying the boundary conditions if and only if
\[
  \d_{P,k}(T)=0 \com{or} \d_{A,k}(T)=0,
\]
respectively.  In this case,
$\Psi\big(x;k\piot\big)\scriptstyle\begin{pmatrix}1\\0\end{pmatrix}$
solves \eqref{SL1}, and since it satisfies the boundary conditions, it
is an eigenfunction of the Sturm-Liouville problem \eqref{SL2},
corresponding to eigenvalue $\lambda_k := \big(k\piot\big)^2$.  As in
previous sections, we expect that if the corresponding mode is
nonresonant, then this should perturb to a nonlinear pure tone
solution of \eqref{genls}.

We are thus led to introduce the Sturm-Liouville (SL) operator
\begin{equation}
  \label{SLop}
  \mc L := - \frac1{\sg^2}\,\frac{d^2}{dx^2}, \com{so that}
  \mc L \phi = - \frac1{\sg^2}\,\ddot\phi,
\end{equation}
and to introduce the associated SL eigenvalue problem,
\begin{equation}
  \label{SLev}
  \mc L\,\vp = \lambda\,\vp, \com{which is}
  - \ddot\vp = \lambda\,\sg^2\,\vp,\com{on} 0<x<\ell,
\end{equation}
and subject to the boundary conditions \eqref{bc0} at $x=0$ and either
\eqref{bcper} or \eqref{bcab} at $x=\ell$.  We note that the boundary
conditions imply self-adjointness, implying that this is a
\emph{regular SL problem} as long as the weight
$\sg^2(x) = -v_p\big(\ol p,s(x)\big)$ is a positive, bounded piecewise
continuous function~\cite{Pryce,BR,CodLev}.  In particular, the
setup we have here is a direct generalization of the cases studied
earlier in this paper, and the results here apply unchanged in that
context.  Moreover, because the corresponding ODEs are linear, the
solutions can be expressed as integrals, and our methods extend to BV
entropy profiles with only minor technical changes.

We collect some well-known classical results for regular SL eigenvalue
problems in the following lemma; see \cite{Pryce,CodLev} for details.
This is the direct generalization of the first statement of
Lemma~\ref{lem:Tj}; indeed, the existence of infinitely many reference
periods for a piecewise constant entropy profile is a special case of
the classical Sturm-Liouville theory.  Many of the statements in this
lemma will be explicitly verified below.

\begin{lemma}
  \label{lem:SLprops}
  Assume that the entropy profile is piecewise continuous.  The
  Sturm-Liouville system \eqref{SLev}, with boundary conditions
  \eqref{bc0} and \eqref{bcper} or \eqref{bcab}, has infinitely many
  eigenvalues $\lambda_k$.  These are positive, simple, monotone
  increasing and isolated, and satisfy the growth condition
  $\lambda_k=O(k^2)$.  The corresponding eigenfunctions form an
  orthogonal $L^2$ basis in the Hilbert space with weight function
  $\sg^2$.
\end{lemma}

We label the eigenvalues $\lambda_k$ and corresponding eigenfunctions
$\vp_k$, scaled so that $\vp_k(0)=1$, so that \eqref{SLev} and
\eqref{SL1} hold for each $k\ge1$, and the $\lambda_k$'s are
increasing with $k$.  Because we prefer to work with the frequencies,
we define the $k$-th \emph{eigenfrequency} $\w_k$ and corresponding
\emph{reference period} $T_k$ by
\begin{equation}
  \label{Tn}
\w_k := \sqrt{\lambda_k}, \com{and} 
  T_k := k\,\frac{2\pi}{\w_k} = k\,\frac{2\pi}{\sqrt{\lambda_k}},
\end{equation}
respectively.  It follows that the $k$-modes
\[
  \begin{aligned}
  Y_k(x,t) &:= \mc T_k(T_k)\,\Psi\big(x;k{\TS\frac{2\pi}{T_k}}\big)
  \,{\TS\begin{pmatrix}1\\0\end{pmatrix}}
  \com{or}\\
  P_k(x,t) &:= \vp_k(x)\,\c\big(k{\TS\frac{2\pi}{T_k}}t\big) = \vp_k(x)\,\c(\w_kt),
  \end{aligned}
\]
solve \eqref{Ylin} or \eqref{Peq}, respectively, while also satisfying
the appropriate boundary conditions.  These $k$-modes lie in the
kernel of the linearized operator $D\mc F(\ol p)$, and we will show
that these perturb to pure
tone solutions of the nonlinear problem.  We note that for given $k$,
the frequency $\w_k$ is determined by the entropy profile, and this in
turn yields the reference period.  Consideration of other resonant and
nonresonant modes then refers to this fixed period $T_k$.
\subsection{Nonresonant modes}

As in our above development, we wish to perturb $Y_k$ (or $P_k$) to a
time periodic solution of the corresponding nonlinear equation.
Following \eqref{y0}, and having identified the reference period,
we consider perturbations of the initial data of the form
\begin{equation}
  \label{pert}
  y^0(0) = p^0(t) := \ol p + \a\,\c\big(k{\TS\frac{2\pi}{T_k}}t\big)
   + z + \sum_{j\ge 1,j\ne k}a_j\,\c\big(j{\TS\frac{2\pi}{T_k}}t\big),
\end{equation}
in which the reference period $T_k$ is determined by the choice of
SL eigenvalue $\lambda_k$ in \eqref{Tn}, and we now regard this as
fixed.

We again declare the $k$-mode to be \emph{nonresonant} if no other
$j$-modes satisfy the boundary condition, that is
\[
  \d_{P,j}(T_k)\ne0 \com{or} \d_{A,j}(T_k)\ne0,
  \com{for all} j\ne k,
\]
respectively for the appropriate boundary conditions \eqref{yper} or
\eqref{yab}.  Here we note that the reference period $T_k$ is
fixed, while $j$ varies, and the collection of $j$-mode basis elements
\[
  \big\{\mc T_j(T_k)\;|\;j\ge 1\big\} =
  \Big\{\(\c(j\piotk t)&\s(j\piotk t)\)\;|\;j\ge 1\Big\}
\]
span the even and odd functions of period $T_k$, respectively.

\begin{lemma}
  \label{lem:nonres}
  The $j$-mode is resonant with the fixed $k$-mode if and only if two
  distinct SL frequencies are rationally related: that is,
  \[
    \d_j(T_k) = 0 \com{iff}
    k\,\w_p = j\,\w_k,
  \]
  for some index $p = p(j)\ne k$.
\end{lemma}

\begin{proof}
  By construction, we have
  \[
    \d_k(T_k) = 0, \com{with}
    T_k := k\,\frac{2\pi}{\w_k}.
  \]
  If $\d_j(T_k) = 0$, then $j\frac{2\pi}{T_k}$ also corresponds to some
  SL frequency, and since $j\ne k$, we must have
  \[
    \w_p = j\frac{2\pi}{T_k} \com{for some}
    p\ne k,
  \]
  and substituting in $T_k$ gives the result.
\end{proof}

The resonance or nonresonance of modes corresponding to SL eigenvalues
depends on the entropy profile and constitutive law.  Moreover,
because the SL eigenvalues, and hence also frequencies, vary
continuously with the entropy profile, we expect that in a strong
sense, as in the previous cases, the set of profiles with resonant
modes should be small, so that generically, all modes should be
nonresonant.

\subsection{Angle Variable}

Once again we treat the general entropy profile as given, and solve
for the base frequencies $\w_k$ that yield $k$-mode solutions of the
linearized equation.  Starting with the linearization \eqref{UP}, we
separate variables by setting
\[
  P(x,t) := \vp(x)\,\c(\w\,t), \qquad
  U(x,t) := \psi(x)\,\s(\w\,t),
\]
which yields the SL system \eqref{SL1}, namely
\begin{equation}
  \label{SL3}
  \dot\vp + \w\,\psi = 0, \qquad
  \dot\psi - \w\,\sg^2\,\vp = 0,
\end{equation}
with $\sg^2 = - v_p(\ol p,s) = \sg^2(x)$.  According to \eqref{bc0},
we first consider initial values
\[
  \vp(0) = c_0 \com{and} \psi(0) = 0,
\]
for appropriate $c_0\ne 0$.  Finding the values of $\w_k$ that meet
the boundary conditions \eqref{bcab} or \eqref{bcper} will then
determine the appropriate reference period
$T_k:=k\,\frac{2\pi}{\w_k}$.

As we have seen earlier, the most important part of the evolution is
the angle $\t=\t(x)$ of the vector $(\vp,\psi)$.  This can be
effectively captured with the use of \emph{modified Pr\"ufer
  variables}, which are
\begin{equation}
  \label{pruf}
  \vp(x) := r(x)\,\frac1{\rho(x)}\,\c\big(\t(x)\big), \qquad
  \psi(x) := r(x)\,\rho(x)\,\s\big(\t(x)\big),
\end{equation}
see \cite{Pryce,CodLev,BR}.  Here we interpret $r(x)$ as the
radial length or amplitude, $\rho(x)$ as the eccentricity or aspect,
and $\t(x)$ as the angle variable.  This is a degenerate description
in which we are free to choose the aspect $\rho(x)>0$, and having done
so, both $r(x)$ and $\t(x)$ will be determined by the equations.
Plugging in \eqref{pruf} into \eqref{SL3} and simplifying, we get the
system
\[
  \begin{aligned}
    \frac{\dot r}r\,\c(\t) - \frac{\dot\rho}{\rho}\,\c(\t)
    - \s(\t)\,\dot\t + \w\,\rho^2\,\s(\t) &= 0, \\
    \frac{\dot r}r\,\s(\t) + \frac{\dot\rho}{\rho}\,\s(\t)
    + \c(\t)\,\dot\t - \w\,\frac{\sg^2}{\rho^2}\,\c(\t) &= 0,
  \end{aligned}
\]
with initial conditions
\[
  \t(0) = 0 \com{and} r(0) = 1,
  \com{so} c_0 = \frac1{\rho(0)}.
\]
After use of elementary trig identities, this in turn becomes
\[
  \begin{aligned}
    \frac{\dot r}r - \frac{\dot\rho}{\rho}\,\c(2\t)
    + \w\,\Big(\rho^2-\frac{\sg^2}{\rho^2}\Big)\s(\t)\,\c(\t) &=0, \\
    \dot\t + \frac{\dot\rho}{\rho}\,\s(2\t)
    - \w\,\Big(\frac{\sg^2}{\rho^2}\,\c^2(\t) + \rho^2\,\s^2(\t)\Big) &=0.
  \end{aligned}
\]
Thus the system is reduced to a nonlinear scalar ODE for the angle
$\t(x)$, coupled with an integration for the amplitude $r(x)$.
Moreover, it is now clear that we should choose $\rho$ such that
\begin{equation}
  \label{rho}
  \rho^2 = \frac{\sg^2}{\rho^2}, \com{that is}
  \rho (x) := \sqrt{\sg(x)}.
\end{equation}
With this choice the equations simplify further, and we get the scalar
equation for $\t(x)$,
\begin{equation}
  \label{theta}
  \dot\t = \w\,\sg - \frac{\dot\sg}{2\sg}\,\s(2\t),
  \qquad \t(0) = 0,
\end{equation}
coupled with a linear homogeneous equation for $r(x)$, namely
\[
  \frac{\dot r}r = \frac{\dot\sg}{2\sg}\,\c(2\t),
  \qquad r(0) = 1,
\]
which immediately yields the quadrature
\begin{equation}
  \label{rint}
  r(x) = \exp\Big\{ \int_0^x \c\big(2\t(y)\big)\;d\log\sqrt\sg(y)\Big\}.
\end{equation}

We note that \eqref{theta} is consistent with our treatment of
piecewise constant entropy; in that case, we have $\dot\t=const$ on
intervals, while each jump contributes a $\d$-function.  Thus the
angle changes linearly on each entropy interval, with a finite jump in
angle (independent of $\w$) at each entropy jump.  In particular, our
framework extends without change to bounded measurable entropy
profiles for which $\log\sg$ has bounded variation.  Moreover, in the
prototypical case of a $\g$-law gas, $\log\sg$ is a multiple of the
entropy.

As above, we are interested in characterizing the base $k$-mode
frequencies and periods, that yield periodic solutions of the
linearized equations.  That is, we want to characterize periods $T_k$,
such that the $k$-mode $\mc T_k(T_k)$ satisfies the boundary
conditions \eqref{bcab} or \eqref{bcper}, respectively, or
equivalently
\[
  \d_{A,k}(T_k)=0 \com{or} \d_{P,k}(T_k)=0,
\]
respectively, these being given by \eqref{dkcts}.  Using \eqref{pruf},
we express the boundary conditions in terms of the angle variable
$\t(x)$, as follows.  For the periodic boundary condition
\eqref{bcper}, we require
\begin{equation}
  \label{kodd}
  \vp_k(\ell) = 0, \com{so} \t(\ell) = \Big(n+\frac12\Big)\pi,
\end{equation}
for $k$ odd, and for $k$ even or the acoustic boundary condition, we
need
\begin{equation}
  \label{keven}
  \psi_k(\ell) = 0, \com{so} \t(\ell) = n\,\pi,
\end{equation}
for some $n$.

We can methodically enumerate all such conditions in terms of the
frequency $\w$ and angle $\t(x) = \t(x,\w)$, using the \emph{angle
  boundary condition}
\begin{equation}
  \label{abc}
  \t(\ell,\w) = k\,\frac\pi2,
\end{equation}
which we interpret as an implicit condition for the coefficient
$\w=\w_k$ of \eqref{theta}.  Integrating, this becomes the condition
\begin{equation}
  \label{wkexp}
  k\,\frac\pi2 = \w_k\,\int_0^\ell\sg\;dy -
  \int_0^\ell \s\big(2\t(y,\w_k)\big)\;d\log\sqrt\sg(y),
\end{equation}
which is an implicit equation for $\w_k$.

Thus, given an entropy profile, we find the base reference period
$T_k$ of the $k$-mode by solving \eqref{wkexp} for $\w_k$, where
$\t(x,\w)$ solves \eqref{theta}, and then using \eqref{Tn}.  Each such
$k$-mode with reference period $T_k$ then determines a periodic
problem of the linearized equation.

Note that \eqref{abc} is equivalent to the periodic boundary condition
\eqref{bcper}, but permits more frequencies than the acoustic
reflection boundary condition \eqref{bcab} allows, these latter two
being equivalent only for even values of $k$.  However, the even modes
form a closed subspace, so if we start from $a_m=0$ for all odd $m$,
this persists when we perturb to the nonlinear problem.  Thus for
notational convenience, we use \eqref{wkexp} to identify all
frequencies, with the understanding that if we are using \eqref{bcab},
all linearizations and perturbations are restricted to even modes.

\begin{lemma}
  \label{lem:wk}
  For each integer $k\ge 1$, and any entropy profile $s(x)$ such that
  $\log\sg\in BV[0,\ell]$, there is a unique $\w_k$ such that the
  solution $\t(x)$ of \eqref{theta} satisfies \eqref{abc}.  Moreover,
  $\{\w_k\}$ is a monotone increasing sequence which grows like $k$,
  that is $\w_k/k\in[1/C,C]$ for some constant $C>0$.
\end{lemma}

\begin{proof}
  We will show that for any fixed $x>0$, the function $\t(x)$ is
  strictly monotone increasing as a function of $\w$.  In particular,
  taking $x=\ell$ in \eqref{abc} implies that $\w_k$ exists for each
  $k\ge1$ and is increasing.  To get the growth rate for $\w_k$, we
  use \eqref{wkexp} to write
  \[
    \bigg|\,\w_k\,\int_0^\ell\sg\;dy - k\,\frac\pi2\bigg|
    \le \int_0^\ell d\,|\!\log\sqrt\sg(y)|.
  \]

  To show monotonicity of $\t(\cdot,\w)$, we note that $\sg$ is
  determined by the entropy $s(x)$ and differentiate the ODE
  \eqref{theta} with respect to $\w$.  Denoting
  $\frac{\del\t}{\del \w}$ by $\zeta$, this yields
  \begin{equation}
    \label{zeta}
    \dot\zeta = \sg - \frac{\dot\sg}{\sg}\,\c(2\t)\,\zeta,
    \qquad \zeta(0) = 0,
  \end{equation}
  which by \eqref{rint} can also be written
  \[
    \dot\zeta = \sg - 2\,\frac{\dot r}{r}\,\zeta,
    \qquad \zeta(0) = 0.
  \]
  It follows that $r^2$ is an integrating factor, and we integrate to
  get
  \[
    r^2(x)\,\zeta(x) = \int_0^xr(y)^2\,\sg(y)\;dy.
  \]
  Thus $\zeta(x)>0$ for $x>0$, and the proof is complete.    
\end{proof}

\subsection{Perturbation and Auxiliary equation}

Assuming now that the $k$-th mode is nonresonant, we show that the
$k$-mode bifurcates to a time periodic solution of the compressible
Euler equations with periodic or acoustic reflective boundary
conditions \eqref{bcper} or \eqref{bcab}, respectively.  As noted
above, it is enough to consider the periodic condition \eqref{bcper}
only, and restrict our attention to even modes when using the
reflective condition \eqref{bcab}.

Our strategy is unchanged: we build Hilbert spaces as in
\eqref{Hsplit} and \eqref{Hplus}, but separating the $k$-mode rather
than the 1-mode.  The analogues of Lemmas~\ref{lem:himodes}
and~\ref{lem:W} then follow in exactly the same way.  For clarity and
brevity, we define the spaces and state the lemmas without proof.

For $k$ fixed, the angle boundary condition \eqref{abc}, or
equivalently \eqref{wkexp}, determines the base frequency $\w_k$, and
by \eqref{Tn} this gives both the SL eigenvalue $\lambda_k$ and
reference period $T=T_k$, namely
\[
  \lambda_k := \w_k^2 \com{and} T_k := k\,\frac{2\pi}{\w_k},
\]
respectively, and while considering other modes, this period $T=T_k$
remains fixed.

Following \eqref{Hsplit}, we define
\begin{equation}
  \label{Hsplit2}
  \begin{aligned}
    \mc H_1 &:= \big\{z + \a\,\c(k\piotk t)\;\big|\;z, \a\in\B R\big\}
    \com{and}\\
    \mc H_2 &:= \Big\{ \sum_{j\ne k}a_j\,\c(j\piotk t)\;\Big|
    \;\sum_{j\ne k} a_j^2\,j^{2s}<\infty\Big\},    
  \end{aligned}
\end{equation}
so the domain is $H^s = \mc H_1\oplus\mc H_2$.  Similarly, from
\eqref{Hplus}, we describe the range $\mc H$ by
\begin{equation}
  \begin{aligned}
  \mc H_+ &:= \Big\{ y = \sum_{j\ne k}a_j\,\s(j\piotk t)\;\Big|
  \;\|y\|<\infty\Big\}, \com{and}\\
  \mc H &:= \big\{ \beta\,\s(k\piotk t) \big\} \oplus \mc H_+, \com{with norm}\\
  \|y\|^2 &:= \beta^2 + \sum_{j>1} a_j^2\,\d_j^{-2}\,j^{2s},
  \end{aligned}
  \label{Hplus2}
\end{equation}
and where we have set
\[
  \d_j := \d_{P,j}(T_k) \com{or} \d_j := \d_{A,j}(T_k),
\]
respectively for the periodic or acoustic boundary condition.  We
similarly define the orthogonal projection,
\[
  \Pi:\mc H\to \mc H_+ \com{by}
  \Pi\Big[\beta\,\s(k\piotk t)+\sum_{j\ne k}a_j\,\s(jt)\Big] :=
  \sum_{j\ne k}a_j\,\s(jt),
\]
which projects onto all but the $k$-mode.

Now let $\mc F$ denote $\mc F_P$ or $\mc F_A$, given by \eqref{Ngenl},
depending on whether periodic or acoustic boundary conditions
\eqref{bcper} or \eqref{bcab} are used, respectively.  Note that the
data consists only of even functions and $\mc F$ incorporates the
projection onto odd modes, and the Hilbert spaces respect this
structure, being spaces of even and odd functions for the domain and
range, respectively.

As above, the nonlinear problem is treated as a
bifurcation problem, which consists of an infinite dimensional
\emph{auxiliary equation}, namely
\begin{equation}
  \label{auxil}
  \Pi\,\mc F\big(y^0\big) = 0, \com{with}
  y^0 = \ol p + z + \a\,\c(k\piotk t) + W,
\end{equation}
together with a scalar \emph{bifurcation equation}, namely
\begin{equation}
  \label{bifurc}
  \Big\langle \s(k\piotk t), \mc F\big(y^0\big)\Big\rangle = 0.
\end{equation}
Here the amplitude $\a$ parameterizes the $k$-mode linear solution
that we are perturbing, $z$ is the free parameter of the 0-mode, which
is always in the kernel, and $W$ is the nonlinear correction off the
kernel.  As usual, we first solve \eqref{auxil} to get a solution
$W(\a,z)$ uniform in $\a$ and $z$, and then we solve for the 0-mode
correction $z$ as a function of $\a$.  Here $z$ is the 0-mode
correction that gives a nonzero derivative in the bifurcation equation
as a consequence of genuine nonlinearity.  Physically, this is a
correction to the ambient pressure that balances rarefaction and
compression, and ultimately this is the physical mechanism for
avoiding shock formation and generating time periodic solutions.

The following summarizes the conclusions of Lemmas \ref{lem:himodes}
and \ref{lem:W}, and follows from the implicit function theorem
exactly as those lemmas do.

\begin{lemma}
  \label{lem:aux}
  If the $k$-mode is nonresonant, there is a neighborhood
  $\mc U\subset\mc H_1$ of the origin and a unique $C^1$ map
  \[
    W:\mc U\to\mc H_2, \com{written}
    W\big(\ol p + z+\a\,\c(k\piotk t)\big) =: W(\a,z) \in\mc H_2,
  \]
  such that, for all $z+\a\,\c(k\piotk t)\in\mc U$, we have 
  a solution of the auxiliary equation \eqref{auxil}, given by
  \[
    \Pi\,\mc F\Big(\ol p + z+\a\,\c(k\piotk t) + W(\a,z)\Big) = 0.
  \]
  Moreover, the map $W(\a,z)$ satisfies the estimate
  \[
    W(\a,z) = o(|\a|),
  \]
  uniformly for $z$ in a neighborhood of 0.
\end{lemma}

\subsection{Bifurcation equation}

It remains to solve the bifurcation equation \eqref{bifurc}, which is
scalar.  Proceeding as before, we define scalar functions $f$ and $g$
in analogy with \eqref{bif}, \eqref{gdef}.  That is, we set
\begin{equation}
  \label{fdef}
  f(\a,z) := \Big\langle \s(\w\, t),
  \mc F\big(\ol p + z+\a\,\c(\w\, t) + W(\a,z)\big)\Big\rangle,
\end{equation}
where $\w := \w_k = k\,2\pi/T_k$, and we set
\[
  \begin{aligned}
    g(\a,z) &:= \frac1\a\, f(\a,z), \quad \a\ne0,\\[2pt]
    g(0,z) &:=  \frac{\del f}{\del\a}(0,z).
  \end{aligned}
\]  
Our goal is to show that there is some $z=z(\a)$ such that
\begin{equation}
  \label{fgbif}
  f\big(\a,z(\a)\big) = 0, \com{or equivalently}
  g\big(\a,z(\a)\big) = 0.
\end{equation}
As above, we cannot apply the implicit function theorem to $f$
directly, because $\frac{\del f}{\del z}\big|_{\a=0}=0$, but we can
apply it to $g$: to do so, we must show that
\begin{equation}
  \label{faz}
  \frac{\del g}{\del z}\Big|_{(0,0)} \ne 0, \com{which is}
  \frac{\del^2f}{\del z\,\del\a}\Big|_{(0,0)} \ne 0.
\end{equation}

We can rewrite $f(\a,z)$ using \eqref{yper}, to get
\[
  f(\a,z) = 
   \Big\langle\s\big(\w\,t),\IR-\,\mc S^{-T/4}\,\mc E^\ell y^0\Big\rangle
  = \Big\langle\s\big(\w\,t-k\,\frac\pi2\big),\mc E^\ell y^0\Big\rangle,
\]
where we have used self-adjointness of $\IR-$ and the
fact that $\s(k\piotk t)$ is odd as a function of $t$, together
with
\[
  \big(\mc S^{-T/4}\big)^\dag\big[\s(\w\,t)\big] =
  \mc S^{T/4}\big[\s(\w\,t)\big] =
  \s\Big(\w\,\big(t - \frac T4\big)\Big)
  = \s\big(\w\,t-k\,\frac\pi2\big),
\]
where $\square^\dag$ denotes the adjoint (in $t$).  It follows that
\begin{equation}
  \label{d2fdaz}
  \frac{\del^2f}{\del z\,\del\a}\Big|_{(0,0)} = 
  \Big\langle \s\big(\w\,t -k\,{\TS\frac\pi2}\big),
  \frac{\del^2}{\del z\,\del\a}\mc E^\ell y^0
  \Big|_{(0,0)}\Big\rangle.
\end{equation}
where
\begin{equation}
  \label{bifdat}
  y^0 := \ol p + z+\a\,\c(\w\,t) + W(\a,z).
\end{equation}
Recall that the acoustic boundary problem \eqref{yab} is obtained by
limiting ourselves to $k$ even.

To proceed we must thus calculate the second derivative of the evolution,
evaluated at the constant state, namely
\begin{equation}
  \label{d2Edadz}
  \frac{\del^2}{\del z\,\del\a}
  \mc E^\ell y^0\Big|_{(0,0)}.
\end{equation}
We postpone this calculation to
the next section so that we can complete the bifurcation argument.
The proof of the following lemma is the topic of Section~\ref{sec:D2E}
below.

\begin{lemma}
  \label{lem:D2Enz}
  For a genuinely nonlinear constitutive equation, the second
  derivative given in \eqref{d2fdaz} is not equal to zero, that is
  \[
    \frac{\del^2f}{\del z\,\del\a}\Big|_{(0,0)} = 
    \Big\langle \s\big(\w\,t -k\,{\TS\frac\pi2}\big),
    \frac{\del^2}{\del z\,\del\a}\mc E^\ell y^0
    \Big|_{(0,0)}\Big\rangle \ne 0.
  \]
\end{lemma}

Applying the lemma to \eqref{d2fdaz} now reduces the solvability of
the bifurcation equation to a single non-degeneracy condition, namely
that the $k$-mode be nonresonant.

\begin{theorem}
  \label{thm:bifurc}
  Assume that the $k$-mode is nonresonant.  Then there exists a
  one-parameter family of solutions of the form \eqref{pert} of
  equation \eqref{yper} or \eqref{yab}, respectively, parameterized by
  the amplitude $\a$ in a neighborhood of $0$.  This in turn generates
  a periodic pure tone solution of the compressible Euler equations.
\end{theorem}

\begin{proof}
  Lemma~\ref{lem:aux} shows that the auxiliary equation has a unique
  solution in a neighborhood of the origin.  It remains only to show that
  the bifurcation equation \eqref{fgbif}, namely
  $f\big(\a,z\big) = 0$, for the $k$-mode can always be solved
  uniquely in a neighborhood of the origin.  Using \eqref{faz},
  Lemma~\ref{lem:D2Enz} and the implicit function theorem imply the
  existence of a unique $z(\a)$ such that \eqref{bifdat} gives the
  data $y^0$ which solves the equation \eqref{fgbif} for $\a$ in a
  neighborhood of the origin.
\end{proof}

\subsection{Nonresonance conditions}

As in our earlier development, we now consider conditions for
resonance or non-resonance of linear modes.  As in
Section~\ref{sec:Tvar}, we regard the resonance conditions as
conditions on the infinite dimensional entropy field, and show that
the set of entropy fields for which any resonance condition holds is
small.

According to Lemma~\ref{lem:nonres}, the $j$-mode is resonant with the
$k$-mode if and only if there is some $p$ such that
\begin{equation}
    \label{resw}
    \w_p = q\,\w_k, \com{with} q := \frac jk.
\end{equation}
The frequencies $\w_k$ and $\w_p$ are in turn chosen by the condition
\eqref{abc} or \eqref{wkexp}, which we rewrite as the implicit
conditions
\[
  \begin{aligned}
  k\,\frac\pi2 &= \w_k\,\int_0^\ell\sg\;dy -
  \int_0^\ell \s\big(2\t(y,\w_k)\big)\;d\log\sqrt\sg(y),\\
  p\,\frac\pi2 &= \w_p\,\int_0^\ell\sg\;dy -
  \int_0^\ell \s\big(2\t(y,\w_p)\big)\;d\log\sqrt\sg(y).
  \end{aligned}
\]
Assuming \eqref{resw} and eliminating the linear term, we get the
condition
\begin{equation}
  \label{wq}
  \int_0^\ell\Big(q\,\s\big(2\t(y,\w_k)\big)
  - \s\big(2\t(y,q\,\w_k)\big)\Big)\;d\log\sqrt\sg(y)
  = (p-j)\,\frac\pi2,
\end{equation}
which is necessary and sufficient for \eqref{resw} to hold.

We view \eqref{wq} as a restriction on the entropy profile $s(x)$, and
as before we set
\[
  \mc Z_{k,j,p} := \Big\{s(x)\;\Big|\;
  \w_p = \frac jk\,\w_k \Big\},
\]
so that $s(x) \in \mc Z_{k,j,p}$ if and only if \eqref{wq} holds.
As before, we then form the union
\[
  \mc Z := \bigcup_{k,p,j}\mc Z_{k,p,j},
\]
which is the set of all entropy profiles having \emph{some} resonant
mode.

In order to show that the resonant set is small, we need to introduce
a convenient topology on the set of entropy profiles.  We thus define
the subset
\begin{equation}
  \label{BVsp}
  \mc B := \Big\{s\in L^1[0,\ell]\;\Big|\;
  \sg \in L^1,\ \log\sg\in BV\Big\},
\end{equation}
together with the $L^1$ topology, where $\sg$ is given by \eqref{UP}.

\begin{lemma}
  \label{lem:Z}
  Each of the sets $\mc Z_{k,j,p}$ is nowhere dense in $\mc B$, and
  the resonant set $\mc Z$ is meagre in $\mc B$.
\end{lemma}

\begin{proof}
  Since the map $\mc B\to\B R$ given by \eqref{wq},
  \[
    s(\cdot) \mapsto 
    \int_0^\ell\Big(q\,\s\big(2\t(y,\w_k)\big)
    - \s\big(2\t(y,q\,\w_k)\big)\Big)\;d\log\sqrt\sg(y)
  \]
  is evidently continuous, each of the sets $\mc Z_{k,j,p}$ is closed
  in $\mc B$.  Let $\epsilon>0$ and $s=s(x)\in\mc Z_{k,j,p}$ be given.
  Since piecewise constant functions are dense, we use
  Theorem~\ref{thm:genpw} to find a piecewise constant entropy profile
  $\ol s(J,\T)$ which approximates $s(x)$ but is fully nonresonant, so
  in particular,
  \[
    \big\|s(x) - \ol s(J,\T)\big\|_{L^1} < \epsilon, \com{with}
    \ol s(J,\T) \notin \mc Z_{k,j,p}.
  \]
  It follows that the interior of the closure of $\mc Z_{k,j,p}$ is
  empty, that is $\mc Z_{k,j,p}$ is nowhere dense, and, being a
  countable union of nowhere dense sets, $\mc Z$ is meagre.
\end{proof}

We cannot use this proof to show directly that $\mc Z$ is nowhere
dense, because $\mc Z$ is not itself closed in $\mc B$.  We expect
that similar to our earlier results, if one were to regard $\mc B$ as
a measure or probability space, then the resonant set $\mc Z$ would
also have zero measure.  However we will not pursue this here.  Our
next theorem is a summary the foregoing lemmas.

\begin{theorem}
  \label{thm:genl}
  Given any entropy profile $s(x)$ with $s\in\mc B$, the linearization
  \eqref{genls} around the constant state solution $(\ol p,0)$
  determines a Sturm-Liouville operator \eqref{SL1}, \eqref{SLop}.
  Imposing the periodic tiling boundary conditions \eqref{yper},
  respectively the acoustic reflection boundary conditions
  \eqref{yab}, determines an increasing sequence $\w_k$, respectively
  $\w_{2k}$, of linearized SL frequencies.  To each of these
  frequencies corresponds a time periodic solution of the linearized
  equations \eqref{UP}, \eqref{Peq}.  For any such frequency which is
  nonresonant, there is a one-parameter family of perturbations of the
  linearized $k$-mode (resp.~$2k$-mode) to a pure tone solution of the
  nonlinear system \eqref{lagr}, parameterized by the amplitude $\a$
  of the $k$-mode component of the linearized data.  The set of fully
  nonresonant profiles, for which every $k$-mode (resp.~$2k$-mode)
  perturbs to a solution of the nonlinear problem, is generic in
  $\mc B$, in that it is residual, or the complement of a countable
  union of nowhere dense sets.
\end{theorem}

\section{Differentiation of the Evolution Operator}
\label{sec:D2E}

To complete the proof of the existence of periodic solutions of the
compressible Euler equations with generic entropy profile, we must
prove Lemma~\ref{lem:D2Enz} stated and used above, which requires
calculation of the second derivative of the evolution operator.

In differentiating the solution twice, we cannot apply
Lemma~\ref{lem:D2E} or Corollary~\ref{cor:D2N} directly, because we no
longer have exact expressions for the nonlinear solutions, although
this can be carried out: see~\cite{YDiffOp}.  However, because our
gradients remain finite and our solutions to the PDE are classical, we
can expand the solution and evaluate the derivative at the origin.

We briefly describe our strategy for evaluating \eqref{d2Edadz}.  Here
we view $\mc E^\ell$ as a Banach space valued function of the two real
variables $\a$ and $z$, evaluated on data given by \eqref{pert}.
Formally we assume that $(p,u)$ are general functions of variables
$(\a,z)$, satisfying \eqref{genls}, and we denote the derivatives with
respect to $\a$ and $z$ as
\[
  \wh{p_\a}:=\frac{\partial p}{\partial \a}, \quad
  \wh{p_z}:=\frac{\partial p}{\partial z}, \com{and}
  \wh{p_{z\a}}:=\frac{\partial^2 p}{\partial z\,\partial \a},
\]
respectively, and similarly for $u$.  Differentiating \eqref{genls} in
$\a$, we get
\[
  \wh{p_\a}_x  +  \wh{u_\a}_t = 0, \qquad
  \wh{u_\a}_x - \big(v_p(p)\,\wh{p_\a}\big)_t = 0,
\]
which are the linearized equations.  Now differentiating in $z$ yields
\begin{equation}
  \label{D2Eeq}
  \wh{p_{z\a}}_x + \wh{u_{z\a}}_t=0, \qquad
  \wh{u_{z\a}}_t - \Big(v_p(p)\,\wh{p_{z\a}} +
  v_{pp}(p)\,\wh{p_\a}\,\wh{p_{z}} \Big)_t = 0,
\end{equation}
which captures the second derivative.  In the context we are working
in, we have $\wh{p_{z}}\big|_{(0,0)} = 1$, and the second derivative
equation is linear inhomogeneous, with varying coefficients
$v_p\big(\ol p,s(x)\big)$ and $v_{pp}\big(\ol p,s(x)\big)\ne 0$ which
encode the leading order effects of the nonlinear interaction of
acoustic waves with the varying entropy field.

\begin{lemma}
  \label{lem:d2ects}
  The second derivative \eqref{d2Edadz} of the evolution operator
  acting on $y^0$ given by \eqref{bifdat} is given by the
  solution of the linear inhomogeneous SL system
  \begin{equation}
    \label{ODE}
    \begin{aligned}
      \dot{\wh\vp} + \w\,\wh\psi &= 0,\\
      \dot{\wh\psi} - \sg^2\,\w\,\wh\vp &= - v_{pp}\,\w\,\vp_k, 
    \end{aligned}
  \end{equation}
  with vanishing initial data $\wh\vp(0) = \wh\psi(0) = 0$,
  with coefficient $\sg^2 = - v_p(\ol p,s)$, and where
  $v_{pp} = v_{pp}(\ol p,s)\ne 0$.  More precisely, we have
  \begin{equation}
    \label{dEdadz}
    \frac{\del^2}{\del z\,\del\a} \mc E^\ell y^0\Big|_{(0,0)} =
    \wh\vp(\ell)\,\c(\w t)+\wh\psi(\ell)\,\s(\w t),
  \end{equation}
  where $\w = \w_k$ is given by \eqref{Tn}.
\end{lemma}

\begin{proof}
  For the data as given by \eqref{auxil}, \eqref{bifdat}, we write the
  corresponding solution of \eqref{genls} as
  \begin{equation}
    \label{pusol}
    \begin{aligned}
      p(x,t) &= \ol p + z + \a\,\vp_k(x)\,\c(\w t) + \wh p(x,t),\\
      u(x,t) &= \a\,\psi_k(x)\,\s(\w t) + \wh u(x,t),
    \end{aligned}
  \end{equation}
  where $\w = k \piotk$, and $\wh p$, $\wh u = o(|\a|)$.  Here we
  have scaled the eigenfunction by $\vp_k(0)=1$, and we have
  \[
    \wh p(0,\cdot) + \wh u(0,\cdot) = W(\a,z).
  \]
  To rederive equation \eqref{D2Eeq}, we make the ansatz
  \eqref{pusol}, expand the nonlinear (classical) solution, and
  differentiate, first in $\a$ and setting $\a=0$, then in $z$ and
  setting $z=0$, to get the derivatives $\wh{p_{\a z}}$ and
  $\wh{u_{\a z}}$.  This is allowed because our solutions are small
  amplitude and locally defined, so do not suffer gradient blowup in
  the region $0\le x\le \ell$.  In this notation, it follows that
  \[
    \frac{\del^2}{\del z\,\del\a}
    \mc E^\ell y^0\Big|_{(0,0)} =
    \wh{p_{\a z}}(\ell,\cdot) + \wh{u_{\a z}}(\ell,\cdot),
  \]
  where $\cdot$ denotes a function of $t$.
  
  We use \eqref{pusol} to expand $v(p,s)$ as
  \[
    v(p,s) = v(\ol p+z,s) + v_p(\ol p+z,s)\,
    \big[\a\,\vp_k(x)\,\c(\w t) + \wh p\big] + O(|\a|^2),
  \]
  and substitute into \eqref{genls} to get
  \[
    \begin{gathered}
      \a\,\dot\vp_k\,\c(\w t) + \wh p_x
      + \a\,\w\,\psi_k\,\c(\w t)+\wh u_t = 0,\\
      \a\,\dot\psi_k\,\s(\w t) + \wh u_x - v_p(\ol p+z,s)\,
      \big(-\a\,\w\,\vp_k\,\s(\w t)+\wh p_t\big) = O(\a^2).
    \end{gathered}
  \]
  Differentiating with respect to $\a$ and setting $\a=0$, we get
  \[
    \begin{gathered}
      \dot\vp_k\,\c(\w t) + \wh{p_\a}_x
      + \w\,\psi_k\,\c(\w t)+\wh{u_\a}_t = 0,\\
      \dot\psi_k\,\s(\w t) + \wh{u_\a}_x - v_p(\ol p+z,s)\,
      \big(-\w\,\vp_k\,\s(\w t) + \wh{p_\a}_t\big) = 0.
    \end{gathered}
  \]
  We now differentiate this in $z$ and set $z=0$, to get
  \begin{equation}
    \label{bifDaz}
    \begin{aligned}
      \wh{p_{\a z}}_x + \wh{u_{\a z}}_t &= 0,\\
      \wh{u_{\a z}}_x + \sg^2\,\wh{p_{\a z}}_t &=
      - v_{pp}(\ol p,s)\,\w\,\vp_k\,\s(\w t),
    \end{aligned}
  \end{equation}
  where we have used \eqref{UP} and the fact that
  $\wh{p_\a}\big|_{(0,0)}=0$.  This is a restatement of the second
  derivative equation \eqref{D2Eeq}.  According to
  Lemma~\ref{lem:aux}, and refering to our ansatz \eqref{bifdat},
  namely
  \[
    y^0 := \ol p + z+\a\,\c(\w\,t) + W(\a,z),
  \]
  we have $W(\a,z) = o(|\a|)$, which yields
  \[
    \frac{\partial W}{\partial \a}\bigg|_{(0,0)} = 0, \com{so also}
    \frac{\partial^2 W}{\partial \a\,\partial z}\bigg|_{(0,0)} = 0.
  \]
  This implies that the initial conditions for \eqref{bifDaz} are
  \[
    \wh{p_{\a z}}(0,t) = 0, \quad \wh{u_{\a z}}(0,t) = 0,
  \]
  and we wish to integrate to $x=\ell$.

  Finally, we make the ansatz
  \[
    \wh{p_{\a z}} = \wh\vp(x)\,\c(\w t), \quad
    \wh{u_{\a z}} = \wh\psi(x)\,\s(\w t), \com{with}
    \w = k\piotk,
  \]
  and separate variables in \eqref{bifDaz} to get the inhomogeneous
  system \eqref{ODE}, as required.
\end{proof}

We solve \eqref{ODE} using Duhamel's principle with the fundamental
solution $\Psi(x;\w)$ given in \eqref{Psi}: it is straight-forward to
check that the solution of \eqref{ODE} is
\begin{equation}
    \label{duhamel}
  \(\wh\vp(x)\\\wh\psi(x)\) = -\w\,\int_0^x
  \Psi(x-x';\w)\(0\\1\)\vp_k(x')\,v_{pp}\;dx',
\end{equation}
and where $v_{pp}=v_{pp}\big(\ol p,s(x')\big)$.  We carry out this
calculation explicitly in Lemma~\ref{lem:duh} below.

This last integrand requires us to find the full fundamental solution
$\Psi(x;\w)$ of the linear system \eqref{SL3}.  To do so, we consider
initial data of the form
\[
  \wt\vp(0) = 0 \com{and} \wt\psi(0) = \wt c_0,
\]
solution of which provides the second column of $\Psi$.  As in
\eqref{pruf}, we again use modified Pr\"ufer coordinates, and so we
make the ansatz
\begin{equation}
  \label{pruf2}
  \wt\vp(x) := - \wt r(x)\,\frac1{\wt\rho(x)}\,\s\big(\wt\t(x)\big), \qquad
  \wt\psi(x) := \wt r(x)\,\wt\rho(x)\,\c\big(\wt\t(x)\big),
\end{equation}
As above, plugging in \eqref{pruf2} into \eqref{SL3} and simplifying
yields
\[
  \begin{aligned}
    - \frac{\dot{\wt r}}{\wt r}\,\s(\wt\t)
    + \frac{\dot{\wt\rho}}{\wt\rho}\,\s(\wt\t)
    - \c(\wt\t)\,\dot{\wt\t}
    + \w\,\wt\rho^2\,\c(\wt\t) &= 0, \\
    \frac{\dot{\wt r}}{\wt r}\,\c(\wt\t)
    + \frac{\dot{\wt\rho}}{\wt\rho}\,\c(\wt\t)
    - \s(\wt\t)\,\dot{\wt\t}
    + \w\,\frac{\sg^2}{\wt\rho^2}\,\s(\wt\t) &= 0,
  \end{aligned}
\]
with initial conditions
\[
  \wt\t(0) = 0 \com{and} \wt r(0) = 1,
  \com{so} \wt c_0 = \frac1{\wt\rho(0)}.
\]
Again choosing $\wt\rho := \sqrt\sg = \rho(x)$, after simplifying we get
\begin{equation}
  \label{rtwt}
  \frac{\dot{\wt r}}{\wt r}
  = - \frac{\dot{\sg}}{2\sg}\,\c(2\wt\t), \qquad
  \dot{\wt\t} = \frac{\dot{\sg}}{2\sg}\,\s(2\wt\t)
  + \w\,\sg,
\end{equation}
and the first of these can again be integrated, giving
\[
  \wt r(x) = \exp\Big\{ -\int_0^x
  \c\big(2\wt\t(y)\big)\;d\log\sqrt\sg(y)\Big\}.
\]  

We summarize the foregoing, which gives a full description of a
fundamental solution of the SL system.

\begin{lemma}
  \label{lem:SLfund}
  A fundamental matrix of the system \eqref{Psi} is
  \[
    \Psi(x;\w) = \(\vp&\wt\vp\\\psi&\wt\psi\) =
    \( r(x)\,\frac1{\rho(x)}\,\c\big(\t(x)\big)
    & -\wt r(x)\,\frac1{\rho(x)}\,\s\big(\wt\t(x)\big)\\[2pt]
    r(x)\,\rho(x)\,\s\big(\t(x)\big)
    & \wt r(x)\,\rho(x)\,\c\big(\wt\t(x)\big) \),
  \]
  with
  \[
    \rho(x) = \sqrt{\sg(x)} \com{and}
    \Psi(0;\w) = \(\frac1{\rho(0)}&0\\[2pt]0&\rho(0)\),
  \]
  where $\t$ and $\wt\t$ solve
  \[
    \dot\t = \w\,\sg - \frac{\dot\sg}{2\sg}\,\s(2\t), \qquad
    \dot{\wt\t} = \w\,\sg + \frac{\dot\sg}{2\sg}\,\s(2\wt\t),
  \]
  with $\t(0) = \wt\t(0) = 0$ respectively, and $r$ and $\wt r$ are
  given explicitly by
  \[
    \begin{aligned}
      r(x) &= \exp\Big\{ \int_0^x
             \c\big(2\t(y)\big)\;d\log\sqrt\sg(y)\Big\}, \\     
      \wt r(x) &= \exp\Big\{ -\int_0^x
                 \c\big(2\wt\t(y)\big)\;d\log\sqrt\sg(y)\Big\},
    \end{aligned}
  \]
  respectively.  Moreover, we have
  \[
    \det\Psi(x;\w) = r(x)\,\wt r(x)\,\c\big(\t(x)-\wt\t(x)\big) = 1
  \]
  for all $x$, and in particular
  \[
    \big|\t(x)-\wt\t(x)\big| <\frac\pi2 \com{for all $x$.}
  \]
\end{lemma}

In this representation of the fundamental solution, the first column
describes the linearized evolution of the even modes $\c(\w\,t)$, and
the second column describes evolution of the odd modes $\s(\w\,t)$.

\begin{proof}
  Combining \eqref{rho}, \eqref{theta} and \eqref{rint} with
  \eqref{rtwt} yields the fundamental solution, and $\Psi(0;\w)$
  follows from our choices $\t(0)=\wt\t(0)=0$ and $r(0)=\wt r(0)=1$.
  Abel's theorem implies the determinant is constant, and the last
  inequality follows by continuity of $\t$ and $\wt\t$.
\end{proof}

Having calculated the fundamental matrix $\Psi$ of \eqref{Psi}, we now
solve the inhomogeneous system \eqref{ODE}.  Rather than use Duhamel's
principle \eqref{duhamel} directly, we rederive it explicitly in the
Pr\"ufer variables using variation of parameters.

\begin{lemma}
  \label{lem:duh}
  The solution of \eqref{ODE} can be written as
  \begin{equation}
    \label{hatpsi}
    \(\wh\vp(x)\\[2pt]\wh\psi(x)\) = \Psi(x;\w)\,\(a(x)\\b(x)\),
  \end{equation}
  for appropriate functions $a$ and $b$, and we have $b(x) < 0$ for
  all $x > 0$.
\end{lemma}

\begin{proof}
  Because $\Psi$ is a fundamental solution, using the ansatz
  \eqref{hatpsi} in \eqref{ODE} yields the simplified system
  \begin{equation}
    \label{abODE}
    \(a\\b\)^{\dot{}} = - v_{pp}\,\w\,\vp\,\Psi^{-1}\,\(0\\1\),
    \qquad a(0) = b(0) = 0.
  \end{equation}
  Next, since $\det\Psi \equiv 1$, we have
  \[
    \Psi^{-1} = \(\wt\psi&-\wt\vp\\-\psi&\vp\),
  \]
  and the second component of \eqref{abODE} simplifies to
  \[
    \dot b = -v_{pp}\,\w\,\vp^2 < 0.
  \]
  It follows that $b(x) < 0$ for all $x > 0$, as required.
\end{proof}

Lemma~\ref{lem:d2ects} gives the second derivative of the evolution in
terms of the fundamental solution, which is calculated in
Lemma~\ref{lem:SLfund}.  In Lemma~\ref{lem:duh}, we explicitly
calculate a nonvanishing term due to genuine nonlinearity, as seen by
the $v_{pp}$ term in \eqref{abODE}.  This in turn allows us to solve
the bifurcation equation as stated in Lemma~\ref{lem:D2Enz} above.

\begin{proof}[Proof of Lemma~\ref{lem:D2Enz}]
  We substitute \eqref{dEdadz} into \eqref{d2fdaz}, to get
  \[
    \Big\langle \s\big(\w\,t -k\,{\TS\frac\pi2}\big),
    \frac{\del^2}{\del z\,\del\a}\mc E^\ell y^0
    \Big|_{(0,0)}\Big\rangle =
    \begin{cases}
      \wh\vp(\ell), & k\text{ odd},\\
      \wh\psi(\ell), & k\text{ even},
    \end{cases}
  \]
  respectively, where $(\wh\vp,\wh\psi)$ solve \eqref{ODE}.
  
  Evaluating \eqref{hatpsi} at $x=\ell$, we get
  \[
    \(\wh\vp(\ell)\\[2pt]\wh\psi(\ell)\) =
    \(\vp(\ell)&\wt\vp(\ell)\\\psi(\ell)&\wt\psi(\ell)\)
    \,\(a(\ell)\\b(\ell)\),
    \com{with} b(\ell)<0.
    \]
  By our choice \eqref{abc} of $\w$, for $k$ odd we have
  \eqref{kodd}, so that $\vp(\ell)=0$, which implies that
  \[
    \wh\vp(\ell) = \wt\vp(\ell)\,b(\ell) \ne 0,
  \]
  and similarly for $k$ even, we have \eqref{keven}, which is
  $\psi(\ell)=0$, so we must have
  \[
    \wh\psi(\ell) = \wt\psi(\ell)\,b(\ell) \ne 0,
  \]
  and the proof is complete.
\end{proof}


\begin{thebibliography}{10}

\bibitem{BR}
Garrett Birkhoff and Gian-Carlo Rota.
\newblock {\em Ordinary differential equations}.
\newblock Wiley, fourth edition, 1989.

\bibitem{Chen1}
Geng Chen.
\newblock Formation of singularity and smooth wave propagation for the
  compressible {E}uler equations.
\newblock {\em J. Hyp. Diff. Eq.}, 8:671--690, 2011.

\bibitem{CPZ1}
Geng Chen, Ronghua Pan, and Shengguo Zhu.
\newblock Singularity formation for compressible {E}uler equations.
\newblock {\em SIAM J. Math. Anal.}, 49:2591--2614, 2017.

\bibitem{CYZ}
Geng Chen, Robin Young, and Qingtian Zhang.
\newblock Shock formation in the compressible {E}uler equations and related
  systems.
\newblock {\em Jour. Hyp. Diff. Eq.}, 10:149--172, 2013.

\bibitem{CodLev}
Earl Coddington and Norman Levinson.
\newblock {\em Theory of ordinary differential equations}.
\newblock McGraw-Hill, 1955.

\bibitem{CF}
R.~Courant and K.O. Friedrichs.
\newblock {\em Supersonic Flow and Shock Waves}.
\newblock Wiley, New York, 1948.

\bibitem{CW}
W.~Craig and G.~Wayne.
\newblock Newton's method and periodic solutions of nonlinear wave equations.
\newblock {\em Comm. Pure Appl. Math.}, 46:1409--1498, 1993.

\bibitem{Deimling}
Klaus Deimling.
\newblock {\em Nonlinear Functional Analysis}.
\newblock Springer, 1985.

\bibitem{G}
J.~Glimm.
\newblock Solutions in the large for nonlinear hyperbolic systems of equations.
\newblock {\em Comm. Pure Appl. Math.}, 18:697--715, 1965.

\bibitem{GL}
J.~Glimm and P.D. Lax.
\newblock Decay of solutions of systems of nonlinear hyperbolic conservation
  laws.
\newblock {\em Memoirs Amer. Math. Soc.}, 101, 1970.

\bibitem{Golush}
M.~Golubitsky and D.G. Schaeffer.
\newblock {\em Singularities and Groups in Bifurcation Theory}.
\newblock Springer-Verlag, 1985.

\bibitem{HamBla}
Mark Hamilton and David Blackstock.
\newblock {\em Nonlinear Acoustics}.
\newblock Academic Press, 1998.

\bibitem{HMR}
J.~Hunter, A.~Majda, and R.~Rosales.
\newblock Resonantly interacting weakly nonlinear hyperbolic waves {II}.\
  {S}everal space variables.
\newblock {\em Stud. Appl. Math.}, 75:187--226, 1986.

\bibitem{J}
Fritz John.
\newblock Formation of singularities in one-dimensional nonlinear wave
  propagation.
\newblock {\em Comm. Pure Appl. Math.}, 27:377--405, 1974.

\bibitem{JohnsonCheret}
J.N. Johnson and R.~Cheret.
\newblock {\em Classic Papers in Shock Compression Science}.
\newblock Springer, 1998.

\bibitem{Jost}
J\"urgen Jost.
\newblock {\em Postmodern Analysis}.
\newblock Springer, 2005.

\bibitem{Lax}
P.D. Lax.
\newblock Hyperbolic systems of conservation laws, {II}.
\newblock {\em Comm. Pure Appl. Math.}, 10:537--566, 1957.

\bibitem{Lax64}
Peter~D. Lax.
\newblock Development of singularities of solutions of nonlinear hyperbolic
  partial differential equations.
\newblock {\em Jour. Math. Physics}, 5:611--613, 1964.

\bibitem{Lindsay}
R.~Bruce Lindsay.
\newblock {\em Acoustics: Historical and Philosophical Development}.
\newblock Dowden, Hutchinson \& Ross, 1973.

\bibitem{Lindsay2}
R.~Bruce Lindsay.
\newblock {\em Physical Acoustics}.
\newblock Dowden, Hutchinson \& Ross, 1974.

\bibitem{TPLblowup}
Tai-Ping Liu.
\newblock Development of singularities in the nonlinear waves for quasi-linear
  hyperbolic partial differential equations.
\newblock {\em J. Diff. Eqns}, 33:92--111, 1979.

\bibitem{Majda}
A.~Majda.
\newblock {\em Compressible Fluid Flow and Systems of Conservation Laws in
  Several Space Variables}.
\newblock Number~53 in Applied Mathematical Sciences. Springer-Verlag, 1984.

\bibitem{MRS}
A.~Majda, R.~Rosales, and M.~Schonbeck.
\newblock A canonical system of integrodifferential equations arising in
  resonant nonlinear acoustics.
\newblock {\em Stud. Appl. Math.}, 79:205--262, 1988.

\bibitem{Moser}
Jurgen Moser.
\newblock A new technique for the construction of solutions of nonlinear
  differential equations.
\newblock {\em Proc. Nat. Acad. Sci.}, 47:1824--1831, 1961.

\bibitem{Muller}
Ingo M\"uller.
\newblock {\em A History of Thermodynamics: The doctrine of Energy and
  Entropy}.
\newblock Springer, 2007.

\bibitem{P}
R.~L. Pego.
\newblock Some explicit resonating waves in weakly nonlinear gas dynamics.
\newblock {\em Studies in Appl. Math.}, 79:263--270, 1988.

\bibitem{Pierce}
Allan Pierce.
\newblock {\em Acoustics: an introduction to its physical principles and
  applications}.
\newblock McGraw Hill, 1981.

\bibitem{Pryce}
John Pryce.
\newblock {\em Numerical solution of Sturm-Liouville problems}.
\newblock Clarendon Press, 1993.

\bibitem{Rabinowitz}
P.H. Rabinowitz.
\newblock Periodic solutions of nonlinear hyperbolic partial differential
  equations.
\newblock {\em Comm. Pure Appl. Math.}, 20:145--205, 1967.

\bibitem{Riemann}
B.~Riemann.
\newblock The propagation of planar air waves of finite amplitude.
\newblock In J.N. Johnson and R.~Cheret, editors, {\em Classic Papers in Shock
  Compression Science}. Springer, 1998.

\bibitem{Shef}
M.~Shefter and R.~Rosales.
\newblock Quasiperiodic solutions in weakly nonlinear gas dynamics.
\newblock {\em Studies in Appl. Math.}, 103:279--337, 1999.

\bibitem{S}
J.~Smoller.
\newblock {\em Shock Waves and Reaction-Diffusion Equations}.
\newblock Springer-Verlag, New York, 1982.

\bibitem{Stokes}
G.G. Stokes.
\newblock On a difficulty in the theory of sound.
\newblock In J.N. Johnson and R.~Cheret, editors, {\em Classic Papers in Shock
  Compression Science}. Springer, 1998.

\bibitem{Taylor}
Michael Taylor.
\newblock {\em Pseudodifferential Operators And Nonlinear PDE}.
\newblock Birkhauser, 1991.

\bibitem{TY}
B.~Temple and R.~Young.
\newblock The large time stability of sound waves.
\newblock {\em Comm. Math. Phys.}, 179:417--466, 1996.

\bibitem{TYperStr}
Blake Temple and Robin Young.
\newblock {A} paradigm for time-periodic sound wave propagation in the
  compressible {E}uler equations.
\newblock {\em Methods and Appls of Analysis}, 16(3):341--364, 2009.

\bibitem{TYperG}
Blake Temple and Robin Young.
\newblock Linear waves that express the simplest possible periodic structure of
  the compressible {E}uler equations.
\newblock {\em Acta Mathematica Scientia}, 29B(6):1749--1766, 2009.

\bibitem{TYperBif}
Blake Temple and Robin Young.
\newblock A {L}iapunov-{S}chmidt reduction for time-periodic solutions of the
  compressible {E}uler equations.
\newblock {\em Methods and Appls of Analysis}, 17(3):225--262, 2010.

\bibitem{TYperLin}
Blake Temple and Robin Young.
\newblock {T}ime-periodic linearized solutions of the compressible {E}uler
  equations and a problem of small divisors.
\newblock {\em SIAM Journal of Math Anal}, 43(1):1--49, 2011.

\bibitem{TYperEV}
Blake Temple and Robin Young.
\newblock A canonical small divisor problem for the {N}ash-{M}oser method.
\newblock {\em Comm. Information and Systems}, 13(4):469--485, 2013.

\bibitem{TYperNM}
Blake Temple and Robin Young.
\newblock A {N}ash-{M}oser framework for finding periodic solutions of the
  compressible {E}uler equations.
\newblock {\em J. Sci. Comput.}, DOI: 10.1007/s10915-014-9851-z, 2014.

\bibitem{TYdiff1}
Blake Temple and Robin Young.
\newblock Inversion of a non-uniform difference operator.
\newblock {\em Methods and Appls of Analysis}, 27(1):65--86, 2020.

\bibitem{TYdiff2}
Blake Temple and Robin Young.
\newblock Inversion of a non-uniform difference operator and a strategy for
  {N}ash-{M}oser.
\newblock {\em Methods and Appls of Analysis}, page 32 pages, 2022.
\newblock to appear.

\bibitem{Tru2}
Cliff Truesdell.
\newblock Rational fluid mechanics, 1687-1765.
\newblock {\em Euleri Opera Omnia, ser. II}, 12, 1954.

\bibitem{Tru1}
Cliff Truesdell.
\newblock {\em Essays in the history of mechanics}.
\newblock Springer-Verlag, 1968.

\bibitem{Vayn}
D.~Vaynblat.
\newblock {\em The Strongly Attracting Character of Large Amplitude Nonlinear
  Resonant Acoustic Waves Without Shocks. A Numerical Study}.
\newblock PhD thesis, M.I.T, 1996.

\bibitem{Yper}
Robin Young.
\newblock Periodic solutions for conservation laws.
\newblock {\em Contemp. Math.}, 255:239--256, 2000.

\bibitem{Ysus}
Robin Young.
\newblock Sustained solutions for conservation laws.
\newblock {\em Comm. PDE}, 26:1--32, 2001.

\bibitem{YDiffOp}
Robin Young.
\newblock Differentiability of nonlinear hyperbolic evolution operators.
\newblock In preparation, 2022.

\end{thebibliography}

\providecommand{\url}[1]{{\tt #1}}

\end{document}